\documentclass[12pt, a4paper]{article}
\newcommand{\oben}[1]{\!\phantom{.}^{#1}}

\usepackage[normalem]{ulem}
\usepackage{tikz}
\usetikzlibrary{decorations}
\usetikzlibrary{decorations.pathmorphing}

\usepackage{amsfonts,amssymb,amsmath,amsthm}
\usepackage[pdftex]{hyperref}
\usepackage[all]{xy}
\usepackage{mathrsfs}
\newtheorem{theorem}{Theorem}[section] 
\newtheorem{prop}[theorem]{Proposition}
\newtheorem{lem}[theorem]{Lemma}
\newtheorem{ddd}[theorem]{Definition}
\newtheorem{kor}[theorem]{Corollary}
\newtheorem{ass}[theorem]{Assumption}
\newtheorem{con}[theorem]{Conjecture}

\newtheorem{fact}[theorem]{Fact}

\theoremstyle{remark}
\newtheorem{rem}[theorem]{Remark}
\newtheorem{ex}[theorem]{Example}
\newtheorem{nota}[theorem]{Notation}

\newcommand{\CAlg}{\mathbf{CAlg}}
\newcommand{\keywords}[1]{}

\newcommand{\Fun}{{\mathbf{Fun}}}

\DeclareMathOperator{\Sm}{\mathbf{Sm}}
\newcommand{\cC}{{\mathcal{C}}}

\newcommand{\Mon}{{\mathbf{ Mon}}}
\newcommand{\bD}{{\mathbf{D}}}
\newcommand{\Nerve}{{\tt N}}
\newcommand{\Map}{{\tt Map}}

\renewcommand{\P}{{\mathbb{P}}}
\newcommand{\rk}{{\tt rk}}

\newcommand{\Imm}{{\tt Im}}

\newcommand{\Ree}{{\tt Re}}

\newcommand{\bS}{{\mathbf{S}}}

\newcommand{\cH}{{\mathcal{H}}}
\newcommand{\cT}{{\mathcal{T}}}

\newcommand{\cl}{{\tt cl}}

\newcommand{\supp}{{\tt supp}}
\newcommand{\rat}{{\tt rat}}

\newcommand{\tr}{{\tt tr}}

\newcommand{\bC}{{\mathbf{C}}}

\newcommand{\bloc}{{\overline{\tt Loc}}}
\newcommand{\End}{{\tt End}}
\newcommand{\cV}{{\mathcal{V}}}

\newcommand{\cycl}{{\tt cycl}}
\newcommand{\Set}{{\mathbf{Set}}}

\DeclareMathOperator{\Tot}{{\tt Tot}}
\newcommand{\Top}{{\mathbf{Top}}}

\newcommand{\cF}{{\mathcal{F}}}

\newcommand{\Aut}{{\tt Aut}}

\newcommand{\ev}{{\tt ev}}

\newcommand{\im}{{\tt im}}
\newcommand{\id}{{\tt id}}
\newcommand{\Ab}{{\mathbf{Ab}}}

\renewcommand{\ker}{{\tt ker}}
\def\hB{\hspace*{\fill}$\Box$ \newline\noindent}
\newcommand{\nat}{{\mathbb{N}}}
\newcommand{\Z}{{\mathbb{Z}}}
\newcommand{\Q}{{\mathbb{Q}}}
\newcommand{\R}{{\mathbb{R}}}
\newcommand{\C}{{\mathbb{C}}}

\newcommand{\Hom}{{\tt Hom}}

\newcommand{\ch}{{\mathbf{ch}}}
\newcommand{\bV}{{\mathbf{V}}}

\newcommand{\bE}{{\mathbf{E}}}

\newcommand{\Tr}{{\tt Tr}}

\newcommand{\pr}{{\tt pr}}
\renewcommand{\lim}{{\tt lim}}
\newcommand{\colim}{{\tt colim}}

\newcommand{\Sp}{\mathbf{Sp}}
\newcommand{\Mf}{\mathbf{Mf}}

\newcommand{\Diff}{{\tt Diff}}
\newcommand{\Ch}{{\mathbf{Ch}}}

\newcommand{\map}{{\tt map}}

\newcommand{\sSet}{{\mathbf{sSet}}}
\newcommand{\Sh}{{\mathbf{Sh}}}

\newcommand{\Mod}{{\mathbf{Mod}}}

\newcommand{\Gal}{{\tt Gal}}
\newcommand{\cCh}{\mathbf{cCh}}
\newcommand{\cP}{{\mathcal{P}}}
\newcommand{\spp}{{\tt sp}}
\newcommand{\bs}{{\mathbf{s}}}
\newcommand{\pt}{{\mathrm{pt}}}
\newcommand{\Spec}{\mathtt{Spec}}

\newcommand{\beil}{{\tt r^{Beil}_\Z}}
\newcommand{\cbeil}{{\tt r_{\C}^{Beil}}}
\newcommand{\cI}{{\mathcal{I}}}

\newcommand{\Rings}{{\mathbf{Rings}}}

\newcommand{\cZ}{{\mathcal{Z}}}
\newcommand{\bbA}{{\mathbb{A}}}

\newcommand{\cL}{{\mathcal{L}}}
\newcommand{\cW}{{\mathcal{W}}}
\newcommand{\PSh}{{\mathbf{PSh}}}

\newcommand{\bA}{{\mathbf{A}}}
\newcommand{\hcycl}{\widehat{\tt cycl}}

\newcommand{\bK}{{\mathbf{K}}}

\newcommand{\Reg}{{\mathbf{Reg} }}
\newcommand{\Alg}{{\mathbf{Alg}}}

\newcommand{\bbG}{{\mathbb{G}}}

\newcommand{\cO}{{\mathcal{O}}}
\newcommand{\cU}{{\mathcal{U}}}

 \newcommand{\Cone}{{\tt Cone}}
 \newcommand{\DR}{{\mathbf{DR}}}
 \newcommand{\Vect}{{\tt Vect}}
 \newcommand{\CommGroup}{{\mathbf{CommGroup}}}
 \newcommand{\CommMon}{{\mathbf{CommMon}}}
 \newcommand{\Cat}{{\mathbf{Cat}}}

\DeclareMathOperator{\Prim}{{\tt Prim}}

\newcommand{\ku}{{\mathbf{ku}}}

\newcommand{\bReg}{\mathbf{Reg}}

\renewcommand{\Re}{\operatorname{Re}}
\renewcommand{\Im}{\operatorname{Im}}

\renewcommand{\cC}{\bC}
\newcommand{\CPn}{\P^{n}_{\C}}

\begin{document}

\date{September 25, 2015}

\title{Regulators and cycle maps in higher-dimensional differential algebraic $K$-theory}
\author{Ulrich Bunke\thanks{Fakult\"at f\"ur Mathematik,
Universit{\"a}t Regensburg,
93040 Regensburg,
GERMANY, ulrich.bunke@mathematik.uni-regensburg.de} 
~and
Georg Tamme\thanks{Fakult\"at f\"ur Mathematik, Universit\"at Regensburg, 93040 Regensburg, GERMANY, georg.tamme@mathematik.uni-regensburg.de}
}


\maketitle
\begin{abstract}
We develop differential algebraic $K$-theory of regular  arithmetic schemes. Our approach is based on a new construction of a functorial,  spectrum level Beilinson regulator using differential forms. We construct a cycle map which   represents differential algebraic $K$-theory classes 
by geometric vector bundles. As an application we derive Lott's relation between short exact sequences of geometric bundles with a higher analytic torsion form. 
\end{abstract}

\tableofcontents

\section{Introduction}

In this paper we develop differential algebraic $K$-theory of smooth complex varieties and their arithmetic analogs.
We construct differential algebraic $K$-theory classes from vector bundles with additional geometric structures. 
Our main and motivating application of this theory concerns the relation (called Lott's relation) between these classes for the constituents of a short exact sequence of such bundles in the case of number rings. 
Technically, this paper is based on a new approach to the  Beilinson regulator via characteristic forms 
for vector bundles with geometry on products of  smooth manifolds and   complex varieties.

 The purpose of the first two subsections of this introduction 
is to introduce the background and the motivation for the paper. In  {Subsection} \ref{klewfwefwefewf} we explain the main ideas of differential cohomology and how they are related with a conjecture of Lott. In the second Subsection \ref{saknaslkdsddqwdqwd} we review the differential algebraic $K$-theory for number rings. We state Lott's relation and the {Transfer Index Conjecture} (TIC) with technical details. The third Subsection \ref{apr1301} gives an overview over the central topic  of the present paper: a new construction of Beilinson's regulator, differential algebraic $K$-theory in the higher-dimensional case, and the cycle map.  {Finally,} Subsection \ref{lkdwqdwqdwdd34234} of this introduction explains the structure of the paper.

\subsection{Differential algebraic $K$-theory and Lott's conjecture}\label{klewfwefwefewf}

We first recall some history and general ideas of differential cohomology. A spectrum $E$ represents a generalized cohomology theory $E^{*}$ on topological spaces. In particular, $E^{*}$ is a collection of   functors $$E^{*}:\Top^{op}\to \Ab$$
from topological spaces to abelian groups indexed by $*\in \Z$. A differential extension of $E$ (in the sense of Hopkins-Singer) depends on the choice of  a chain complex $C$ of real vector spaces and a morphism of spectra  \begin{equation}\label{r23rj23olr23r23r}c:E\to H(C)\ .\end{equation}
Here the spectrum $H(C)$ is the Eilenberg-MacLane spectrum associated to $C$ such that
$\pi_{*}(H(C))\cong H^{-*}(C)$. The triple $(E,C,c)$ is called differential data for $E$.
 A differential extension of $E^{*}$ is a collection of functors 
\[
\hat E^{*}\colon \Mf^{op} \to \Ab
\]
from smooth manifolds to abelian groups, again  indexed by $*\in \Z$. For a manifold $M$  an element of $\hat E^{*}(M)$ combines the topological information of a class in $E^{*}(M)$ and  a cycle in the complex $\Omega^{*}(M;C)$ of differential forms with coefficients in $C$. The main properties {of the differential extension} 
are best expressed by the differential cohomology hexagon \cite{MR2365651}:
\begin{equation}\label{diag:diff}
\xymatrix{&\Omega^{*-1}(M;C)/\im(d)\ar[rr]^{d}\ar[dr]^{a} && \Omega^{*}_{\cl}(M;C)\ar[dr]^{ {\ de\ Rham}}& \\
H^{*-1}(M;C)\ar[dr]\ar[ur]&&\hat E^{*}(M)\ar@{->>}[dr]^{I}\ar[ur]^{R}&&H^{*}(M;C) \\
&F^{*-1}(M)\ar[rr]\ar@{^{(}->}[ur]&&E^{*}(M)\ar[ur]^{c}& }
\end{equation}
This is a diagram of natural homomorphisms,
where $\Omega^{*}_{\cl}(M;C)\subseteq \Omega^{*}(M;C)$ denotes closed $C$-valued forms, and $F^{*}$
is the cohomology theory represented by the spectrum $F$ defined by the fibre sequence  \begin{equation}\label{e23e332e2e3} E\xrightarrow{c} H(C)\to F\to \Sigma E\ .\end{equation}
The transformation $R$ is called the curvature map and extracts the differential form part of a differential cohomology class. The  transformation $I$ maps a differential cohomology class to its underlying cohomology   class.
Finally, the transformation $a$ maps differential forms to differential cohomology classes and encodes the 
secondary information contained in differential cohomology.
The upper and the lower parts of the hexagon are segments of long exact sequences. 
Furthermore, the diagonals at $\hat E^{*}(M)$ are exact.

The historically first example of a differential cohomology theory was the differential extension $\widehat{H\Z}^{*}$  of ordinary singular cohomology with $\Z$-coefficients. It has been constructed  by Cheeger and Simons  \cite{MR827262} in terms of 
  differential characters. In this case the differential data is given by $C:=\R$ and the map
 $c:H\Z\to H\R$ is  {the canonical one} induced by the inclusion $\Z\to \R$. This theory receives refined Chern classes for vector bundles with connection. The differential form part of these classes is given by Chern-Weil forms. In particular, these refined characteristic classes induce   non-trivial secondary invariants for flat bundles.

Versions of differential extensions of topological complex  $K$-theory $KU^{*}$ have been constructed and studied by Karoubi \cite{KaroubiAst, KaroubiTheoGen} (under the name multiplicative $K$-theory), and later by Freed and Hopkins
 \cite{MR1769477},  {by the first author and Schick} in \cite{MR2664467}, by Freed and Lott \cite{Freed-Lott}, and by Simons and Sullivan \cite{MR2521641}.
 In this case, the differential data is given by  $C:=\R[b,b^{-1}]$ with $\deg(b):={-2}$  and  the Chern character
 $c:=\ch:KU\to H(C)$.
   
   Differential extensions of bordism theories  and 
 Landweber exact cohomology theories have been considered in \cite{MR2550094}. 

 In all these examples the differential extension has been constructed in terms of geometric cycles and relations.
A general construction of a differential extension using methods of homotopy theory was given in the ground-breaking paper \cite{MR2192936} by Hopkins and Singer. It follows from the axiomatic characterization \cite{MR2608479}  that the geometric examples above are equivalent to the differential extensions defined by the Hopkins-Singer construction.

The Hopkins-Singer construction  was streamlined and formalized in \cite{bg}, where it was applied to the differential extension of  {the cohomology theory defined by} the algebraic $K$-theory  {of} number rings. We will review this version of the  Hopkins-Singer construction in Subsection \ref{dlqkwdjlqdwqdwqdwqdwqdw24234234}.  {We refer the reader} to \cite{Bunke:2013aa},
which embeds the Hopkins-Singer construction into an even more general context and nicely explains the origin of the
differential cohomology hexagon \eqref{diag:diff}.


Let $M\R/\Z$ denote the Moore spectrum of the group $\R/\Z$. The spectrum $E\R/\Z:=E\wedge M\R/\Z$ represents the cohomology theory $E\R/\Z^{*}$, often called the cohomology of $E$ with coefficients in $\R/\Z$. One of the motivations to define multiplicative K-theory in   \cite{KaroubiAst, KaroubiTheoGen} was to  give a geometric construction of $KU\R/\Z^{*}$. 
This works in greater generality.
We call the differential data $(E,C,c)$ strict, if the morphism $c$  (see \eqref{r23rj23olr23r23r}) induces an isomorphism $\pi_{*}(E)\otimes \R\stackrel{\cong}{\to} H^{-*}(C)$. In this case we have an equivalence
$F\simeq E\R/\Z$ (see \eqref{e23e332e2e3} for $F$) and therefore by \eqref{diag:diff} a natural isomorphism
\begin{equation}\label{eq:flat}
E\R/\Z^{*-1}(M)\cong \hat E^{*}_{flat}(M):=\ker\{R:\hat E^{*}(M)\to \Omega_{\cl}^{*}(M;C)\}\ .
\end{equation} Thus, in the case of strict differential data, the differential cohomology $\hat E^{*}(M)$ contains $E\R/\Z^{*-1}(M)$ as a subgroup of flat classes. Note that in general, without any condition on the differential data, the functor $\hat E^{*}_{flat}$ is homotopy invariant.


To a ring $R$ one can associate an algebraic $K$-theory spectrum $KR$ and an associated cohomology theory $KR^{*}$. We refer to Subsection \ref{saknaslkdsddqwdqwd} for more details. In analogy to the construction of $KU\R/\Z^{-1}$  in terms of multiplicative $K$-theory  by Karoubi,  
Lott defines in \cite{MR1724894}  a candidate $\overline{KR}^{0}_{Lott}$ for the flat part of a differential extension $\widehat{KR}^{0}$  {by cycles and relations}.  
However, at that time, there was no differential extension $\widehat{KR}^{0}$ of $KR^{0}$ available. The main obstacle against the geometric construction of $\widehat{KR}^{0}$ is that, in contrast to complex $K$-theory, there is no simple construction of $KR^{0}(M)$ in terms of bundles of $R$-modules on $M$.
 
 The functor $\overline{KR}^{0}_{Lott}$ is homotopy invariant. Another property which it shares with cohomology theories is a push-forward operation for
 a proper submersion
 $$\pi\colon M \to B$$ of smooth manifolds{: Lott defines 
$$\pi_{!}\colon \overline{KR}^{0}_{Lott}(M) \to \overline{KR}^{0}_{Lott}(B)$$
 by analytic means. On the other hand, for}
 every cohomology theory $E^{*}$ represented by a spectrum $E$, we have the Becker-Gottlieb transfer 
$$\tr:E^{*}(M)\to E^{*}(B)\ .$$ 
It is defined as the pull-back along a canonical equivalence class of stable maps $\Sigma_{+}^{\infty}B\to \Sigma_{+}^{\infty}M$, see \cite{MR0377873}.
 {Lott conjectures} that  {the analytic} push-forward  {$\pi_{!}$} is related to the Becker-Gottlieb transfer  in a precise way.  If $R$ is a ring of integers in a number field, then one can formulate his conjecture as follows:
\begin{con}[Lott]\label{conj:Lott}\mbox{}
\begin{enumerate}
\item There is a natural transformation $\overline{KR}^{0}_{Lott} \to KR\R/\Z^{-1}$.
\item The diagram 
$$
\xymatrix{
\overline{KR}^{0}_{Lott}(M)\ar[r]\ar[d]^{\pi_{!}}&KR\R/\Z^{-1}(M)\ar[d]^{\tr}\\
\overline{KR}^{0}_{Lott}(B)\ar[r]&KR\R/\Z^{-1}(B)
}
$$
 commutes.
\end{enumerate}
%
\end{con}

By now, we can apply  the general Hopkins and Singer construction \cite{MR2192936} to the algebraic $K$-theory spectrum $KR$ of the ring $R$. Following  \cite{bg},
we choose the canonical differential data given by $C:=K_{-*}(R)\otimes \R$ and the canonical  morphism $c:KR\to H(C)$. By definition, this differential data is strict. It gives rise to the differential extension $\widehat{KR}^{0}$ and we 
  have $KR\R/\Z^{-1}(M)\cong \widehat{KR}^{0}_{flat}(M)$ by \eqref{eq:flat}.

 {In}  \cite{bg}, a class in 
 $\widehat{KR}^{0}_{flat}(M)$ was defined
   for every cycle used by Lott in his definition of $\overline{KR}^{0}_{Lott}(M)$. 
This construction would provide the transformation conjectured in the first part of  \ref{conj:Lott}, if the relations used by Lott in the definition of $\overline{KR}^{0}_{Lott}(M)$ also hold  true  in 
 $\widehat{KR}^{0}_{flat}(M)$.  We will formulate Lott's relation in detail in Theorem \ref{jan2104} after the introduction of the necessary notation.
  Lott's relation has been verified partially in  \cite{bg}. 
 It is one of the main objectives of the present paper to verify Lott's relations in general and therefore to provide the transformation asked for in  Conjecture \ref{conj:Lott}, 1.
As explained in \cite{bg}, the second part of Lott's  Conjecture \ref{conj:Lott}  is then an immediate consequence of the (still open) Transfer Index Conjecture (TIC)  {recalled as} \ref{conj:TIC}  {below}.

In contrast to cohomology in the sense of homotopy theory,  differential cohomology is not homotopy invariant. The deviation from homotopy invariance is measured by the homotopy formula, see Subsection \ref{jan1965}. The idea to verify  Lott's relation is to deform
a short exact sequence of bundles of $R$-modules on a manifold $M$ to a split one and to apply the homotopy formula to this deformation. Here by a deformation of a flat bundle we mean a flat bundle on the product $[0,1]\times M$. 
Unfortunately, in this sense flat bundles are rigid. But there exists such a   deformation on $M\times \P_{R}^{1}$
which is not flat, but algebraic along the $\P_{R}^{1}$-direction, where $\P^{1}_{R}$ is the projective line  over the number ring $R$.
 {In order to make use of this deformation, we have to generalize} the theory of \cite{bg} from number rings to higher dimensional schemes, {and we have to establish a new homotopy formula in the algebraic direction (Proposition \ref{jan0111}).}

\subsection{Rings of integers: a review of \cite{bg}} \label{saknaslkdsddqwdqwd}

We consider a unital ring $R$. The zeroth algebraic  {$K$-group} $K_{0}(R)$ of $R$ is    the Grothendieck group of the monoid of isomorphism classes of 
finitely generated projective $R$-modules (see Example \ref{kldqwjnldqwd}). According to Quillen, the higher algebraic  {$K$-groups} of $R$ are defined as the homotopy groups of a pointed space:
$$K_{n}(R):=\pi_{n}(K_{0}(R)\times BGL(R)^{+})\ .$$
This pointed space $K_{0}(R)\times BGL(R)^{+}$ is equivalent to the infinite loop space of the connective algebraic $K$-theory spectrum $KR$ of $R$, see Example \ref{kljdklqwdwqdwqdwqdwqdq}. 
So on the one hand,  we can equivalently express the algebraic  {$K$-groups} of $R$ in terms of the homotopy groups of this spectrum by
 $$K_{n}(R)\cong \pi_{n}(KR)\ .$$ 
 On the other hand, the spectrum $KR$ contains more information as it  represents a generalized  cohomology theory $KR^{*}$, which can be evaluated on
  topological spaces, so in particular on a smooth manifold $M$.

In the present paper we are  especially  interested in the    group $KR^{0}(M)$. Even sticking to these degree-zero  groups, we can encode information about all the higher
  algebraic  {$K$-groups} of $R$.  {Indeed}, for a compact manifold  $M$,  we have the Atiyah-Hirzebruch spectral sequence
  $$E_{2}^{p,q}\cong H^{p}(M,K_{-q}(R))\Rightarrow KR^{p+q}(M)\ .$$ 
So, morally, all the groups $H^{p}(M,K_{p}(R))$ for $p\ge 0$ contribute to $KR^{0}(M)$.

In the following we describe cycles representing classes in $KR^{0}(M)$. 
 A locally constant sheaf of finitely generated projective $R$-modules on   a smooth manifold $M$ will be called an $R$-bundle,  {for short}. 
An $R$-bundle gives rise to a class 
\begin{equation}\label{eq:cycleclass}
\cycl(V)\in KR^{0}(M)\ .
\end{equation}
In order to describe this class, let us first assume that the fibres of $V$ are free of rank $\ell\in \nat$. Then the isomorphism class of $V$ is classified by a homotopy class of maps $v:M\to BGL(\ell,R)$. The class
$\cycl(V)$  {is given} by the homotopy class of the composition
$$\hspace{-0.5cm}M\stackrel{v}{\to} BGL(\ell,R)\to BGL(R)\to BGL(R)^{+}\stackrel{\{\ell\}\times \id}{\to}  \Z\times  BGL(R)^{+}\to K_{0}(R)\times  BGL(R)^{+}\simeq \Omega^{\infty} KR\ .$$
Here all non-named maps are the canonical ones. 
If $V$ is  arbitrary and $M$ is connected, then we can reduce to the case of free fibres
by adding a suitable trivial bundle. If $M$ is  non-connected  we construct the class $\cycl(V)$  component-wise. For more details we refer to \cite{bg}.

For $p\in \Z$, let $H(\R[p])$   denote the 
Eilenberg-MacLane spectrum of $\R$ shifted   in such  a way  that its non-trivial homotopy group is $\R$ in degree $p$.
If $H(\R[p])^{*}$ denotes the cohomology theory represented by this spectrum, then we have by construction
$$H(\R[p])^{0}(M)\cong H^{p}(M;\R)\ .$$

 By definition, a regulator is a map of spectra  $r:KR\to H(\R[p])$. It defines a natural homomorphism  between cohomology groups
$r:KR^{0}(M)\to H^{p}(M;\R)$.  We consider the question how to calculate the class $$r(\cycl(V))\in H^{p}(M;\R)\ .$$

Let us assume that $R$ is  the ring of integers in a number field. We choose an embedding $\sigma:R\to \C$  {and an odd positive integer $p$.}
For these choices Borel \cite{MR0387496} introduced a regulator
\begin{equation}\label{dqwdwqnbdwmqndbmwdwqdwqd87687}
r_{\sigma,p}:KR\to H(\R[p])
\end{equation} 
whose induced map on $p$-th homotopy  {groups} generalizes
 the classical Dirichlet regulator 
$$R^\times \cong K_1(R) \to \R\ , \quad u \mapsto \log |\sigma(u)|,$$
in the case $p=1$.

In the following, we give a differential geometric description of  Borel's regulator.  
The complexification  $V\otimes_{R,\sigma}\C$ of the $R$-bundle $V\to M$
is a locally constant sheaf of finite dimensional complex vector spaces on $M$, which we interpret as the sheaf of parallel sections of
 a complex vector bundle $V_{\sigma}\to M$   with a flat connection $\nabla^{V_{\sigma}}$.
 If we choose a hermitian metric $h^{V_{\sigma}}$ on the complex vector bundle $V_{\sigma}$ 
(not necessarily parallel with respect to $\nabla^{V_{\sigma}}$), then we can form the adjoint connection $\nabla^{V_{\sigma},*}$. 
In general,  given two connections $\nabla_{0},\nabla_{1}$ on a  complex vector bundle, we have the transgression Chern character form
 $\tilde \ch(\nabla_{1},\nabla_{0})$ such that $d\tilde \ch(\nabla_{1},\nabla_{0})=\ch(\nabla_{1})-\ch(\nabla_{0})$.
Since in our case $\nabla^{V_{\sigma}}$ and $\nabla^{V_{\sigma},*}$ are flat, the  degree-$p$ component $\tilde \ch_{p}(\nabla^{V_{\sigma},*},\nabla^{V_{\sigma}})$
 of the transgression  Chern character form is closed. Up to normalization, this is the Kamber-Tondeur form of the flat bundle with metric 
 $(V_{\sigma},\nabla^{V_{\sigma}},h^{V_{\sigma}})$ introduced in \cite{MR1303026}. 
Its cohomology class is independent of the choice of the metric $h^{V_{\sigma}}$. The regulator \eqref{dqwdwqnbdwmqndbmwdwqdwqd87687} is now  characterized by 
 $$r_{\sigma,p}(\cycl(V))=[\tilde \ch_{p}(\nabla^{V_{\sigma},*},\nabla^{V_{\sigma}})]\in H^{p}(M;\R)\ .$$
 
The choice of a conjugation invariant collection $h^{V} :=(h^{V_{\sigma}})_{\sigma}$ of metrics $h^{V_{\sigma}}$ for all embeddings $\sigma:R\to \C$ is called a geometry on $V$. One motivation for introducing the differential extension $\widehat{KR}\oben{0}$ of algebraic $K$-theory of $R$ is that it can   capture a refined cycle class
\begin{equation}\label{jan2108}
\hat \cycl(V,h^{V})\in \widehat{KR}\oben{0}(M)
\end{equation} 
which combines the information about the class
$\cycl(V)\in KR^{0}(M)$  and  the collection of characteristic forms  $(\tilde \ch_{p}(\nabla^{V_{\sigma},*},\nabla^{V_{\sigma}}))_{\sigma}$ with secondary invariants.

We now consider a proper submersion $\pi:M\to B$ between manifolds. On the one hand,  an $R$-bundle $V$ can be pushed forward along $\pi$ by taking fibre-wise cohomology. In sheaf theoretic terms,
one considers the higher derived images $R^{i}\pi_{*}(V)$ as $R$-bundles on $B$. 
On the other hand,
for every cohomology theory $E^{*}$ represented by a spectrum $E$,
we have the Becker-Gottlieb transfer  $\tr:E^{*}(M)\to E^{*}(B)$. Via the cycle map  we can compare the sheaf-theoretic push-forward with the Becker-Gottlieb transfer
in algebraic $K$-theory. 
We  {have} the  following equality in 
 $KR^{0}(B)$:
\begin{equation}\label{jan2101}\sum_{i\ge 0} (-1)^{i} \cycl(R^{i}\pi_{*}(V))=\tr(\cycl(V)) \ .\end{equation}
The  Bismut-Lott index theorem  \cite{MR1303026} implies that \eqref{jan2101} holds true in $H^{p}(B;\R)$ after application of {the} regulator $r_{\sigma,p}$ described above. The equality \eqref{jan2101} in $KR^{0}(B)$ itself is a consequence of the Dwyer-Weiss-Williams index theorem 
\cite{MR1982793}.

The fundamental question considered in \cite{bg}  concerned the refinement of \eqref{jan2101} to differential algebraic $K$-theory.  First of all, in 
\cite{bg} we construct a differential Becker-Gottlieb transfer 
$$\hat \tr:\hat E^{*}(M)\to \hat E^{*}(B)$$
 for every (Hopkins-Singer) differential cohomology theory $\hat E^{*}$. This differential refinement of $\tr$ depends on  the additional  choice of a Riemannian structure $g^{\pi}$ on $\pi$.

If $(V,h^{V})$ is a bundle of $R$-modules with geometry, then we define geometries
$h_{L^{2}}^{R^{i}\pi_{*}(V)}$ on the bundles $R^{i}\pi_{*}V$ using fibre-wise Hodge theory. 
These constructions allow  {us} to lift both sides of \eqref{jan2101} to differential algebraic $K$-theory.
A simple check using the local version of the Bismut-Lott index theorem  \cite{MR1303026}
shows that the naive differential version of \eqref{jan2101} is not true. But the theory of  Bismut-Lott provides a natural candidate for a correction term, which can be expressed in terms of a higher analytic torsion form $\cT(\pi,g^{\pi},V,h^{V})\in \Omega^{-1}(B;C)$. We refer to \cite{bg} for details. In the following formula, $a$ denotes the map from differential forms to differential cohomology as in \eqref{diag:diff}.
\begin{con}[{Transfer Index Conjecture (TIC)}]\label{conj:TIC} 
 The equality
\begin{equation}\label{jan2102}
\sum_{i\ge 0} (-1)^{i} \hat \cycl(R^{i}\pi_{*}(V),h_{L^{2}}^{R^{i}\pi_{*}(V)})+a(\cT(\pi,g^{\pi},V,h^{V}))=\hat \tr(\hat \cycl(V,h^{V})) \end{equation} 
 holds true  in $\widehat{KR}\oben{0}(B)$. 
 \end{con}The TIC is an interdisciplinary statement which combines
homotopy theory, global analysis, and arithmetic. It subsumes  known results like
the local version of the Bismut-Lott index theorem  \cite{MR1303026}, the Dwyer-Weiss-Williams index theorem 
\cite{MR1982793}, or a version of the Cheeger-M\"uller theorem \cite{MR528965}, \cite{MR498252}. Further special cases and arithmetic consequences are discussed in
\cite[Sec.~5]{bg}. 
At the moment, it is still wide open.

 We now describe Lott's relation. We consider a short exact sequence 
\begin{equation}\label{jan1980}\cV:0\to V_0\to V_1\to V_2\to 0\ ,\quad g^{\cV}:=(g^{V_i})_{i=0,1,2}\ ,\end{equation}
of $R$-bundles on $M$  together with a choice of geometries. In $KR^{0}(M)$ we then have the relation
\begin{equation}\label{jan2103}
\cycl(V_{0})-\cycl(V_{1})+\cycl(V_{2})=0\ .
\end{equation}
It is again a natural question whether this equality refines to differential algebraic $K$-theory.
By the theory of Bismut-Lott  \cite{MR1303026}, the naive refinement does not hold  in general, but there is a natural correction term given by a finite dimensional version of the higher analytic torsion form
$\cT(\cV,g^{\cV})\in \Omega^{-1}(M;C)$. In \cite{bg} we  conjectured and partially verified the following result:
\begin{theorem}\label{jan2104}
Lott's relation
$$\hat \cycl(V_{0},g^{V_0})-\hat \cycl(V_{1} ,g^{V_1})+\hat \cycl(V_{2}, ,g^{V_2})=a(\cT(\cV,g^{\cV}))$$
holds true in $\widehat{KR}\oben{0}(M)$.
\end{theorem}
The complete proof of this theorem is one of the main achievements of the present paper (see Theorem \ref{jan0430}). As a consequence, one gets the existence of the natural transformation asserted in Lott's Conjecture \ref{conj:Lott}. 
The Transfer Index Conjecture then implies the second statement of Conjecture \ref{conj:Lott}.

\subsection{The higher-dimensional case} \label{apr1301}

The paper  \cite{bg} is devoted to the study of the differential algebraic $K$-theory of  rings of integers in number fields. In the  {current} paper we generalize this theory 
   to general noetherian, regular, and separated schemes $X$ such that  the base change $X \otimes\Q:=\Spec(\Q)\times_{\Spec(\Z)}X$ is of finite type over $ \Spec (\Q)$.
We call them arithmetic schemes.    In this realm, the case of rings of integers in a number field corresponds to  arithmetic schemes $X=\Spec(R)\to\Spec(\Z)$ which are finite over $\Spec(\Z)$.

To a scheme $X$
 we can associate its connective algebraic $K$-theory spectrum. In the present paper, this spectrum is  denoted by   $\bK_{\Z}(X)$ and    defined  in Definition \ref{jul0870} as the evaluation  of a certain sheaf of spectra $\bK_{\Z}$  {on $X$}. The connection with the notation used in Subsection  \ref{saknaslkdsddqwdqwd} is established by
$KR  \simeq  \bK_{\Z}(\Spec(R))$. The  algebraic  {$K$-groups} of $X$ are defined in terms of the homotopy groups of the  algebraic $K$-theory spectrum of $X$ as 
\begin{equation}\label{ddhqwdhqwkjhdwkqdwqd}
\bK_{n}(X):=\pi_{n}(\bK_{\Z}(X)). \end{equation}

In order to construct a differential extension\footnote{We omit the subscript $\Z$ in order to simplify the notation.}  $\widehat{\bK}(X)^{0}$ of the cohomology theory      determined by the spectrum  $\bK_{\Z}(X)$,  we must fix differential data, i.e.~a   chain complex $C$ and the comparison map $c$  as in \eqref{r23rj23olr23r23r}.
 {They should satisfy the following requirements:}

 {Firstly}, the chain complex $C$ should model the real homotopy type of the spectrum $\bK_{\Z}(X)$ as close as possible. We adopt this requirement in order to assure that the flat part of the differential extension may contain a maximum of interesting secondary information.

 {Secondly}, in order to have a chance to generalize the cycle map for geometric bundles, the groups $\Omega_{\cl}(M;C)$ should be the  {recipient} of explicitly given characteristic forms like  the forms $\tilde \ch_{p}(\nabla^{V_{\sigma},*},\nabla^{V_{\sigma}})$ in the number ring case.

For a number ring $R$   the real homotopy type of $KR$  {has been computed by} Borel, see Example \ref{dkqjwdqwdqwdwqdwqdqdwqd} for the statement.
In contrast, for a  higher dimensional scheme $X$ the real homotopy type of $\bK_{\Z}(X)$ is {unknown} and  {one} subject of Beilinson's conjectures, 
which will be reviewed    in Subsection \ref{nov2602}. According to these conjectures, the absolute Hodge cohomology $$ H^{-*}_{Hodge}(X):=\bigoplus_{p\in \Z} H^{2p-*}_{Hodge}(X,\R(p))$$  of $X$ can serve as a good approximation of $\bK_{*}(X)\otimes \R$.  The relation between $K$-theory and absolute Hodge cohomology is implemented by Beilinson's regulator{, which generalizes Borel's regulator from the case of number rings}.
Originally, this regulator and absolute Hodge cohomology were introduced by Beilinson in \cite{MR760999, MR862628} in order to study special values of $L$-functions of algebraic varieties.  

The absolute Hodge cohomology of a scheme $X$ can be defined as the cohomology  of a complex built from differential forms on the complex manifold $X(\C)$ of $\C$-valued points of $X$. 
This chain complex was
constructed by Burgos and Wang \cite{BurgosLogarithmic, MR1621424}.  
In the present paper we denote it by $\DR_{\Z}(X)$, see Definition \ref{nov2220}.
  The Beilinson regulator will be defined in Definition \ref{jul08161} as a natural map   of spectra 
\begin{equation}\label{qwkjdhqwkjdwqdwqd}
\beil:\bK_{\Z}(X)\to H(\DR_{\Z}(X))\ ,
\end{equation}  
where $H(\DR_{\Z}(X))$ is the Eilenberg-MacLane spectrum of $\DR_{\Z}(X)$. {Now, one} could take the differential data
$(\bK_{\Z}(X), H(\DR_{\Z}(X)),\beil)$ in order to define the differential extension  $\widehat{\bK}(X)^{0}$
using the Hopkins-Singer construction. But we shall observe that these choices do not satisfy our second requirement, namely that
$\Omega_{\cl}(M,H(\DR_{\Z}(X)))$ receives appropriate characteristic forms.
We explain the necessary modifications in the  course of reviewing our new approach to the construction of Beilinson's regulator, which occupies Sections \ref{jan1950}, \ref{jan1960}, and \ref{jan1962}.

 Previous constructions {\cite{MR760999, Schechtman, MR1621424}} of the regulator were mainly designed to produce natural homomorphisms
\begin{equation}\label{dqwdwqdwqdqwdqwd}
\pi_{*}(\bK_{\Z}(X))\to  H^{ -*}_{Hodge}(X)
\end{equation}  
of abelian group valued functors.
 In order to define differential algebraic $K$-theory we must lift this to a morphism between sheaves of spectra \eqref{qwkjdhqwkjdwqdwqd}. 
 Our main contribution in this direction is  a construction of such a lift.
 
 We now explain the main ideas going into this construction. 
If $X$ is a scheme as above and $M$ is a smooth manifold, then by a bundle on $M\times X$ we understand a locally free and finitely generated  $\pr^{*}_{X}\cO_{X}$-module, where $\pr_{X}:M\times X\to X$ is the projection. If $X=\Spec(R)$ for a number ring $R$, then a bundle over
 $M\times X$ is the same thing as an $R$-bundle as discussed in Subsection \ref{saknaslkdsddqwdqwd}.
 We consider a site consisting of products of the form $M\times X$ with the topology given by
 open coverings of $M$ and Zariski open coverings of $X$. On this site we get a stack of bundles with a symmetric monoidal structure given by the direct sum. The algebraic $K$-theory sheaf $\bK_{\Z}$ is defined by group completing
 the nerve of this stack. The details will be given in  Section \ref{feb1002}. In view of this construction, a bundle  $V$ on $M\times X$ naturally yields a class
 $$\cycl(V)\in \pi_{0}(\bK_{\Z}(M\times X))\cong \bK_{\Z}(X)^{0}(M)\ .$$
 This is the higher-dimensional generalization of \eqref{eq:cycleclass}.

 A geometry $g^{V}$ on a bundle $V$ on $M\times X$ involves a hermitian metric and a connection $\nabla$
 on its complexification $V(\C)$ over $M\times X(\C)$. The connection captures the holomorphic structure of the bundle in the direction of $X(\C)$. A geometry gives rise to two  Chern-Weil representatives
 of the Chern character of $V(\C)$.  On the one hand, the  degree-$2p$ part of the Chern form $\ch_{2p}(\nabla)$ belongs to the $p$th step of the Hodge filtration along $X(\C)$. On the other hand, the form  $i^{p}\ch_{2p}(\nabla^{u})$ is real, where $\nabla^{u}$ is the unitarization of $\nabla$ with respect to the chosen metric. We  finally have a transgression
 $\tilde{\ch}_{2p-1}(\nabla,\nabla^{u})$ such that $$d \tilde{\ch}_{2p-1}(\nabla,\nabla^{u})=\ch_{2p}(\nabla)-\ch_{2p}(\nabla^{u})\ .$$  We consider the collection of triples  of forms  \begin{equation}\label{gdhjqwgdhgqwjgdwqdwqdqqwd}
(\ch_{2p}(\nabla^{u}),\ch_{2p}(\nabla),\tilde{\ch}_{2p-1}(\nabla,\nabla^{u}))_{p\in \nat}
\end{equation}  as the characteristic form associated to the geometric bundle
 $(V,g^{V})$.   The complex $\DR_{Mf,\Z}(M\times X)$ introduced in Definition \ref{nov2031} is built from forms on the manifold $M\times X(\C)$
 {in exactly such a way that}
 it captures these characteristic forms as cycles of degree zero. 
Moreover,
 there is a natural inclusion  $\Omega (M,\DR_{\Z}(X))\to \DR_{Mf,\Z}(M\times X)$, which is a quasi-isomorphism. 
 This implies that
 $$H^{0}( \DR_{Mf,\Z}(M\times X))\cong  H(\DR_{\Z}(X))^{0}(M),$$
 which is the correct target of the regulator {$\beil$} applied to elements of   $\bK_{\Z}(X)^{0}(M)$.
 

We {thus} define the differential extension $\widehat{\bK}(X)^{0}$ of the algebraic $K$-theory of $X$ by the Hopkins-Singer construction with the slight modification that we replace the complex $\Omega (M,\DR_{\Z}(X))$ by the quasi-isomorphic  {complex} $\DR_{Mf,\Z}(M\times X)$. It is  {then} not difficult to define the differential algebraic $K$-theory class
\begin{equation}\label{qsqqwdwqdwdj}
\widehat{\cycl}(V,g^{V})\in \widehat{\bK}(X)^{0}(M)
\end{equation} 
of a geometric bundle $ {(V,g^{V})}$,
 see Definition \ref{nov2050neu}.
This is our higher-dimensional generalization of \eqref{jan2108}.
 
In the following, we comment on one important aspect which we have neglected in the discussion above.
In order to represent the absolute Hodge
 cohomology of $X$ {when $X(\C)$ is not compact,} the differential forms entering into the complex $\DR_{Mf,\Z}(M\times X)$ must satisfy growth conditions at infinity  {(see Subsection \ref{apr0501} for details)}.
If  {$X(\C)$ is not compact}, then for an arbitrary geometry $g^{V}$ the forms in  \eqref{gdhjqwgdhgqwjgdwqdwqdqqwd} will
 not  {necessarily} satisfy the required growth conditions.
We therefore introduce  the notion of  good  {geometries (Definition \ref{dez2402}), whose associated characteristic forms \eqref{gdhjqwgdhgqwjgdwqdwqdqqwd} are nice in this respect.}
By the work of  Burgos-Wang \cite{MR1621424},
 good geometries exist locally on $M$. {However, if $X(\C)$ is not compact,}
then in general we do not know whether good geometries exist globally on $M$.
Hence we are forced to work with local geometries. For this reason we must further  replace the absolute Hodge
 complex $\DR_{Mf,\Z}(M\times X)$ by its quasi-isomorphic \v{C}ech complex with respect to coverings of the
 manifold $M$ denoted by $\cL\DR_{Mf,\Z}(M\times X)$.

As a consequence of the approach to differential algebraic $K$-theory $\widehat{\bK}(X)^{0}(M)$ sketched above,
it is clear that it is not only functorial in the variable $M$, but also in $X$. In particular, the cycle map
$\widehat{\cycl}$ in \eqref{qsqqwdwqdwdj} is  compatible with this functoriality. Even if we restrict
to number rings, this is an improvement of \cite{bg}.

An aspect truly related to the higher-dimensional generalization of differential algebraic $K$-theory is the homotopy formula in the algebraic direction discussed in \ref{jan0110}. It is the clue in the proof of Lott's relation.

\begin{rem} 
In a recent paper \cite{HopkinsQuick}, Hopkins and Quick introduce a Hodge filtered version of complex bordism. Its construction is very similar in spirit to our definition of differential algebraic $K$-theory and also uses related constructions with differential forms. However, Hodge filtered complex  bordism is not a differential version of complex   bordism since  from its cohomology classes  one cannot recover the differential form representatives of the relevant cohomology classes. 
\end{rem}

\subsection{Organization of the paper}\label{lkdwqdwqdwdd34234}

In Section \ref{jan1950} we define algebraic $K$-theory and  analyze the construction of regulator maps. 
In Subsection \ref{jan1001} we define the algebraic $K$-theory spectrum of a symmetric monoidal category in terms of an $\infty$-categorical version of commutative group completion  of its nerve. 
 {We} introduce the necessary $\infty$-categorical language in the place of its first appearance.
In the following two Subsections \ref{dwqkjdhqwkdwqdqwdqw89769}  and \ref{mar0802} we introduce the notion of a regulator and explain a procedure to construct  a regulator from a characteristic cocycle. In this way we set up  {an} interface
relating classical  {$1$}-categorical objects (characteristic cocycles) with $\infty$-categorical objects (the regulator).

 In Section \ref{jan1960} we introduce the absolute Hodge complex $\DR_{\Z}(X)$ and specialize the definition of algebraic $K$-theory to the cases we need. In Subsection \ref{sec:AbsHodgeComplex} we first consider a version $\DR_{\C}(X)$ for an algebraic variety $X$  {over $\C$,} where most of the technical issues like the growth conditions at $\infty$ have to be resolved.
 In Section \ref{jul1001} we define $\DR_{\Z}(X)$ for  {an arithmetic} scheme $X$
 in terms of $\DR_{\C}$ applied to the base change of $X$ to $\C$. In Subsection \ref{nov2602}
 we  state a version of Beilinson's conjecture concerning kernel and cokernel of the regulator.

The purpose of  Section \ref{jan1962} is our new construction of the Beilinson regulator. We introduce the site of products
$M\times X$ of manifolds and schemes $X$ in Subsection \ref{kjefwefewfewfewfewfe}. In Subsection \ref{apr0501}
 we extend the definition of the absolute Hodge   complex to  such  products  $M\times X$ and define
$\DR_{Mf,\Z}$. In the next two Subsections \ref{nov1101} and \ref{dez2501} we introduce the notion of 
bundles on $M\times X$  {and of} good geometries, and  {we} construct the associated characteristic forms. In the Subsections \ref{jul1060}
and \ref{sep2601} we apply the general theory developed in Section \ref{jan1950} in order to define the algebraic $K$-theory sheaf $\bK_{Mf,\Z}$  {and}  the Beilinson regulator
$\beil:\bK_{Mf,\Z}\to H(\DR_{Mf,\Z})$ which induces \eqref{qwkjdhqwkjdwqdwqd} by evaluation at $*\times X$.
In Subsection \ref{nov2202} we verify that our spectrum level regulator \eqref{qwkjdhqwkjdwqdwqd} really induces the classical regulator \eqref{dqwdwqdwqdqwdqwd}.
In Subsection  \ref{jan1961} we show that the set-up developed in the present paper  
allows  {us} to rederive easily the generalization of the Karoubi regulator, which was first obtained in  \cite{Tamme-Beil}. 

 The final Section \ref{jan1964} is devoted to differential algebraic $K$-theory and the  verification of Lott's relation. 
 Its first three  {subsections} 
 contain the construction of the differential algebraic $K$-theory spectrum $\widehat{\bK}(X)$,  the calculation of its homotopy groups, and the cycle map $\widehat{\cycl}$, see \eqref{qsqqwdwqdwdj}. In Subsection \ref{jan1965}
 we discuss the homotopy formulas  in the manifold and in the algebraic direction. 
We comment on the relation of differential algebraic $K$-theory to arithmetic $K$-theory in Subsection \ref{sec:Relation-to-arithmetic-K}.
Since Lott's relation and the TIC  {have been} stated  {in \cite{bg}} in the set-up developed  for number rings   there, we provide  
the precise relation between \cite{bg} and the specialization of the  theory of the present paper to number rings
in Subsection \ref{mar18001}.
This will be used in   Subsection \ref{mar18002}, where we complete the proof of Lott's relation.

We tried to make the main text self contained. But the  explanation at the place of first appearance of some of the technical details   would take the reader to far apart from the main course of reading. We therefore have collected  {some elements of $\infty$-categorical sheaf theory  in the Appendix~\ref{apr0502}.}

{\em Acknowledgements:} We thank David Gepner and Thomas Nikolaus for valuable hints. We further thank Jakob Scholbach for interesting discussions. {We are grateful to the referee for his numerous and detailed remarks on the manuscript. They led us to restructure the text and to clarify many points in the exposition. }

\section{Regulators}\label{jan1950}

\subsection{Definition of algebraic $K$-theory}\label{jan1001}

In this subsection we describe the algebraic $K$-theory spectrum of a symmetric monoidal category in terms of nerves and group completion. This will later be applied to the categories of finitely generated projective modules  over a ring or of vector bundles on a regular scheme.

We start with the inclusion  of commutative monoids into commutative groups, which is part of an adjunction
$$K_0:\CommMon(\Set)\leftrightarrows\CommGroup(\Set):incl\ .$$
The left-adjoint of $incl$ is the group completion functor $K_0$, also called the Grothendieck construction.

\begin{ex}\label{kldqwjnldqwd}
For a  unital associative ring $R$, we can consider the  category $ \cP(R)$ of  finitely generated projective $R$-modules  {and its maximal subgroupoid $i\cP(R)$, which has the same objects as $\cP(R)$, but only the isomorphisms as morphisms}. The direct sum induces the structure of a commutative monoid on its
set of isomorphism classes $\pi_{0}( {i}\cP(R))$.  We thus have
$\pi_{0}( {i}\cP(R))\in \CommMon(\Set)$ and $K_{0}(R):=K_{0}(\pi_{0}( {i}\cP(R)))$ is the zeroth $K$-theory group of $R$. In order to define higher algebraic $K$-theory in a similar manner we want to apply the Grothendieck construction not to the set of isomorphism classes of $\cP(R)$, but to the category $\cP(R)$ itself. This is best formulated using the language of $\infty$-categories.
\end{ex}

In the present paper, by an $\infty$-category we mean more precisely an $(\infty,1)$-category, i.e.~we require   all  higher ($\ge 2$) morphisms  to be invertible. 
We will model $\infty$-categories by quasi-categories, i.e.~simplicial sets which satisfy {the} inner horn filling condition.
 {These were first introduced by Boardman-Vogt \cite{Boardman-Vogt}, and then developed systematically by Joyal \cite{Joyal} and Lurie. Our main references are the books \cite{HTT} and \cite{highalg}. For a quick overview, we refer the reader to \cite{Groth}.}

 {Given $\infty$-categories $\bD$ and $\bE$, the simplicial mapping space of maps of simplicial sets from $\bD$ to $\bE$ is again an $\infty$-category. We denote it   by $\Fun(\bD,\bE)$.}

If $\cC$ is an ordinary category, then its nerve $\Nerve(\cC)$ is a simplicial set which happens to be an $\infty$-category.
 Most of the $\infty$-categories used in the present paper are of the form $\Nerve(\cC)[W^{-1}]$ for   
some ordinary category $\cC$ and some class of morphisms $W$ in $\cC$. The localization $ \Nerve(\cC)[W^{-1}]$
and  the morphism $ \iota:\Nerve(\cC)\to  \Nerve(\cC)[W^{-1}]$ are characterized essentially uniquely by    {the} universal property  {\cite[Def.~1.3.4.1]{highalg}}: For every $\infty$-category $\bE$, the composition with $\iota$ induces a fully faithful embedding $\Fun(\Nerve(\bC)[W^{-1}],\bE) \to \Fun(\Nerve(\bC), \bE)$ whose  essential image consists of those functors $\Nerve(\bC) \to \bE$ that map each morphism in $W$ to an equivalence in $\bE$.
Usually, we will not write $\iota$ explicitly. If $X$ is an object of $\cC$, then we will use the notation $X$ also for $\iota(X)$ and make clear in the surrounding text that we consider $X$ as an object in the localization.

As an example, we let $\sSet$ be the category of simplicial sets and $W$ be the class of weak equivalences, i.e.~maps which induce isomorphisms in homotopy groups. Then we get the $\infty$-category $\Nerve(\sSet)[W^{-1}]$ of spaces.

  For a pair of objects $X,Y$ of an ordinary category, we have a  set of morphisms $\Hom(X,Y)$.  
In contrast, for two objects $X,Y$ of an $\infty$-category we have a  mapping space $\map(X,Y)\in \Nerve(\sSet)[W^{-1}]$  {\cite[1.2.2.1]{HTT}}.

The theory of symmetric monoidal $\infty$-categories and  the notion of commutative monoids therein is developed in \cite{highalg}. 
In the present paper 
 we abuse language slightly. We
 call an $\infty$-category $\bC$ symmetric monoidal,
if it is  equivalent to the  fibre $\bC_{\langle 1\rangle}$ of a  {cocartesian} fibration
 $\bC^{\otimes}\to \Nerve(\mathcal{F}in_{*})$ which is the data of a symmetric monoidal $\infty$-category  in the sense of \cite[Definition 2.0.0.7]{highalg}.  The underlying symmetric monoidal structure will always be clear from the context.
We define the $\infty$-category of
commutative algebra  objects $\CAlg(\cC)$ in $\cC$ as in \cite[Definition 2.1.3.1]{highalg}.

The 
  $\infty$-category $\Nerve(\sSet)[W^{-1}]$ is  symmetric monoidal  with its cartesian symmetric monoidal structure.  In this special case,
   we call commutative algebras commutative monoids and use the notation 
$$\CommMon(\Nerve(\sSet)[W^{-1}]):=\CAlg(\Nerve(\sSet)[W^{-1}])\ .$$
Commutative monoids in $\Nerve(\sSet)[W^{-1}]$   are also known as  $E_{\infty}$-spaces.
A grouplike   $E_{\infty}${-}space will be called a commutative group. In order to define this notion, we use the symmetric monoidal functor
 $\pi_{0}:\Nerve(\sSet)[W^{-1}]\to \Set$. 
 We define the $\infty$-category of commutative groups in $\Nerve(\sSet)[W^{-1}]$ to be the full subcategory $$\CommGroup(\Nerve(\sSet)[W^{-1}])\subset \CommMon(\Nerve(\sSet)[W^{-1}])$$   
 of commutative monoids which are mapped to groups under $\pi_{0}$.

 The notion of group completion now generalizes.
In fact, the   inclusion of the $\infty$-category of commutative groups  {into} the $\infty$-category of commutative monoids    is again the right-adjoint of an adjunction 
$$\Omega B:\CommMon(\Nerve(\sSet)[W^{-1}])\leftrightarrows \CommGroup(\Nerve(\sSet)[W^{-1}])\ .$$ We use the symbol $\Omega B$ resembling an explicit model for the group completion since we want to reserve the letter
$K$ to denote the  $K$-theory spectrum functor.

\begin{ex}
A simplicial abelian monoid or group is mapped under the canonical map $\Nerve(\sSet)\to \Nerve(\sSet)[W^{-1}]$  to a monoid  respectively a group in 
$\Nerve(\sSet)[W^{-1}]$. 
An example of a commutative group in $\Nerve(\sSet)[W^{-1}]$ which does not come from a simplicial abelian group is the infinite loop space   $\Omega^{\infty}\mathbb{S}$ of the sphere spectrum $ \mathbb{S}$. 
\end{ex}

We let $\Sp$  denote the $\infty$-category of spectra. We refer to  \cite[Subsection  1.4.3]{highalg} for details.  It is related with the $\infty$-category of pointed simplicial sets $\Nerve(\sSet)_{*}[W^{-1}]$ by an adjunction
$$\Sigma^{\infty}_{*}:\Nerve(\sSet_{*}) [W^{-1}]\leftrightarrows \Sp:\Omega^{\infty}_{*}\ .$$
The  infinite loop space functor $\Omega_{*}^{\infty}$  {refines} to  a functor
\begin{equation}\label{bhjfwefewfewfewkjfjkfkjfjewfwef}
\Omega^{\infty}:  \Sp \to \CommGroup(\Nerve(\sSet)[W^{-1}])\ .
\end{equation}

The $\infty$-category $\Sp$ is the basic example of a stable $\infty$-category  {in the sense of \cite[1.1.1.9]{HTT}}.
 {In general, if $X$ and $Y$ are objects of any stable $\infty$-category, then}
we have a mapping spectrum $\Map(X,Y)\in \Sp$. 
The $\infty$-categorical mapping space can be recovered from the mapping spectrum by the equivalence   $$\map(X,Y)\simeq \Omega^{\infty}_{*}\Map(X,Y)$$ in $\Nerve(\sSet)[W^{-1}]$.

The inclusion of the full subcategory of connective spectra into all spectra is part of an adjunction
$$
 \mathrm{incl}:  {\Sp_{\ge 0}} \leftrightarrows \Sp : {\tau_{\ge 0}}\ .
$$
    
The  restriction of \eqref{bhjfwefewfewfewkjfjkfkjfjewfwef} to connective spectra provides an equivalence of $\infty$-categories
\begin{equation}\label{bhjfwefewfewfewkjfjkfkjfjewfwef111}
\Omega^{\infty}:   {\Sp_{\ge 0}} \xrightarrow{\sim}\CommGroup(\Nerve(\sSet)[W^{-1}])\ .
\end{equation}
The composition of its inverse with the forgetful functor from connective to all spectra will be denoted by 
\begin{equation}\label{jul0701}
\spp:\CommGroup(\Nerve(\sSet)[W^{-1}])\to  \Sp \ .
\end{equation}

One source of commutative monoids in simplicial sets are the nerves of symmetric monoidal categories. 
In general, a symmetric monoidal category can be considered as a commutative monoid in the $2$-category {of categories} $\Cat_{2}$. 
If the associator and symmetry transformations are not given by identities, then it is not a commutative monoid in the  {1}-category  {of categories} $\Cat$. But it gives rise to a commutative monoid in the $\infty$-category  $\Nerve(\Cat)[W^{-1}]$ with its cartesian structure, where $W$ is the class of categorical equivalences.
Following the notation for simplicial sets, we will also use the notation $\CommMon$ instead of $\CAlg$ for commutative monoids in $\Nerve(\Cat)[W^{-1}]$.
The nerve functor $\Nerve:\Nerve(\Cat)[W^{-1}]\to \Nerve(\sSet)[W^{-1}]$ is symmetric monoidal   with respect to the cartesian structures and therefore induces a transformation
$$\Nerve:\CommMon(\Nerve(\Cat)[W^{-1}])\to \CommMon(\Nerve(\sSet)[W^{-1}])\ .$$
We have a functor $i: \Cat \to \Cat$ which maps any category $\cC$ to its maximal  subgroupoid  $i\cC$. The induced functor $i: \Nerve(\Cat)[W^{-1}] \to \Nerve(\Cat)[W^{-1}]$ is also symmetric monoidal.

\begin{ddd}\label{jul1501}
We call the composition
$$K:=\spp\circ \Omega B\circ \Nerve{\circ i}:\CommMon(\Nerve(\Cat)[W^{-1}])\to  \Sp$$
the algebraic $K$-theory functor.
 \end{ddd}

\begin{rem}\label{djqwkjdqwdqwdwqdwqd12312321323}
Note that we use the symbol $K$ in two ways depending on the type of its argument. On the one hand, if  $R$ is a ring, then $KR$ and $K_{*}(R)$ denote the algebraic $K$-theory spectrum and {$K$-groups} of $R$,  {respectively}. On the other hand, if $\cC$ is a symmetric monoidal category, then $K(\cC)$ is the spectrum as in Definition \ref{jul1501}.
\end{rem}

\begin{ex}\label{kdjqlkwdqwdqwdwqdwqdwqdqwd} 
We consider the monoid $\nat$ as a symmetric monoidal category with only identity morphisms. This is actually a commutative monoid in $\Cat$, so in particular in $\Nerve(\Cat)[W^{-1}]$.
We have an equivalence $K(\nat)\simeq H\Z$, where $H\Z$ is the Eilenberg-MacLane spectrum characterized by
\[
\pi_{n}(H\Z)\cong
\begin{cases}
\Z, & \text{if }n=0,\\
0 & \text{else.}
\end{cases}
\]
\end{ex}

\begin{ex}\label{kljdklqwdwqdwqdwqdwqdq}
For  a unital associative  ring $R$, we   consider the  category $ \cP(R)$ of finitely generated projective $R$-modules.
It is symmetric monoidal with respect to the {direct} sum $\oplus$.  It is not  {even} a  monoid in the  {1}-category of categories $\Cat$ since the associator {is not the identity.}
But $ \cP(R)$ can naturally be viewed as a commutative monoid in $\Cat_{2}$ and therefore in $\Nerve(\Cat)[W^{-1}]$. 
 {Hence we can apply the algebraic $K$-theory functor $K$ to $\cP(R)$. It is known that}
we have an equivalence {$K(\cP(R)) \simeq KR$} (see e.g.~\cite[IV.4.8, IV.4.11.1]{WeibelKBook}).    
In the present paper we will   actually take this as a definition of the algebraic $K$-theory spectrum $KR$ of the ring $R$.
\end{ex}

The considerations above generalize to diagrams indexed by a simplicial set $\bS$. In order to simplify the notation,
we retain the symbol of a functor in order to denote its object-wise extension to diagrams.
\begin{ddd}\label{nov1901}
The connective algebraic $K$-theory of a diagram
$$\cV\in \Fun(\bS,\CommMon(\Nerve(\Cat)[W^{-1}]))$$ of symmetric monoidal categories
  is the diagram of spectra
$$K(\cV)\in \Fun(\bS, \Sp  )\ ,\quad  K(\cV):= \spp(\Omega B(\Nerve({i}\cV)))\ .$$
\end{ddd}

\subsection{The space of regulators}\label{dwqkjdhqwkdwqdqwdqw89769}

Let $\bV$ be a symmetric monoidal category. To determine the homotopy type of $K(\bV)$ is in general a difficult task. So we are interested, as a first approximation, in some cohomological information about $K(\bV)$. We consider cohomology theories represented by Eilenberg-MacLane spectra. 

 The $\infty$-category  of spectra $\Sp$ has a closed symmetric monoidal structure based on the  smash product. In particular,  we can define the $\infty$-category of commutative ring spectra $\CAlg(\Sp)$.
 
 We let $\Ch$ be the category of chain complexes and chain maps. We form the $\infty$-category $\Nerve(\Ch)[W^{-1}]$, where $W$ is the class of quasi-isomorphisms. 
 It is a presentable and stable $\infty$-category with 
 a symmetric monoidal structure induced by the tensor product of chain complexes. The tensor  
unit is given by the chain complex $\Z$ concentrated in degree zero. 

Since the $\infty$-category $\Nerve(\Ch)[W^{-1}]$ is stable, for a pair of objects $C,D\in \Nerve(\Ch)[W^{-1}]$ we have a mapping spectrum $\Map(C,D)\in \Sp$. In particular,  we can consider the endomorphism spectrum 
$\Map(\Z,\Z)$ which   has the structure of an associative algebra  with product given by the composition. 
Since $\Nerve(\Ch)[W^{-1}]$ is symmetric monoidal and $\Z$ is the tensor unit, this spectrum $\Map(\Z,\Z)$ is actually a commutative algebra which we will denote by
 $H\Z\in \CAlg(\Sp)$. It gives rise to the $\infty$-category of $H\Z$-modules $\Mod(H\Z)$.
 The Eilenberg-MacLane equivalence
   \begin{equation}\label{ynahdkjwdwqdqwdqwdqwd-neu}
H:\Nerve(\Ch)[W^{-1}]\stackrel{\sim}{\to} \Mod(H\Z)
\end{equation} 
is an equivalence of closed symmetric monoidal $\infty$-categories. Its   existence is asserted in  \cite[ {Thm.~7.1.2.13}]{highalg}.  
In particular, it sends $\Z$ to $H\Z$. See also \cite[Subsection 6.8]{bg}.

%
  
\begin{rem} 
The Eilenberg-MacLane correspondence sends a chain complex $A\in \Nerve(\Ch)[W^{-1}]$ to the $H\Z$-module $\Map(\Z,A)$, where the $H\Z=\Map(\Z,\Z)$-module structure is given by precomposition.
 In particular, we compute
\begin{equation}\label{eq:homotopy-of-EilenbergMacLane}
\pi_{*}(H(A)) \cong \pi_{*}(\Map(\Z,A)) \cong H^{-*}(A),
\end{equation}
as expected. 
\end{rem}


We have an adjunction
$$
\mathrm{Free}:\Sp\leftrightarrows\Mod(H\Z): {U}
$$
where $\mathrm{Free}$ maps a spectrum $E$ to the $H\Z$-module $E\wedge H\Z$, and the forgetful functor  {$U$} takes the underlying spectrum of an $H\Z$-module. We will also use the notation $H$ for the composition
\begin{equation}\label{klrsefpbpw34978h}
\Nerve(\Ch)[W^{-1}]\xrightarrow{H} \Mod(H\Z) \xrightarrow{ {U}} \Sp\ .
\end{equation}
This composition of an equivalence and a right-adjoint preserves limits.
By definition, an Eilenberg-MacLane spectrum is a spectrum which can be presented as $H(A)$ for some $A\in \Nerve(\Ch)[W^{-1}]$.

Given a symmetric monoidal category  $\bV\in \CommMon(\Nerve(\Cat)[W^{-1}])$ and a chain complex $F\in \Nerve(\Ch)[W^{-1}]$, we can now consider the commutative group
$$\Omega^{\infty} \Map(K(\bV),H(F))\in \CommGroup(\Nerve(\sSet)[W^{-1}])  \ ,$$
where $ \Omega^{\infty}$ is as in \eqref{bhjfwefewfewfewkjfjkfkjfjewfwef}.
The homotopy groups of 
$\Omega^{\infty} \Map(K(\bV),H(F))$ are  the cohomology groups of $K(\bV)$ with coefficients in
$F$.

\begin{ex}   
Let $\R[-i]$ be the chain complex whose only non-trivial component is $\R$ in cohomological degree $i$.
Then we have
$$\pi_{0} (\Omega^{\infty} \Map(K(\bV),H(\R[ -i ])))\cong H^{i}(K(\bV);\R)\ .$$
\end{ex}

In order to give a precise {functorial} definition of regulators
we  extend these considerations to diagrams.
We consider a simplicial set $\bS$, a diagram of chain complexes  
\begin{equation}\label{jan1004}
F\in  \Fun(\bS,\Nerve(\Ch)[W^{-1}]),
\end{equation} 
and a diagram
 $$\cV\in \Fun(\bS,\CommMon(\Nerve(\Cat)[W^{-1}]))$$
of symmetric monoidal categories. Using that $\Fun(\bS,\Sp)$ is again a stable $\infty$-category, 
  two diagrams of spectra   $X,Y\in  \Fun(\bS, \Sp)$  give rise to  the mapping spectrum
$\Map(X,Y)\in  \Sp $.
\begin{ddd}
 {For $F$ and $\cV$ as above, the}
commutative group of $F$-valued regulators for $\cV$ is defined by
$$\Reg(\cV,F):=\Omega^{\infty}\Map(K(\cV),H(F))\in  \CommGroup(\Nerve(\sSet)[W^{-1}])\ .$$  
\end{ddd}
Below we also use the term regulator  for elements in $\pi_0
(\Reg(\cV,F))$.

\newcommand{\NumberRings}{\mathbf{NumberRings}}
\newcommand{\Domains}{\mathbf{Domains}}

\begin{ex}  \label{kckljcldscsdcsdcsdcdsc}
Let $\Rings$ be the category of  {commutative, associative,} unital rings.  We can define a morphism of $\infty$-categories
$$\cP:\Nerve(\Rings)\to \CommMon(\Nerve(\Cat)[W^{-1}])$$ which   maps a ring $R$ to the symmetric monoidal category $\cP(R)$ (see Example \ref{kljdklqwdwqdwqdwqdwqdq}), and a morphism of rings
$R\to S$ to the functor $\cP(R)\to \cP(S)$ of extension of scalars given by  $M\mapsto M\otimes_{R}S$.
 Then we can consider  $\bS:=\Nerve(\Rings)$ and $\cV:=\cP$.
 For $F$ we take the constant functor  {$\underline{\Z}$} with value $\Z$. A regulator $\phi\in \pi_{0}(\Reg(\cP, {\underline{\Z}}))$
 gives rise to a natural homomorphism
 $\phi:K_{0}(R)\to \Z$.  Such a natural homomorphism is necessarily trivial, i.e.~$\pi_{0}((\Reg(\cP,{\underline{\Z}})))\simeq *$.\footnote{ {Hint: For any ring $R$, consider the class of $R\times 0$ in $K_{0}(R\times R)\cong K_{0}(R)\oplus K_{0}(R)$}.}

 In contrast, if $ {\Domains}\subseteq \Rings$ is the full subcategory of  {integral domains}, and $\bS_{0}:
 =\Nerve( {\Domains})$, then there is an element $r(\omega_{\rk})\in \pi_{0}(\Reg(\cP_{|\bS_{0}}, {\underline{\Z}}_{|\bS_{0}}))$  which induces the   {rank} homomorphism 
 $\rk:K_{0}(R)\to \Z$. A precise construction of $r(\omega_{\rk})$ will be given below in Example \ref{hdewjdhewkjdhwekdhkewde}.
  \end{ex}

\subsection{Regulators and characteristic cocycles}
\label{mar0802}

Because of the $\infty$-categorical nature of the objects involved, it appears to be a complicated task to produce regulators explicitly.
Fortunately, 
the notion of a characteristic cocycle introduced below in Definition \ref{jan1002} amounts to a formidable
reduction of complexity. The main goal of the present  subsection  is to indicate a general machine to produce regulators from characteristic cocycles.

We consider a {simplicial set $\bS$,} a diagram 
$$\cV\in \Fun(\bS,\CommMon(\Nerve(\Cat)[W^{-1}]))$$ 
of symmetric monoidal categories, and a diagram 
$$F\in \Fun(\bS,\Nerve(\Ch))$$ 
of chain complexes. Note that we do not invert the quasi-isomorphisms here on purpose.
{We derive from $\cV$}
the diagram of {commutative} monoids of isomorphism classes of objects {of $\cV$}
$$\pi_{0}({i}\cV)\in \Fun(\bS,\Nerve({\CommMon}{(\Set)}))$$ 
and from $F$ the
diagram of degree-zero cycles
$$Z^{0}(F)\in \Fun(\bS,\Nerve(\Ab))\ .$$

\begin{ddd}\label{jan1002}
A characteristic cocycle is a  
transformation
$$\omega:\pi_{0}({i}\cV)\to Z^{0}(F)$$ between objects of $\Fun(\bS,\Nerve({\CommMon{(\Set)}}))$.
\end{ddd}

We can consider sets as categories with only identity morphisms. Similarly, commutative monoids are symmetric monoidal categories.  In this sense we have  the symmetric monoidal functor
${i}\cV\to \pi_{0}({i}\cV)$ which associates to every object its isomorphism class.
We consider the composition 
$${i}\cV\to \pi_{0}({i}\cV)\xrightarrow{\omega} Z^{0}(F)$$ of morphisms
{in $\Fun(\bS,\CommMon(\Nerve(\Cat)[W^{-1}]))$}
to which we apply   the algebraic $K$-theory functor {from Definition~\ref{nov1901}.}
 We get a map
\begin{equation}\label{jul0710}
K(\cV)\xrightarrow{K(\omega)} K(Z^{0}(F))\ . 
\end{equation}

\begin{rem}
We need to understand the target of \eqref{jul0710}. We claim that we have a commutative diagram
\begin{equation}\label{jul1190}\begin{split}
\xymatrix{
\Nerve(\Ab)\ar[d]^{(-)[0]} \ar[r]^-{\operatorname{cat}}&\CommMon(\Nerve(\Cat)[W^{-1}])\ar[d]^{K}\\
\Nerve(\Ch)[W^{-1}]\ar[r]^{H}& \Sp 
}
\end{split}
\end{equation}
where the left vertical arrow maps an abelian group $A$ to the chain complex $ {A[0]}$ with the group $A$ placed in degree zero, $\operatorname{cat}$ interprets an abelian group $A$ as a symmetric monoidal category with set of objects $A$ and only identity morphisms,
and $H$ and $K$ are as in  \eqref{klrsefpbpw34978h} and Definition \ref{jul1501},  respectively. In fact, both compositions naturally factor through the full subcategory $\Sp_{0}\subset \Sp$ of spectra $E$ satisfying $\pi_{i}(E) = 0$ for all $i\not=0$.
The functor $\pi_{0}: \Sp_{0}\to \Nerve(\Ab)$
is an equivalence, since  for $E,F\in \Sp_{0}$ the mapping space $\map(E,F)$ has contractible components and
$\pi_{0}(\map(E,F))\cong \Hom(\pi_{0}(E),\pi_{0}(F))$.
The claim now follows, since we have a natural isomorphism of functors $\Nerve(\Ab) \to \Nerve(\Ab)$
\[
\pi_{0}\circ H\circ (-)[0] \cong \id_{\Ab} \cong \pi_{0}\circ K \circ \operatorname{cat}.
\]
%
%
%
\end{rem}

 {The square \eqref{jul1190}} gives an equivalence
\begin{equation}\label{jul0720}
K(Z^{0}(F)) \xrightarrow{\simeq} H( {Z^{0}(F)[0]})\ .
\end{equation}
We have a natural map of diagrams of chain complexes
$$
 {Z^{0}(F)[0]}\to F
$$
which induces the map
 \begin{equation}\label{jul0711}
H({Z^{0}(F)[0]})\to H(F)\ .
\end{equation}

\begin{ddd}\label{jul0810}
We define the regulator $$r(\omega)\in \pi_{0} (\Reg(\cV,F))$$ associated to the characteristic cocycle $\omega:\pi_{0}({i}\cV)\to Z^{0}(F)$
to be the {composition}
$$K(\cV) {\xrightarrow{\eqref{jul0711}\circ \eqref{jul0720}\circ \eqref{jul0710}}} H(F)\ .$$
\end{ddd}

\begin{ex}\label{hdewjdhewkjdhwekdhkewde}
We continue  the Example \ref{kckljcldscsdcsdcsdcdsc}.
We let $\bS_{0}:= {\Nerve(\Domains)}$ and $F:=\underline{\Z}$. 
Then we get the characteristic cocycle
$$\omega_{\rk}:\pi_{0}(i\cP)\to Z^{0}(\underline{\Z})\cong \underline{\Z}$$ which maps
the isomorphism class $[M]\in \pi_{0}(i\cP(R))$ of a finitely generated projective $R$-module $M$  to its
rank $\rk(M)\in \Z$.  The regulator 
$$
r(\omega_{\rk})\in  \pi_{0}(\Reg(\cP_{|\bS_{0}}, {\underline{\Z}}_{|\bS_{0}}))
$$  
is the non-trivial 
regulator mentioned in Example \ref{kckljcldscsdcsdcsdcdsc}.
\end{ex}

\section{Absolute Hodge cohomology and Beilinson's regulator}\label{jan1960}

In the previous section we described a
general machinery to produce regulators.  In the present section  we start  with its application to the Beilinson regulator for  {smooth varieties} over $\C$ or arithmetic schemes. 
In particular, we describe    the relevant domain and target for the regulator map. In view of Beilinson's famous conjectures \cite{MR760999, MR862628}, the appropriate cohomology is the absolute Hodge cohomology of the scheme considered. In order to construct a differential version of algebraic $K$-theory for such schemes we need a complex computing absolute Hodge cohomology built out of forms, with a good functorial behaviour. 
This is accomplished in Subsections \ref{sec:AbsHodgeComplex} and \ref{jul1001} building upon work of Burgos \cite{BurgosLogarithmic}.
In Subsection \ref{feb1002}, we introduce the domain of the regulator, a sheaf of spectra representing algebraic $K$-theory. While the actual construction of the regulator is deferred to Section \ref{jan1962}, we formulate Beilinson's conjecture 
in Subsection \ref{nov2602}. 
The constructions  in Section \ref{jan1962}   are technically  involved because of the presence of an additional manifold direction.  The present section can be considered as an introduction to  Section \ref{jan1962},   specialized to the case where the additional manifold is a point.

\subsection{The  absolute Hodge complex $\DR_{\C}$}
\label{sec:AbsHodgeComplex}

 By a variety over $\C$ we will mean a  {reduced,} separated scheme  such that all  {its} connected components are  of finite type over $\C$.
 We let $\Sm_\C$ denote the site of {smooth} varieties over $\C$ with the topology given by Zariski open  coverings.   

\begin{ex}
The site $\Sm_{\C}$ contains for example:\begin{enumerate}\item 
   the multiplicative group $\mathbb{G}_{m,\C}$, \item the affine space $\bbA_{\C}^{n}$, \item  but also large objects like the disjoint union of complex projective spaces
$\bigsqcup_{n\in \nat} \P_{\C}^{n}$. 
\end{enumerate}
 \end{ex}

 Let $X$ be a smooth variety over $\C$.  According to Nagata \cite{Nagata}   and Hironaka \cite{Hironaka} there exists an open immersion $j\colon X \hookrightarrow \overline X$ into a smooth  variety $\overline X$ over $\C$ such that  each connected component of $\overline X$ is proper over $\C$ and each connected component of  $D:=\overline X - X$ is a divisor with normal crossings. We call $\overline X$ a good compactification of $X$.
 
 \begin{ex} \label{jqwhdkjqwhdwqkdhkwqdwqdwqdwqd}\mbox{}
\begin{enumerate}\item  A good compactification of $\mathbb{G}_{m,\C}$ is given by $\P_{\C}^{1}$ with $D=\{0,\infty\}$.
\item A good compactification of the affine space $\bbA_{\C}^{n}$ is given by $\P_{\C}^{n}$ with $D\cong \P_{\C}^{n-1}$.
\item A good compactification of $\bigsqcup_{n\in \nat} \P_{\C}^{n}$ is $\bigsqcup_{n\in \nat} \P_{\C}^{n}$ itself.
\end{enumerate} 
 \end{ex}

 The set of $\C$-valued points $X(\C)$ has a natural structure of a complex manifold. By abuse of notation we will denote it simply by $X$. It will always be clear from the context whether we consider $X$ as an abstract variety or as a complex manifold. 
Burgos introduces in \cite{BurgosLogarithmic} a differential graded algebra $A_{\overline X}(X,\log D)$ of complex valued smooth differential forms on $X$ with logarithmic singularities along $D$.
 {It} is part of a   mixed $\R$-Hodge complex  in the sense of \cite[8.1.5]{HodgeIII}  (see \eqref{jul1010} below).
We first describe its real subcomplex  
$$A_{\overline X,\R}(X,\log D)\subset A_{\overline X}(X,\log D)$$ with its weight filtration 
\begin{equation}\label{eq:WeightFiltration}
\cW_{*}A_{\overline X,\R}(X,\log D)\ .
\end{equation}
For a smooth manifold $Y$,  let 
\begin{equation}\label{frefewfwffewfewfewfewfe}
A_{\R}(Y)\subset A(Y)
\end{equation} 
denote the differential graded algebras of smooth {real, respectively, complex valued differential forms on $Y$}.
We define $A_{\overline X,\R}(X,\log D)\subset A_{\R}(X)$ to be the subcomplex which is  locally generated as an algebra over $A_{\R}(\overline{X})$ by $1$ and the forms
$$\log(z_{i}\bar z_{i})\ ,\quad \Ree(\frac{dz_{i}}{z_{i}})\ , \quad \Imm(\frac{dz_{i}}{z_{i}})\ ,\quad   i\in I\ .$$ Here 
 the $z_{i}$, $i\in I$ are local coordinates of $\overline X$ in the analytic topology which define the divisor $D$ locally by the equation $\prod_{i\in I}z_{i}=0$.
The weight filtration on $A_{\overline X,\R}(X,\log D)$ is the minimal multiplicative increasing filtration  {by subcomplexes}
such that
$\cW_{0}A_{\overline X,\R}(X,\log D)=A_{\R}(\overline X)$
and the forms $\log(z_{i}\bar z_{i})$,  $\Ree(\frac{dz_{i}}{z_{i}})$, and $ \Imm(\frac{dz_{i}}{z_{i}})$ generate the weight-$1$  {subcomplex}.

We let $A_{\overline X}(X,\log D)\subseteq A(X)$ be the  image of  $ A_{\overline X,\R}(X,\log D)\otimes \C$ under the canonical map
 $ A_{ \R}(X )\otimes \C\to A(X)$.
The complex  $A_{\overline X}(X,\log D)$ has a  decreasing multiplicative Hodge filtration 
\begin{equation}\label{eq:HodgeFiltration}
\cF^{*}A_{\overline X}(X,\log D)\ .
\end{equation}
By definition, a form $\omega\in A_{\overline X}(X,\log D)$ belongs to $\cF^{p}A_{\overline X}(X,\log D)$ if it is locally
a sum of forms of the form  {$d\zeta_{i_{1}}\wedge \dots\wedge d\zeta_{i_{p}}\wedge \omega^{\prime}$, where the $\zeta_{i}$} are complex coordinates of $X$.

 Let $$\iota:A_{\overline X,\R}(X,\log D)\otimes_{\R}\C\stackrel{\cong}{\to}A_{\overline{X}}(X,\log D)$$ be the canonical identification.
 {By \cite[Corollary 2.2]{BurgosLogarithmic}, the} triple  
\begin{equation}\label{jul1010}
\left((A_{\overline X,\R}(X,\log D),\cW_{*}),(A_{\overline{X}}(X,\log D), \cW_*,\cF^*),\iota\right) 
\end{equation} 
is a mixed $\R$-Hodge complex. 
Recall that this implies in particular that the cohomology $H^*\left(A_{\overline X,\R}(X,\log D)\right)$ carries a mixed $\R$-Hodge structure  {in the sense of \cite[2.3.1]{HodgeII}.}
 {Precisely, it is given as follows:}
 By  \cite[Theorem 2.1]{BurgosLogarithmic} together with   \cite[3.1.8]{HodgeII}
    the embeddings
\begin{equation}\label{dez0401}
A_{\bar X,\R}(X,\log D)\to A_{\R}(X), \   \quad \quad \ A_{\bar X}(X,\log D)\to A(X)
\end{equation}
are quasi-isomorphisms.
 {So $H^{n}(A_{\overline X,\R}(X,\log D)) \cong H^{n}(X;\R)$ is the singular cohomology of the complex manifold $X(\C)$ with $\R$-coefficients.}
The weight filtration  {on $H^{n}(X;\R)$} is  induced via  \eqref{dez0401} by the filtration \eqref{eq:WeightFiltration},  shifted by $n$,
\[
\cW_kH^n(X;\R) := \im\left(H^n(\cW_{k-n}A_{\bar X,\R}(X,\log D)) \to H^n(A_{\bar X,\R}(X,\log D))\right)\ ,
\]
and the Hodge filtration {on the complexification $H^n(X;\R) \otimes_\R \C \cong H^n(X;\C)$} is induced   via  \eqref{dez0401} by \eqref{eq:HodgeFiltration}.
By \cite[3.2.5]{HodgeII}, the mixed $\R$-Hodge structure on $H^{*}(X;\R)$ is independent of the choice of the  good compactification.

 \begin{ex}\begin{enumerate}\item
We have  {$H^{0}(\bbA^{n}_{\C};\R)\cong \R(0)$}, where the mixed $\R$-Hodge structure $\R(0)$ on the underlying real vector space $V:=\R$
  is characterized by 
$$ {0=\cW_{1}V \subset\cW_{0} V  = V} \ , \quad  {V_{\C}=\cF^{0}V_{\C}\supset \cF^{1} V_{ \C} =0}\ .$$
 Indeed,  the real vector space {$H^{0}(\bbA^{n}_{\C};\R)$} is generated by the class of the constant function $1\in  {A^{0}_{\P_{\C}^{n},\R}(\bbA_{\C}^{n},\log\P_{\C}^{n-1})}$, which belongs to weight- and Hodge filtration zero.
\item
As mixed $\R$-Hodge structures, we have   {$H^{1}(\bbG_{m,\C};\R) \cong \R(-1)$, where $\R(-1)$ denotes the dual of the Tate $\R$-Hodge structure $\R(1)$ (cf.~\cite[2.1.13]{HodgeII})}. The underlying real vector space of  {$\R(-1)$ is   $V:=\R$}. The weight- and Hodge filtrations are  characterized by 
$$
 {0 = \cW_{1}V \subset \cW_{2}V = V}
$$ 
 and
$$
 {V_{\C} = \cF^{1}V_{\C} \supset \cF^{2}V_{\C}=0.}
$$
Indeed, $H^{1}( {\bbG_{m,\C}};\R)$ is generated by the class of
$$
\frac{dz}{z}+\frac{d\bar z}{\bar z}\in  \cW_{ {1}}A^{1}_{ {\P_{\C}^{1}},\R}( {\bbG_{m,\C}},\log\{0,\infty\})\ ,
$$
and  
$H^{1}( {\bbG_{m,\C}};\C)$ is generated by the class of 
$$
\frac{dz}{z} \in \cF^{1}A^{1}_{ {\P_{\C}^{1}}}( {\bbG_{m,\C}},\log\{0,\infty\})\ .
$$
 \end{enumerate}
\end{ex}
 
Recall that the d\'ecalage of the weight filtration \eqref{eq:WeightFiltration} is given by 
\begin{equation}\label{sep0407}
\hat \cW_{k}A_{\overline X,\R}^{n}(X,\log D) := \{ \omega\in \cW_{k-n}A^{n}_{\overline X, \R}(X,\log D) \,|\, d\omega\in \cW_{k-n-1}A^{n+1}_{\overline X, \R}(X,\log D)\}\ .
\end{equation}
It  induces the weight filtration on $H^*(X,\R)$  without the shift by the degree (see \cite[1.3.4]{HodgeII}), and by \cite[3.2.10]{HodgeII} the associated spectral sequence degenerates at $E_1$. In particular, we have
\begin{equation}\label{eq:WeightFiltOnCohom}
\cW_kH^*(X;\R) \cong  H^*(\hat \cW_{k}A_{\overline X,\R}(X,\log D))\ .
\end{equation}
 The same reasoning applies to the cohomology of any mixed $\R$-Hodge complex (see \cite[8.1.9]{HodgeIII}).

In the following, we get rid of the choice of the good compactification and define a {pre}sheaf of mixed $\R$-Hodge complexes on the site $\Sm_\C$. 
For a  smooth variety $X$, the category $I_{X}$ of good  compactifications of $X$ with respect to maps under $X$ {is cofiltered and essentially small}.
For a morphism
$$\xymatrix{&X\ar[dr]\ar[dl]&\\\overline X\ar[rr]&& \overline X^{\prime}} $$
in $I_{X}$
we have an inclusion 
\begin{equation}\label{dez1101}
A_{\overline X^{\prime}}(X,\log (\overline X^{\prime}\setminus X))\subseteq A_{\overline X}(X,\log (\overline X\setminus X))\end{equation}
which is compatible with the real subcomplex and the weight and Hodge filtrations.
 By \cite[Thm. 3.2.5]{HodgeII} and \eqref{eq:WeightFiltOnCohom} it is in fact a bifiltered quasi-isomorphism with respect to $\hat \cW$ and $\cF$. Hence we get functors
\begin{equation}\label{dez0312}
 I^{op}_{X}\to \Ch\ ,\quad (X \to \overline X)\mapsto A_{\overline X}(X,\log (\overline X\setminus X)),
\end{equation}
and subfunctors corresponding to the real subcomplex and the weight and Hodge filtrations.

 {For any 1-category $\bC$, we denote the 1-category of presheaves on $\Sm_{\C}$ with values in $\bC$ by $\PSh_{\bC}(\Sm_{\C})$. There is an equivalence of $\infty$-categories 
\[
\Nerve(\PSh_{\bC}(\Sm_{\C})) \simeq \Fun(\Nerve(\Sm_{\C}^{op}), \Nerve(\bC)).
\]
We will often use similar identifications implicitly when we consider the category of functors between 1-categories as $\infty$-category.
}
\begin{ddd}
\label{dez1401}
We define the presheaf of chain complexes
$$A_{\log}\in \PSh_{\Ch}(\Sm_{\C})$$ by  
$$X \mapsto A_{\log}(X) := \colim_{I_{X}^{{op}}} A_{\overline X}(X,\log (\overline X \setminus X))\ .$$ 
Furthermore, we define
$$ A_{\log,\R},\  \hat \cW_{*}A_{\log,\R},\ \hat \cW_{*} A_{\log},\ \cF^{*}A_{\log}\  \in \PSh_{\Ch}(\Sm_{\C})$$ in a similar manner.
\end{ddd}
{Using the exactness of a filtered colimit it is clear that for any $X\in \Sm_\C$ the triple
\[
\left((A_{\log,\R}(X), \cW_*), (A_{\log}(X), \cW_*, \cF^*), \iota\right)
\]
is a mixed $\R$-Hodge complex computing the mixed $\R$-Hodge structure on $H^*(X;\R)$ described above which moreover depends in a functorial way on $X$.}

Let $S$ be a site, and let $\bS$ be the $\infty$-category $\Nerve(S^{op})$.
Let $\cC$ be a presentable $\infty$-category  {\cite[Ch.~5]{HTT}}. 
Using the covering families of the site  {$S$}, we  define the full subcategory of  sheaves
\begin{equation}\label{eq:sheaves}
 \Fun^{desc}(\bS,\cC)\subseteq  \Fun(\bS,\cC)
\end{equation}
 which are the objects satisfying descent.
See  Subsection~\ref{dez2701} for more details.
 
\begin{ex}\label{ex:sheaves}
Consider the site $\Mf$ of smooth manifolds with the open covering topology, and write $\bS_{Mf}:= \Nerve(\Mf^{op})$.
\begin{enumerate}
\item Let $\underline{\R} \in \Sh_{\Ch}(\Mf)$ be the constant sheaf on $\Mf$ with value $\R$, considered as a complex concentrated in degree 0. 
Then $\underline{\R}$, considered as an object of the $\infty$-category
\(
\Fun(\bS_{Mf}, \Nerve(\Ch)) \simeq \Nerve(\PSh_{\Ch}(\Mf))
\)
is a sheaf.
However, its image $\iota{(\underline{\R})}$ in $\Fun(\bS_{Mf}, \Nerve(\Ch)[W^{-1}])$ is not a sheaf. For example, 
if $U_{\bullet}$ is the \v{C}ech nerve (in the sense of Remark~\ref{rem:sheaves}) of the covering of the 1-sphere $S^{1}\subset \C$ by $S^{1}\setminus\{1\}, S^{1}\setminus\{-1\}$, then 
 $\lim_{\Nerve(\Delta)}\iota(\underline{\R})(U_{\bullet})$ has cohomology $\R$ in degrees 0 and 1, whereas $\iota(\underline{\R})(S^{1})$ has cohomology only in degree 0.
\item The complex of smooth forms $A$ is a sheaf when considered in $\Fun(\bS_{Mf}, \Nerve(\Ch))$  since it consists degree-wise of sheaves in the classical sense. By Lemma \ref{dqwhdlqdwqdqwdwqd}, 2.~it is also a sheaf when considered in $\Fun(\bS_{Mf}, \Nerve(\Ch)[W^{-1}])$.
\end{enumerate}
\end{ex}

 {
We return to the framework of smooth varieties over $\C$.
We consider the $\infty$-category
\[
\bS_{\C} := \Nerve(\Sm_{\C}^{op}).
\]
}
The presheaves of Definition \ref{dez1401} can be seen as objects of the $\infty$-category 
$$\Fun(\bS_{\C}, \Nerve(\Ch))\simeq \Nerve(\PSh_{\Ch}(\Sm_{\C}))$$
or also as objects of $\Fun(\bS_\C, \Nerve(\Ch)[W^{-1}])$.   We claim that they are in fact sheaves when considered as objects in the latter.

\begin{lem}\label{lem:DescentHodgeCplx}
We have 
$$A_{\log}\ ,A_{\log,\R}\ , \hat \cW_{{k}}A_{\log,\R}\ , \cF^{p}\cap \hat \cW_{k} A_{\log}  \in\Fun^{desc}(\bS_{\C}, \Nerve(\Ch)[W^{-1}])\ .$$
\end{lem}
\begin{proof}
The functors $A$ and $A_{\R}$ satisfy descent in the analytic, and hence in the Zariski topology.
Because of the  quasi-isomorphisms \eqref{dez0401}, the inclusions
\begin{equation}\label{dez0402}A_{\log}(X)\to A(X)\ ,\quad A_{\log,\R}(X)\to A_{\R}(X)\end{equation}
are quasi-isomorphisms for all $X\in \Sm_{\C}$.
 Hence the subfunctors
$A_{\log}$ and  $A_{\log,\R}$ satisfy Zariski descent, too.

Next we handle $\hat \cW_k$. Obviously, condition (A) of Remark~\ref{rem:sheaves} is satisfied. We now check (B). Let  \begin{equation}\label{apr1601}
\Tot: \cCh \to \Ch
\end{equation}
be the total complex functor from cosimplicial chain complexes to chain complexes{, given on a cosimplicial complex $[q]\mapsto F^{*}[q]$ by}
 {$\Tot(F)^{n}:=\prod_{p+q=n}F^{p}[q]$}. For $F \in  \cCh$,  we can calculate the limit of the
corresponding object $\iota(F)\in \Fun(\Nerve(\Delta),\Ch[W^{-1}])$ by
$$\lim_{\Nerve(\Delta)} \iota(F)\simeq \Tot(F)\ $$
 {(see \cite[Problem~4.3.2]{skript} for an argument).}

Let $U_{\bullet}$ be the \v{C}ech nerve of a Zariski cover of $X$. 
We have to show that the map
\begin{equation}\label{may1102}
\hat \cW_{k} A_{\log,\R}(X)\to \Tot\left(\hat \cW_{k}A_{\log, \R}(U_{\bullet})\right)
\end{equation} 
is a quasi-isomorphism.
Now 
$$
\left((A_{\log,\R}(U_{\bullet}), \cW_*), (A_{\log}(U_{\bullet}), \cW_*, \cF^*), \iota\right)
$$
is a cosimplicial mixed $\R$-Hodge complex. Deligne has shown  {in} \cite[8.1.15]{HodgeIII}\footnote{ {Note that our complexes are bounded below. Hence it makes no difference, whether we use the direct product total complex $\Tot$ or the direct sum total complex as Deligne does.}} that the associated total complex $ {\Tot} A_{\log,\R}(U_{\bullet})$ inherits the structure of a mixed $\R$-Hodge complex whose weight and Hodge filtrations $\cW_*$ and $\cF^*$ are given by
\[
\cW_{k}(\Tot A_{\log,\R}(U_{\bullet})):= \bigoplus_p \cW_{k+p}A_{\log,\R}(U_{p}) \quad \text{and}\quad
\cF^p(\Tot A_{\log}(U_{\bullet})) := \Tot (\cF^pA_{\log}(U_{\bullet})),
\]
respectively. Hence the cohomology $H^*( {\Tot} A_{\log,\R}(U_{\bullet}))$ carries a mixed $\R$-Hodge structure, and by \eqref{eq:WeightFiltOnCohom} and the remark following it, the horizontal maps in the diagram
\begin{equation}\label{may1101}\begin{split}
\xymatrix{
H^*\left(\hat\cW_k\left(\Tot A_{\log,\R}(U_{\bullet})\right)\right) \ar[r]^{\cong} & \cW_kH^*\left(\Tot A_{\log,\R}(U_{\bullet})\right)\\
H^*\left(\hat\cW_k A_{\log,\R}(X)\right) \ar[r]^{\cong} \ar[u] & \cW_kH^*\left(A_{\log,\R}(X)\right)\ar[u]
}
\end{split}
\end{equation}
are isomorphisms.

It is easy to see that the d\'ecalage of the weight filtration satisfies
\[
\hat \cW_{k} \left(\Tot A_{\log,\R}(U_{\bullet})\right)= \Tot\left(\hat \cW_{k}A_{\log,\R}(U_{\bullet})\right)\ .
\]
Hence, in order to prove that \eqref{may1102} is a quasi-isomorphism, it suffices to check that the right vertical map in \eqref{may1101} is an isomorphism. But $H^*(A_{\log,\R}(X)) \to H^*(\Tot A_{\log,\R}(U_{\bullet}))$ is a morphism of mixed $\R$-Hodge structures, hence strict with respect to both filtrations \cite[1.2.10]{HodgeII}, and an isomorphism by Zariski descent for $A_{\log,\R}$.

The main non-formal input for this argument {was \eqref{eq:WeightFiltOnCohom}}, 
which {is a consequence} of the fact that the  spectral sequence associated to ${\hat\cW_{ * }}$ degenerates at $E_{1}$. 
Since {this holds}
true  for $\cF^{ * }A_{\log}$ and  the induced filtration  $\hat \cW_k\cap \cF^{ {*}}A_{\log}$, we can argue similarly  for the remaining cases. \end{proof}

\begin{nota}\label{nota:cone} 
Let us introduce some notation concerning {complexes and} cones.
Using the convention $C_{-i} = C^{i}$, a chain complex $(C,\partial_C)$ will either be indexed cohomologically, 
$$\dots \to C^i\xrightarrow{\partial_C} C^{i+1} \to \dots\ ,$$
 or homologically, 
 $$\dots \to C_{-i} \xrightarrow{\partial_C}   C_{-i-1} \to \dots \ .$$
 
For an integer $k$, we define the shifted complex $(C[k], \partial_{C[k]})$ by $C[k]^i:=C^{i+k}$, or equivalently $C[k]_i=C_{i-k}$, with differential $\partial_{C[k]}=(-1)^k\partial_{C}$.

If $\phi:E\to F$ is a morphism of chain complexes, then we have 
$$
\Cone(E\to F)^{i}:=E^{i+1}\oplus F^{i}\ .
$$
For $e\in E^{i+1}$ and $f\in F^{i}$ we write 
\begin{equation}\label{wefwefewfewfefewf89798234234234234} 
(e,f)\in \Cone(E\to F)^{i}
\end{equation}
for the corresponding element in the cone. The differential of the cone is given by 
$$d(e,f):=(-de,df-\phi(e))\ .$$
\end{nota}

\begin{ddd}\label{apr1701}
For $p\in\Z$ we define the \emph{absolute Hodge complex   twisted by $p$} as 
\[\hspace{-0.3cm}
\DR_{\C}(p) := \Cone\left(\left((2\pi i)^{p}\hat \cW_{2p}A_{\log, \R}\right) \oplus \left(\hat \cW_{2p} \cap \cF^{p}A_{\log}\right) \to \hat \cW_{2p}A_{\log} \right)[ 2p -1] \in \PSh_{\Ch}(\Sm_{\C}) \] 
where the map defining the cone is given by $(\omega, \eta) \mapsto \omega - \eta$.
Furthermore, we set $$\DR_{\C} :=\prod_{p\ge 0} \DR_{\C}(p)\ .$$
\end{ddd}

Note that by Lemma \ref{lem:DescentHodgeCplx} we can consider
\begin{equation}\label{dez1406}\DR_{\C},\:\DR_{\C}(p)\in \Fun^{desc}(\bS_{\C}, \Nerve(\Ch)[W^{-1}])\ .\end{equation}
 The cohomology of $\DR(p)(X)$ is, up to a shift, the absolute Hodge cohomology of $X$ as defined by Beilinson in \cite{MR862628}:
\begin{equation}\label{jan3001}
H^k(\DR_{\C}(p)(X)) \cong H^{k+2p}_{\text{Hodge}}(X,\R(p)).
\end{equation}

\begin{ex}\label{kdjqlkwdqwdqwdwqdwqdwqdwqd}
As an illustration, we calculate the absolute Hodge cohomology of 
$\Spec(\C)$. We get $$
H^k(\DR_{\C}(p)(\Spec(\C)))\cong 
\begin{cases}
\C/i^{p}\R & \text{if } p\ge 1, k=1-2p,\\
\R &\text{if } p=0,k=0,\\
0 & \text{else.}
\end{cases}
$$
Here we write $\C/i^{p}\R$ instead of $\R$ in order to indicate  the action of the  Galois group $\Gal(\C/\R)\cong \Z/2\Z$ by complex conjugation. This will be further used in Example \ref{wklqdkjkqwdjqwldwqdqdwqdwqd}. 
\end{ex}

 \subsection{Arithmetic sites and the absolute Hodge complex $\DR_{\Z}$} \label{jul1001}
 
In Subsection \ref{sec:AbsHodgeComplex} we described the absolute Hodge complex $\DR_{\C}(X)$ of a smooth variety $X$  {over $\C$}.
In the present subsection we extend this construction to noetherian, regular, and separated schemes $X$ such that  the base change $X\otimes \Q:=\Spec(\Q)\times_{\Spec(\Z)}X$ is of finite type over $ \Spec (\Q)$.  {We call them arithmetic schemes.} The basic idea  is to apply $\DR_{\C}$ to   the complexification $X\otimes \C:= \Spec (\C)\times_{ \Spec(\Z)}X$,  which is a smooth variety over $\C$,  and to take the action of complex conjugation into account. 

We let $\Reg_{\Z}$ be the site of {arithmetic schemes}
 with the topology given by Zariski open coverings. We further introduce the $\infty$-category $$\bS_{\Z}:=\Nerve(\Reg_{\Z}^{op})\ .$$

The group $\Gal(\C/{\R})\cong \Z/2\Z$ acts by complex conjugation   on the site $\Sm_\C$.
 {It also acts on the site of complex manifolds:}
The
 complex conjugate of a complex manifold $X$ is the complex manifold
$\bar X$ which   has the same underlying smooth manifold as $X$ but the
opposite complex structure. If $X$ is complex algebraic, then so is $\bar X$.

The base change functor
$$B : {\bReg}_\Z\to \Sm_\C\ , \quad B(X):=X\otimes \C$$
 is compatible with the topologies and $\Z/2\Z$-invariant in the sense that
there exists a natural isomorphism 
\begin{equation}\label{jan3101}
u_{X}:B(X)\stackrel{\sim}{\to} \overline{B(X)}
\end{equation}
given  by complex conjugation.

   An equivariant structure on a presheaf $F\in \PSh_{{\cC}}(\Sm_{\C})$ with values in a category $\cC$  is given by a natural isomorphism $c_{X}:F(X)\to F(\bar X)$ for all $X\in \Sm_{\C}$ such that $c_{\bar X}\circ c_{X}=\id$. For example, the sheaf of complex differential forms $A\in\PSh_{\Ch}(\Sm_{\C})$ is equivariant by
$c_{X}: A(X)\stackrel{\sim}{\to} A(\bar X)$,  $c_{X}(\omega):= \bar \omega$. 
This action clearly preserves the real subspace $A_{\R}(X)$,   restricts to  $c_{X}:A_{\log,\R}(X)\stackrel{\sim}{\to}  A_{\log,\R}(\bar X)$,  and preserves the weight and Hodge filtrations.
Hence it  induces an equivariant structure on the absolute Hodge complexes $\DR_{\C}(p)$ for $p\in \nat$ and their product $\DR_{\C}$.

  If a presheaf 
$F\in \PSh_{{\cC}}(\Sm_{ \C })$ is equivariant, then the pull-back $B^{*}F\in \PSh_{{\cC}}( {\bReg}_\Z)$ has a natural $\Z/2\Z$-action given by $\omega\mapsto u_{X}^{*} c_{B(X)}(\omega)$ for $\omega\in F(B(X))$.

If $F$ is a chain complex  with an action of a group $G$, then we let $F^G\subseteq F$ denote the subcomplex of  $G$-invariants. In the case of presheaves of chain complexes with $G$-action we take the invariants object-wise.

\begin{ddd}\label{nov2220}
We define the arithmetic versions of the absolute Hodge complex by
$$\DR_{\Z}(p):=(B ^{*}\DR_{\C}(p))^{\Z/2\Z}\in \PSh_{\Ch}( {\bReg}_{\Z})$$ and
$$\DR_{\Z}:=  (B ^{*}\DR_{\C})^{\Z/2\Z}\in \PSh_{\Ch}( {\bReg}_{\Z})\ .$$
\end{ddd}
Note that $\DR_{\Z}\cong \prod_{p\ge 0}\DR_{\Z}(p)$.

Let  {$G\Ch_{\Q}$} be the  category of chain complexes of $\Q$-vector spaces with a $G$-action.
If $G$ is finite, then the functor $G\Ch_{\Q}\to  \Ch$ given by $F\mapsto F^{G}$ and forgetting the $\Q$-vector space structure preserves quasi-isomorphism. 
Since
 $B^{*}\DR_{\C}$ takes values in $\Z/2\Z\Ch_{\Q}$, 
 the object-wise operation of 
taking invariants  $B^{*}\DR_{\C}\leadsto B^{*}\DR_{\C}^{\Z/2\Z}$  preserves descent when we consider them as presheaves with target in $\Nerve(\Z/2\Z\Ch)[W^{-1}]$ and  $\Nerve( \Ch)[W^{-1}]$, respectively.
 It thus follows from    \eqref{dez1406}  that  the arithmetic versions of the absolute Hodge complexes have the descent property, too:
$$\DR_{\Z}(p),\:\DR_{\Z}\in \Fun^{desc}(\bS_{\Z},\Nerve(\Ch)[W^{-1}])\ .$$

\begin{ex} \label{wklqdkjkqwdjqwldwqdqdwqdwqd} 
As an illustration, we calculate the absolute Hodge cohomology of 
$X=\Spec(R)$ for a number ring $R$.  Note that $X(\C)$ is the set of embeddings $R\hookrightarrow \C$. It decomposes into the union 
 $$X(\C)=X(\C)^{\Gal(\C/\R)}\sqcup X(\C)^{\prime}$$ 
 {of real and complex embeddings.}
Then $H^{k}(\DR_{\Z}(p)(X))$ decomposes as a sum over the orbits $X(\C)/\Gal(\C/\R)$. Every   {real embedding} $\sigma\in 
X(\C)^{\Gal(\C/\R)}$ contributes a copy of $H^{k}(\DR_{\C}(p)(\Spec(\C)))^{\Gal(\C/\R)}$, while every
 {pair  $(\sigma, \bar\sigma)\in X(\C)^{\prime}/\Gal(\C/\R)$ of a complex embedding and its conjugate}
contributes a copy of $H^{k}(\DR_{\C}(p)(\Spec(\C)))$.
  Using Example \ref{kdjqlkwdqwdqwdwqdwqdwqdwqd}, we get
$$
H^k(\DR_{\C}(p)(\Spec(\C)))^{\Gal(\C/\R)}\cong 
\begin{cases}
\R & \text{if } p\ge 1 \text{ odd}, k=1-2p,\\
\R &\text{if } p=0,k=0,\\
0 & \text{else.}
\end{cases}
$$
We define the integers 
\begin{equation}\label{dklqwkldwqdwqdwqdwqdwqdwd}
r_{1}:= {\#} X(\C)^{\Gal(\C/\R)}\ , \quad r_{2}:= {\#}(X(\C)^{\prime}/\Gal(\C/\R))\ .
\end{equation}
Then we have
$$
H^k(\DR_{ {\Z}}(p)( {X}))^{\Gal(\C/\R)}\cong 
\begin{cases}
\R^{r_{1}+r_{2}} & \text{if }p\ge 1 \text{ odd}, k=1-2p,\\
\R^{r_{2}} &\text{if } p\ge 1 \text{ even}, k=1-2p,\\
\R^{r_{1}+r_{2}} & \text{if }p=0, k=0
\\0 & \text{else.}
\end{cases}
$$ 
 \end{ex}

\subsection{$K$-theory}
\label{feb1002}

In this subsection we discuss the definition of algebraic $K$-theory in the complex and the arithmetic case.
A vector bundle on a smooth variety $X$ over $\C$  {or on an arithmetic scheme $X$} is a  locally free sheaf of $\cO_{X}$-modules which is locally of finite rank. The category $\Vect_{\C}(X)$ (or $\Vect_{\Z}(X)$, respectively) of vector bundles on $X$ is symmetric monoidal with respect to the direct sum. A morphism $X\to Y$   gives rise to a symmetric monoidal functor
$\Vect_{\C}(Y)\to \Vect_{\C}(X)$ (or $\Vect_{\Z}(Y)\to \Vect_{\Z}(X)$, respectively).  Finally, vector bundles can be glued along open coverings.
All this can be formalized by saying that we have   sheaves
\begin{equation}\label{rh23kr23r32e3e3er32r32r3r}
\Vect_{\C}\in \Fun^{desc}(\bS_{\C},\CommMon(\Nerve(\Cat)[W^{-1}]))\ ,
\end{equation} 
and 
$$
\Vect_{\Z}\in \Fun^{desc}(\bS_{\Z},\CommMon(\Nerve(\Cat)[W^{-1}]))\ .
$$
 {If we apply the algebraic $K$-theory functor from Definition~\ref{nov1901} to these sheaves, then we get presheaves of spectra 
$K(\Vect_{\C}) \in \Fun(\bS_{\C},\Sp)$  and $K(\Vect_{\Z}) \in \Fun(\bS_{\Z},\Sp)$. Note that, since
$K$ involves the group-completion $\Omega B$, which is a left-adjoint and does not preserve limits, in general, $K$ does not preserve sheaves. 
We} will use the sheafification functors  {$L$} which are the left-adjoints of the inclusions of sheaves into presheaves  (see Subsection \ref{dez2701}) 
$$
L:\Fun(\bS_{\C},\Sp)\leftrightarrows \Fun^{desc}(\bS_{\C},\Sp)\ , \quad L:\Fun(\bS_{\Z},\Sp)\leftrightarrows \Fun^{desc}(\bS_{\Z},\Sp)\ .
$$
\begin{ddd}\label{jul0870}
We define the sheaves of $K$-theory spectra by 
$$
\bK_{\C}:=L(K(\Vect_{\C}))\in \Fun^{desc}(\bS_{\C},\Sp)\ , \quad \bK_{\Z}:=L(K(\Vect_{\Z}))\in \Fun^{desc}(\bS_{\Z},\Sp).
$$
\end{ddd} 


%

\begin{rem} In this remark
  we explain why Definition \ref{jul0870} reproduces 
 the    standard version of algebraic $K$-theory  {as defined by Quillen \cite{Quillen}.}
In order to be specific we consider the arithmetic case.
%

We consider the subcategory of  affine schemes
$\Reg^{aff}_{\Z}\subset\Reg_{\Z}$ and the corresponding inclusion  of nerves
 $\bS_{\Z}^{aff}\subset \bS_{\Z}$.
If $\Spec(R)\in \Reg^{aff}_{\Z}$, then we have an equivalence
of symmetric monoidal groupoids
$$
\Vect_{\Z}(\Spec(R))\simeq \cP(R) \ ,
$$ 
where $\cP(R)$ is as in  Example \ref{kljdklqwdwqdwqdwqdwqdq}.  
As argued in this example, we have the first equivalence in the chain $$KR\simeq K(\cP(R))\simeq K(\Vect_{\Z}(\Spec(R)))$$ (see Remark \ref{djqwkjdqwdqwdwqdwqd12312321323} for the two usages of the symbol $K$).
 It follows that the functor
$$K(\Vect_{\Z})_{|\bS_{\Z}^{aff}}\in \Fun(\bS_{\Z}^{aff}, \Sp)$$
associates to each object $\Spec(R)\in \bS_{\Z}^{aff}$
the usual $K$-theory spectrum $KR$ of the ring $R$. It is further known {that}
$K(\Vect_{\Z})_{|\bS_{\Z}^{aff}}$ satisfies Zariski descent {(see below)}. The latter property allows us to drop the sheafification in the definition \eqref{jul0870} of $\bK_\Z$ restricted to $\bS_{Z}^{aff}$, and we have 
 an equivalence
$$(\bK_{\Z})_{|\bS_{Z}^{aff}}\simeq K(\Vect_{\Z})_{|\bS_{\Z}^{aff}}\in \Fun^{desc}(\bS_{\Z}^{aff}, \Sp )\ .$$

If $X\in \Reg_{\Z}$ is not affine,  an exact sequence of vector bundles on $X$ does not split  in general.
Because of this the $K$-theory of $X$ is usually  {not} defined   by group completion, but by applying a different $K$-theory machine to the exact  category of vector bundles. 
Such a machine
leads to a functor
$$\widetilde\bK_{\Z}\in \Fun(\bS_{\Z}, \Sp).$$
Since the schemes in $\bReg_{\Z}$ are separated, noetherian, and regular,  {Quillen's} $K$-theory of vector bundles coincides with 
Thomason's $K^B$-theory  {\cite[6.4]{TT}} which satisfies Zariski descent \cite[7.6, 8.4]{TT}.\footnote{ {In fact, this follows already from Quillen's localization and resolution theorems.}}
We conclude that
$$\widetilde\bK_{\Z}\in \Fun^{desc}(\bS_{\Z}, \Sp )\ .$$
Since every scheme $X\in \Reg_{\Z}$ admits a Zariski covering by affines,  
the restriction
$$\Fun^{desc}(\bS_{\Z},{\Sp})\to \Fun^{desc}(\bS_{\Z}^{aff},{\Sp})$$
is an equivalence. 
In particular, since
$$
(\bK_{\Z})_{|\bS_{\Z}^{aff}} \simeq K(\Vect_{\Z})_{|\bS_{\Z}^{aff}} \simeq (\widetilde\bK_{\Z})_{|\bS_{\Z}^{aff}}\ ,
$$
we conclude that
$\bK_{\Z} \simeq \widetilde\bK_{\Z}$.
\end{rem}

\subsection{Beilinson's conjectures}\label{nov2602}

In the present paper the Beilinson regulator is a  morphism of sheaves of spectra
$$\beil :\bK_\Z\to H(\DR_{\Z})\ ,$$ see  {\eqref{jul0860-neu}}.
The construction of $ \beil $ in this form is technical and occupies a large part of Section \ref{jan1962}.
As we  {will} verify in Subsection \ref{nov2202}, the induced map in homotopy groups
 $$\beil:\pi_{i}(\bK_{\Z}(X))\to H^{-i}(\DR_{\Z}(X))$$
 for $X\in \Reg_{\Z}$
 coincides with classical definitions.

In this subsection, we describe, as a service for the interested reader,  the conjectural picture of kernel and cokernel of the regulator map $\beil$ as given by the conjectures of Beilinson and Parshin \cite[\S 8]{MR760999, MR862628}. 
 {For simplicity, we} 
fix  a scheme $X \in \bReg_{\Z}$, proper and flat over $ \Spec(\Z)$, which has potentially good reduction at every prime. For example, $X$ could be 
smooth and proper over the ring of integers in some number field. 

Then we have homomorphisms
$$
\beil:\pi_{i}(\bK_{\Z}(X))\to H^{-i}(\DR_{\Z}(X))
$$
for all $i\ge 0$.
\begin{con}[Beilinson]\phantomsection
\label{mar0902}
\begin{enumerate}
\item For $i\ge 2$ the regulator induces an
isomorphism
\begin{equation}\label{feb0201}
\pi_i(\bK_{\Z}(X))\otimes \R \xrightarrow{\cong}  H^{-i}(\DR_{\Z}(X))\ .
\end{equation}
\item
For $i=1$  the map {\eqref{feb0201}} is expected to be injective.
Let $N^{p}(X_{ {\Q}})$ be the group of codimension $p$-cycles in  {the generic fibre} $X_{ {\Q}}$ modulo
homological equivalence and $N(X_{ {\Q}}):=\prod_{p\ge 0} N^{p}(X_{ {\Q}})$. Then 
 the cycle map in de Rham cohomology
\[
N^{p}(X_{\Q}) \hookrightarrow H^{2p}(B(X);\Omega^{\bullet}_{B(X)}) \cong H^{2p}(\hat\cW_{2p}A_{\log}(B(X)))
\]
together with the natural map $H^{2p}(\hat\cW_{2p}A_{\log}(B(X))) \to H^{-1}(\DR_{\Z}(p+1)(X))$, coming from the definition of $\DR_{\Z}( {p+1})$ as a cone, induce an injection
\[
\bar z^p: N^p(X_{\Q}) \hookrightarrow H^{-1}(\DR_{\Z}(p+1)(X))\ .
\]

We let $\bar z:=\prod_{p\ge 0} \bar z^{p}$.
It is expected that
$$ {\beil}\oplus \bar z:\pi_1(\bK_{\Z}(X))\otimes \R\oplus N(X_{ {\Q}})\otimes \R \to H^{-1}(\DR_{\Z}(X))  $$ is an isomorphism.
\item 
For $i=0$ we have
$$H^{0}(\DR_{\Z}(p)(X))\cong \left(H^{2p}(X\otimes_{\Z} \C;\R(p))\cap \cF^{p} H^{2p}(X\otimes_{\Z}\C;\C)\right)^{\Z/2\Z}\ .$$
Hence $H^{0}(\DR_{\Z}(X))$ is the group of real Hodge classes on $X\otimes_{\Z}\C$.
As the example of number rings shows, the map  $ \pi_{0}(\bK_{\Z}(X))\otimes \R\to  H^{0}(\DR_{\Z}(X))$ can be far from being surjective.
 \end{enumerate}
 \end{con}

The above conjecture is a consequence of Conjectures 8.4.1 and 8.3.3 (Parshin's conjecture) in \cite{MR862628}.
Beilinson formulates his conjecture in terms of motivic cohomology, which he defines as certain Adams eigenspaces of rational algebraic $K$-theory. One obtains the above statement by taking the sum over all Adams eigenspaces. Moreover, instead of the $K$-theory $\pi_i(\bK_\Z(X))=K_i(X)$ of $X$, he uses the image of $K_i(X)$ in the $K$-theory $K_i(X_\Q)$ of the generic fibre $X_\Q$. However, under the assumption that $X$ is regular and has potentially good reduction everywhere, Quillen's localization sequence for the $K$-theory of coherent sheaves together with Parshin's conjecture on the vanishing of rational $K$-theory of smooth proper varieties over finite fields imply, that the map $K_i(X) \to K_i( {X_\Q})$ is rationally injective.
One can weaken the assumptions on $X$, but then the statement of the conjectures becomes more complicated. 
\begin{ex}\label{dkqjwdqwdqwdwqdwqdqdwqd}
 {Using a different construction of the regulator for number rings, Borel \cite{MR0387496} has computed the ranks of the $K$-groups of a number ring $R$. Using the notation of Example \ref{wklqdkjkqwdjqwldwqdqdwqdwqd}, his result is given by
\[
\pi_{k}(KR)\otimes \R\cong 
\begin{cases}
\R^{r_{1}+r_{2}} & \text{if }k=2p-1, p\ge 3 \text{ odd} \\
\R^{r_{2}} & \text{if } k=2p-1\ , p\ge 2 \text{ even}\\
\R^{r_{1}+r_{2}-1} & \text{if } k=1, \\
 \R & \text{if } k=0 \\
0 & \text{if } k<0.
\end{cases}
\]
In view of the computations in Example \ref{wklqdkjkqwdjqwldwqdqdwqdwqd}, and using the comparison between the regulators of Borel and Beilinson, this proves the above conjecture for $X=\Spec(R)$. Essentially, this is the only known case of Beilinson's conjectures.
}
\end{ex}

\section{Smooth manifolds}\label{jan1962}

In this section we extend the definition of the absolute Hodge complexes and $K$-theory to the setting where we consider products $M\times X$  of an arbitrary smooth manifold $M$ with  {a smooth} variety  over $\C$ or an arithmetic scheme $X$. We introduce the notion of good geometries and local geometries. We define
   characteristic forms in the absolute Hodge complex and in its \v{C}echification, respectively.
  We  then construct the regulator map using these characteristic forms. In \ref{nov2202}, we show that our construction of the regulator reproduces Beilinson's one if one specializes  the manifold $M$ to a point. Finally, in the last subsection, we apply the techniques of the current paper in order to compare Karoubi's relative Chern character with Beilinson's regulator.


\subsection{The sites}\label{kjefwefewfewfewfewfe}

By a smooth manifold we understand a smooth manifold with corners  {(see, for example, \cite[\S 2.1.3]{Bunke-index})}. Corners of codimension $n$ in $k$-dimensional manifolds are modelled on $[0,\infty)^{n}\times \R^{k-n}$. For simplicity, we require that our transition maps preserve the germs of  product coordinates to the boundary faces. We use  these product coordinates in order to define the notion of product structures for various geometric objects.
The most important example of a manifold with corners is the standard simplex $\Delta^{n}$.

We consider the category $\Mf$ of smooth manifolds as a site equipped with the topology 
of open coverings. The  nerve of its opposite is the $\infty$-category $\bS_{Mf}:=\Nerve(\Mf^{op})$.

We let $\Sm_\C$ denote the site of smooth varieties over $\C$ (see Subsection \ref{sec:AbsHodgeComplex}) and Zariski open coverings, and we set
$\bS_\C
:=\Nerve(\Sm_\C^{op})$.

Finally, we let $\bReg_\Z$ be the site  of noetherian,  regular, and separated schemes $X$ such that   $X\otimes \Q $ is of finite type over $ \Spec (\Q)$,
 with the topology given by Zariski open coverings
(see  
  Subsection \ref{jul1001})  and write
$\bS_\Z
:=\Nerve(\bReg_\Z^{op})$.  {We remind the reader that we also call objects in $\Reg_{\Z}$ arithmetic schemes.}

We consider the product sites $$\Mf\times \Sm_\C\ , \quad \Mf\times \Reg_\Z$$ and write
$$\bS_{Mf,\C}:=\bS_{Mf}\times \bS_{\C}\ , \quad \bS_{Mf,\Z}:=\bS_{Mf}\times \bS_{\Z} $$
for the corresponding $\infty$-categories. We write objects of these product sites in the form
$M\times X$, where $M\in \Mf$ and $X$ is the algebraic object. We have canonical topology preserving   functors
\begin{equation}\label{jul1002}
e:  \bS_\C\to \bS_{Mf,\C}\ , \quad e: \bS_\Z\to \bS_{Mf,\Z}
\end{equation}
both induced by $X\mapsto *\times X$, where $*$ is the point considered as a manifold.

\begin{ex}
Here are some typical objects of these sites.
\begin{enumerate}
\item $S^{1}\times \P^{1}_{\C}\in \bS_{Mf,\C}$,
\item $T^{2}\times \Spec(\Z[\sqrt{2}])\in  \bS_{Mf,\Z}$  {where $T^{2}$ is the 2-dimensional torus,}
\item $\Delta^{n}\times E\in  \bS_{Mf,\C}$ for an elliptic curve $E$ over $\C$,
\item $T^{2}\times \Spec(\Q[\sqrt{2}])$.
\end{enumerate}
 Note that 
$S^{3} \times \Spec(\R)$ does not belong to
$ \bS_{Mf,\Z}$ since $\Spec(\R)\cong \Spec(\R)\otimes \Q$ it is not of finite type over $\Spec(\Q)$. Examples of morphisms are
\begin{enumerate}
\item $\R\times  \P^{1}_{\C}\to S^{1}\times \P^{1}_\C$ induced by the universal covering $\R\to S^{1}$,
\item $T^{2}\times \Spec(\Z[\sqrt{2}])\to *\times \Spec(\Z)$ induced by $T^{2}\to *$ and $\Spec(\Z[\sqrt{2}])\to \Spec(\Z)$.
\end{enumerate}
Note that the flip 
$\P_{\C}^{1}(\C)\times *\to *\times \P^{1}_{\C}$
is not a morphism in $\bS_{Mf,\C}$.
\end{ex}

We will often consider $M\times X$ as a ringed space with structure sheaf $\pr_X^*\cO_X$.

\begin{rem}
One can also interpret $M\times X$ as a relative scheme over the ringed space $(M, \underline{\Z})$ in the sense of \cite{Hakim}, where $(M,\underline{\Z})$ denotes the underlying topological space of the manifold $M$ with the  {constant sheaf $\underline{\Z}$ as structure sheaf}.
\end{rem}

\subsection{The sheaves $\DR_{Mf,\C}$ and $\DR_{Mf,\Z}$}
 \label{apr0501}

We start with the  construction of  the analogue  
$$
A_{\log}\in \PSh_{\Ch}(\Mf\times \Sm_{\C})
$$ 
of the bifiltered complex introduced in Definition \ref{dez1401}. 
Let $M$ be a manifold and $X$ a smooth variety over $\C$  (cf.~Subsection \ref{sec:AbsHodgeComplex}). 
We fix some good compactification 
$$X \overset{j}{\hookrightarrow} \overline X$$ 
and write $D := \overline X - X$ for the divisor at infinity. We consider $X$ and $\overline X$ as  smooth manifolds. In this sense, the product $M\times X$ is a manifold and $A_{\R}(M\times X)$ is defined as in \eqref{frefewfwffewfewfewfewfe}.
We define the chain complex   
 $$A_{M\times\overline X,\R}(M\times X,\log D) \subseteq A(M\times X)_{\R}$$ to be the subcomplex which is locally generated as an  algebra over $A(M\times\overline X)_{\R}$  by $1$ and  
\begin{equation}
\log(z_{i}\bar z_{i}),\,\: \Re \frac{dz_{i}}{z_{i}},\, \:\Im \frac{dz_{i}}{z_{i}}, \text{ for }i\in I\ . \label{eq:wt1}
\end{equation}
Here  the $z_i$, $ i \in I$, are local coordinates of $\overline X$  (for the analytic topology) which define  $D$ locally by the equation $\prod_{i_\in I} z_i = 0$. 

\begin{ex}
The following example should clarify the meaning of the notion {\em locally generated}. We consider $X:= {\P_{\C}^{1}\setminus\{0\}}$ with local coordinate $z$ at $ {0}$ and $M =\R$. We let
$(\chi_{n})_{n\in \Z}$ be a partition of unity on $\R$ such that
$\supp(\chi_{n})\subset [n-2,n+2]$ for all $n\in \Z$. Then
\begin{equation}\label{dez3105}
\sum_{n\in \Z} \log(|z|)^{|n|} \chi_{n}\in A^{0}_{\R\times {\P^{1}_{\C}}}(\R\times X, \log  {\{0\}})\ .
\end{equation}
\end{ex}

We are going to introduce several filtrations on $A_{M\times\overline X,\R}(M\times X,\log D)$.
The naive weight filtration $\widetilde \cW$ is the multiplicative increasing filtration by  $A(M\times\overline X)_{\R}$-modules obtained by assigning weight $0$ to the section $1$  and weight $1$ to the sections listed in  \eqref{eq:wt1}.

We define a decreasing filtration $\cL$ of $A_{ {M\times \overline X,\R}}(M\times X,  {\log D})$ such that $\cL^pA_{ {M\times \overline X,\R}}(M\times X,  {\log D})$ is the subcomplex of differential forms that are given locally by
\[
\sum _{I,J,K,|I|\ge p} \omega_{I,J,K} \:dx^{I}\wedge \Re dz^{J}\wedge \Im dz^{K}\ ,
\]
where the $x_i$ are local coordinates for $M$, the $z_j$ local coordinates for $\overline X$, and $\omega_{I,J,K}$ local  smooth functions on  $M\times X$.  

We now define the weight filtration $\cW$ as the diagonal filtration of $\widetilde \cW$ and $\cL$:
\begin{equation}\label{dgqwdgqwjhdgwqjdhwqdq}
\cW_k A_{M\times\overline X,\R}(M\times X,\log D) := \sum_p \widetilde \cW_{k+p} \cap \cL^p A_{M\times\overline X,\R}(M\times X,\log D).
\end{equation}
As usual, its d\'ecalage {(cf.~\eqref{sep0407})} will be denoted by $\hat \cW_{*}
 A_{M\times\overline X,\R}(M\times X,\log D)$. 

We further define the complex dg-algebra $$A_{M\times \overline X}(M\times X,\log D) := A_{M\times\overline X,\R}(M\times X,\log D) \otimes_{\R} \C$$ with the induced weight filtration. This complex carries the  decreasing Hodge filtration $\cF$ such that the elements of  $\cF^pA_{M\times\overline X}(M\times X,\log D)$   are locally of the form
\[
\sum _{I,J,K,|J|\ge p} \omega_{I,J,K} \:dx^{I}\wedge dz^{J}\wedge d\bar z^{K},
\]
where the $x_i$ and $z_j$ are local coordinates of $M$ and $X$, respectively.

In the special case that $M=\ast$, the complex  $A_{\ast\times\overline X,\R}(\ast\times X,\log D)$ 
is exactly the complex of global sections of 
Burgos' sheaf of complexes \cite[Section 2]{BurgosLogarithmic}.  With its filtrations it gives rise to the mixed Hodge complex    \eqref{jul1010}.

Fix $X \hookrightarrow \overline X$ as above. Then  the functors
\begin{align}\begin{split}\label{eq:HodgeCplxMf}
&M \mapsto A_{M\times\overline X,\R}(M\times X,\log D)\ ,\\
&M \mapsto  \hat \cW_{k}A_{M\times\overline X,\R}(M\times X,\log D),  \\
&M \mapsto A_{M\times\overline X}(M\times X,\log D), \\
&M \mapsto \hat  \cW_{k}\cap \cF^{p} A_{M\times\overline X}(M\times X,\log D)
\end{split}
\end{align}
are defined
as objects of $\PSh_{\Ch}(\Mf)$. 
 They are sub-presheaves of the presheaf
  $M\mapsto A(M\times X)$.
  
Let $\cC$ be any presentable $\infty$-category.  A functor $F\in \Fun(\bS_{Mf},\cC)$ is called homotopy invariant, 
if for every manifold $M$ the map $F( M)\to F(I\times M)$ induced by the projection $I\times M\to M$ (with $I=[0,1]$) is
 an equivalence. We define the notion of sheaves (see \eqref{eq:sheaves}) using the covering families of the site $\Mf$.
Finally, we write $$
\Fun^{desc,h}(\bS_{Mf},\cC)\subseteq \Fun (\bS_{Mf},\cC)
$$
for the full subcategory of functors which are sheaves and homotopy invariant. We refer to {Subsection}~\ref{mar0804} for {further} details.   
 
For an illustration of the sheaf condition we refer to Example \ref{ex:sheaves}.
\begin{ex}In this example we illustrate the condition of being homotopy invariant. An abelian group will be  considered as a chain complex concentrated in degree zero.
\begin{enumerate}
\item Let  $C^{\infty}$  be the sheaf which maps a manifold $M$ to its smooth complex valued functions $C^{\infty}(M)$.   
Similarly as in Example \ref{ex:sheaves}, we can consider $C^{\infty}$ as an object in $\Fun^{desc}(\bS_{Mf}, \Nerve(\Ch))$ or in $\Fun^{desc}(\bS_{Mf}, \Nerve(\Ch)[W^{-1}])$. In both cases, it is not homotopy invariant.
\item The constant sheaf $ \underline{\Z}$,  considered as an object in $\Fun(\bS_{Mf}, \Nerve(\Ch))$ is a homotopy invariant sheaf, i.e.~it belongs to $\Fun^{desc,h}(\bS_{Mf},\Nerve(\Ch))$. It is also homotopy invariant when considered as an object of $\Fun(\bS_{Mf}, \Nerve(\Ch)[W^{-1}])$.
\item   {The de Rham complex $A$, considered as an object in $\Fun^{desc}(\bS_{Mf}, \Nerve(\Ch)[W^{-1}])$ is homotopy invariant. This follows from the Poincar\'e lemma. It is, of course, not homotopy invariant when considered as an object of  $\Fun^{desc}(\bS_{Mf}, \Nerve(\Ch))$.} 
 \end{enumerate}
\end{ex} 
The objects listed in \eqref{eq:HodgeCplxMf}   can naturally be considered as objects of $\Fun(\bS_{Mf},\Nerve(\Ch)[W^{-1}])$.
 As such, one can ask whether they satisfy descent, i.e.~are sheaves, or  whether they  are homotopy invariant.
\begin{lem}\label{dez1405}
The presheaves of complexes \eqref{eq:HodgeCplxMf}  satisfy descent and are homotopy invariant, i.e.~they belong
to
$\Fun^{desc,h}(\bS_{Mf},\Nerve(\Ch)[W^{-1}])$.
\end{lem}
\begin{proof}
In order to verify   descent, we show below that the presheaves of complexes \eqref{eq:HodgeCplxMf} 
 are  degree-wise    sheaves of modules over the sheaf $C_{\R}^{\infty}\in \Sh_{\Alg(\R)}(\Mf)$ of algebras of smooth real valued   functions.
 It then follows from Lemma \ref{dqwhdlqdwqdqwdwqd}   that the presheaves listed in \eqref{eq:HodgeCplxMf} satisfy descent, i.e.~belong to $\Fun^{desc}(\bS_{Mf},\Nerve(\Ch)[W^{-1}])$.
 
We discuss the case of  $\hat \cW_{k} A_{\dots\times\overline X,\R}(\dots\times X,\log D) $.  The other cases are similar.

As a preparation, note that for an $n$-form $\omega\in \cW_{k-n} $ (see \eqref{dgqwdgqwjhdgwqjdhwqdq})
 we have $\omega \in \hat \cW_k $ if and only if $d^X\omega \in \cW_{k-n-1} $, where $d^X$ is the differential 
in $X$-direction. Indeed, we may assume that $\omega \in \widetilde \cW_{k-n+p}\cap \cL^p $ for some $p$. 
Then $d\omega = d^M\omega + d^X\omega$ and $$d^M\omega \in \widetilde W_{k-n+p}\cap \cL^{p+1}   \subseteq \cW_{k-n-1} \ .$$

Now we assume that  $$\omega \in \hat \cW_kA_{M\times\overline X,\R}^{n}(M\times X, \log D)\ ,\quad f\in C^{\infty}(M,\R)\ .$$  Then obviously $f\omega \in \cW_{k-n}  $ and $d^X(f\omega)=fd^X\omega \in \cW_{k-n-1}   $, hence $$f\omega \in \hat \cW_k A_{M\times\overline X,\R}^{n}(M\times X, \log D) \ .$$ 

In order to   verify homotopy invariance,  we show that the  integration $$\int_I\colon A^{*}(I\times M \times X) \to A^{*-1}(M\times X)$$ preserves the subcomplexes \eqref{eq:HodgeCplxMf}. 
This fact is then employed as follows. Let $i_{t}:M\to I\times M$, $t=0,1$, be the inclusions given by the endpoints of the interval. The integral provides a chain homotopy between $i_{1}^{*}$ and $i_{0}^{*}$ applied to one of the presheaves listed in  \eqref{eq:HodgeCplxMf}. We now consider the map $a:I\times I\times M\to I\times M$ given by $a(s,t,m)=(st,m)$.
Then we have
$$a\circ (i_{0}\times \id_{I\times M}) =i_{0}\circ \pr\ , \quad  a\circ i_{1}= \id_{I\times M}\ .$$
Consequently, we get a chain homotopy between
$\pr^{*}\circ i_{0}^{*}$ and $\id$ which exhibits $i_{0}^{*}$ as a chain homotopy inverse of $\pr^{*}$.

First note that $\int_I$ preserves the subcomplex of forms with logarithmic singularities along $D$. For the filtrations we only 
discuss the case $\hat \cW_{k} A_{\dots\times\overline X,\R}(\dots\times X,\log D) $.     Take $$\omega \in  \hat \cW_kA^{n}_{I\times M\times\overline X,\R}(I\times M\times X,\log D)\ .$$ We may again assume that $\omega \in \widetilde \cW_{k-n+p}\cap \cL^p    $ for some $p$. Then $$\int_I \omega \in \widetilde \cW_{k-n+p}\cap \cL^{p-1}  
 \subseteq \cW_{k-n+1} $$ and $d^X\int_I\omega = -\int_I d^X\omega \in \cW_{k-n} $ since $d^X\omega \in \cW_{k-n-1} $. This shows that $$\int_I\omega \in \hat \cW_k A^{n-1}_{M\times\overline X,\R}(M\times X,\log D)\ .$$
\end{proof}


Now assume that we have two good compactifications $\overline X$ and $\overline X'$ of $X$ and a morphism $\overline X' \to \overline X$ inducing the identity on $X$. Then the induced maps
\begin{align*}
&A_{M\times\overline X,\R}(M\times X,\log D)  \to A_{M\times\overline X',\R}(M\times X,\log D')\ ,\\
& \hat \cW_k A_{M\times\overline X,\R}(M\times X,\log D) \to \hat \cW_k A_{M\times\overline X',\R}(M\times X,\log D')\ , \\
&\cF^pA_{M\times\overline X}(M\times X,\log D) \to   \cF^p A_{M\times\overline X'}(M\times X,\log D')\ ,\\
&\hat \cW_{k}\cap \cF^pA_{M\times\overline X}(M\times X,\log D) \to  \hat \cW_{k}\cap \cF^p A_{M\times\overline X'}(M\times X,\log D')
\end{align*}
are quasi-isomorphisms. Indeed,   descent and homotopy invariance in the $M$-direction  reduce the claim   to the case  $M=\ast$ and there the claim follows from Hodge theory (cf. the discussion  {following} \eqref{dez1101}).

In order to get rid of the choice of the compactification $\overline X$, we
now  proceed as in {Subsection} \ref{sec:AbsHodgeComplex}. We define the presheaf $A_{\log, Mf,\R}\in \PSh_{\Ch}(\Mf\times \Sm_{\C})$ by
\begin{equation}\label{mar2902}
A_{\log, Mf,\R}(M\times X) := \colim_{\overline X\in I_{X}}  A_{M\times\overline X,\R}(M\times X,\log (\overline X - X))
\end{equation}
where the colimit runs over the directed system $I_{X}$ of all good compactifications of $X$. It has an induced weight filtration $\cW$ and a Hodge filtration $\cF$ on the complexification
$$A_{\log, Mf}(M\times X):=  \colim_{\overline X\in I_{X}}  A_{M\times\overline X}(M\times X,\log (\overline X - X))
 \cong A_{\log, Mf,\R}(M\times X)\otimes_{\R}\C\ .$$
By the above, all maps in the directed system are quasi-isomorphisms which are bifiltered with respect to the d\'ecalage $\hat \cW$ of the weight filtration and the Hodge filtration.

We extend the notion of homotopy invariance to the product site $\Mf\times \Sm_{\C}$: A presheaf $F \in \Fun(\bS_{Mf,\C}, \bC)$ with values 
in the $\infty$-category $\bC$ is called homotopy invariant, if, for all $M\times X \in \Mf\times \Sm_{\C}$, the map $F(M\times X) \xrightarrow{\pr^{*}} F((I\times M)\times X)$ is an equivalence.
The full subcategory of homotopy invariant sheaves will be denoted by $\Fun^{desc,h}(\bS_{Mf,\C},\bC)$.
Alternatively, we could use the equivalence
$$\Fun^{desc,h}(\bS_{Mf,\C},\bC)\simeq \Fun^{desc,h}(\bS_{Mf},\Fun^{desc}(\bS_{\C},\bC))\ .$$
\begin{lem}\label{lem:descentSmfC}
We have 
\[
A_{\log, Mf}\ ,A_{\log, Mf,\R}\ , \hat \cW_{k}A_{\log, Mf,\R}\ , \cF^{p}\cap \hat \cW_{k} A_{\log, Mf}  \in\Fun^{desc,h}(\bS_{Mf,\C}, \Nerve(\Ch)[W^{-1}])\ .
\]
\end{lem}
\begin{proof}
We can check descent in the $M$- and $X$-directions separately.
Since the structure maps of the colimit over $I_{X}$ are quasi-isomorphisms,  it follows from Lemma \ref{dez1405}   that the complexes in the statement of the lemma fulfil descent  and homotopy invariance
 in the $M$-direction.

In order to  show Zariski descent in the $X$-direction, we use descent and homotopy invariance in the $M$-direction in order  to reduce to the case $M=\ast$. In this special case,    the Zariski descent was proven as Lemma \ref{lem:DescentHodgeCplx}.
\end{proof}

 {We now introduce the extension of the absolute Hodge complex of Definition~\ref{apr1701} to the site $\Mf\times \Sm_{\C}$:}
\begin{ddd}\label{nov2031e}
For $p\ge 0$ we define the complex
$$\hspace{-0.5cm}\DR_{Mf,\C}(p):=  \Cone\left(\left((2\pi i)^{p}\hat \cW_{2p}A_{\log, Mf,\R}\right)\oplus \left(\hat \cW_{2p}\cap \cF^{p}A_{\log, Mf}\right)  \xrightarrow{(\alpha,\beta)\mapsto \alpha-\beta } \hat \cW_{2p}A_{\log, Mf}\right)[2p-1] \ ,$$
and $$\DR_{Mf,\C}:=\prod_{p\ge 0} \DR_{Mf,\C}(p)\ .$$
\end{ddd}
By the above, we may view $\DR_{Mf,\C}(p)$ and $ \DR_{Mf,\C}$ as objects in $\PSh_{\Ch}(\Mf\times \Sm_{\C})$.
 By Lemma~\ref{lem:descentSmfC}, they can also considered as objects  
$$ 
 {\DR_{Mf,\C}(p),}\ \DR_{Mf,\C}\in 
 \Fun^{desc,h}(\bS_{Mf,\C},\Nerve(\Ch)[W^{-1}] )
\ .
$$


\begin{ddd}\label{nov2031}
Generalizing Definition \ref{nov2220}, 
we define
$$\DR_{Mf,\Z}\in \PSh_{\Ch}(\Mf\times   {\bReg}_{\Z})$$
by 
$$\DR_{Mf,\Z}:=((\id_{ {\Mf}}\times B)^{*}\DR_{Mf,\C})^{\Z/2\Z}\ .$$
\end{ddd}

We have
$$ 
\DR_{Mf,\Z}\in 
 \Fun^{desc,h}(\bS_{Mf,\Z},\Nerve(\Ch)[W^{-1}] )
\ .$$ 

In order to understand the cohomology of $\DR_{Mf,\C} $ and  $\DR_{Mf,\Z}$, we show a version of the de Rham Lemma. Recall the definitions of the Eilenberg-MacLane correspondence $H$ 
 {from \eqref{klrsefpbpw34978h}}
 and the constant sheaf functor $\underline{...}$ 
defined in \eqref{eq:relconstsheaf} as the composition
\[
\Fun^{desc}(\bS_{\C}, \Nerve(\Ch)[W^{-1}]) \xrightarrow{p^{*}} \Fun(\bS_{Mf,\C}, \Nerve(\Ch)[W^{-1}]) \xrightarrow{L} \Fun^{desc}(\bS_{Mf,\C}, \Nerve(\Ch)[W^{-1}])
\]
of the pull-back along the projection $p\colon \bS_{Mf,\C}\to \bS_{\C}$ and the sheafification $L$, and similarly with $\bS_{Mf,\C}$ replaced by $\bS_{Mf,\Z}$ (see Subsection~\ref{mar0804}). 
Note that  the Eilenberg-MacLane correspondence $H$ preserves limits and therefore maps sheaves with values in $\Nerve(\Ch)[W^{-1}]$ to sheaves with values in $\Sp$.

\begin{lem}\label{jan0201}
  We have  canonical equivalences 
\begin{equation}\label{nov2047}H(\DR_{Mf,\C} ) \simeq \underline{H(\DR_\C)}
\ , \quad  H(\DR_{Mf,\Z} ) \simeq \underline{H(\DR_\Z)} \ .
\end{equation}
\end{lem}
\begin{proof}
We consider the complex case. The arithmetic case is similar.
Since
$H(\DR_{Mf,\C}) $ is homotopy invariant, we have a chain of equivalences
 $$ H(\DR_{Mf,\C}) \overset{ {\eqref{eq:relA9}}}{\simeq} \underline{e^{*}H(\DR_{Mf,\C})} \simeq \underline{H(\DR_\C)}$$ 
 {where $e$ is as in \eqref{jul1002}.}
\end{proof}

Using \eqref{nov2604} we get the following consequence.
\begin{kor}\label{jdlkjkljwedewdwedwedewdewdewdewded}
We  have isomorphisms $$H^{*}(\DR_{Mf,\C}(M\times X))\cong H(\DR_\C(X))^{*}(M)\ , \quad X\in \Sm_\C$$ and
$$H^{*}(\DR_{Mf,\Z}(M\times X))\cong H(\DR_\Z(X))^{*}(M)\ ,\quad X\in \Reg_\Z.$$
\end{kor}

In other words, $H^{*}(\DR_{Mf,\C}(M\times X))$ is the cohomology of $M$ with coefficients in the absolute Hodge cohomology of $X$.

 \subsection{Vector bundles and geometries}
 \label{nov1101}

We consider a product $M\times X$ of a   real manifold $M$ and a  complex manifold $X$. The manifold $M$ may have corners with distinguished germs  of product coordinates.  
The complex manifold $X$ is smooth and without boundary. Both, $M$ and $X$,  are allowed to be non-compact and non-connected.
We let $V\to M\times X$ be a complex vector bundle.
\begin{ddd}\label{dez2601}
A partial geometry on $V$ is a pair $(\nabla^{I},\bar \partial)$ consisting of
\begin{enumerate}
\item a partial connection $\nabla^{I}$ in the $M$-direction with a product structure, and  
\item a holomorphic structure $\bar \partial$ in the $X$-direction which is compatible with the product structure.
\end{enumerate}
\end{ddd}

\begin{ddd}\label{nov1904}
A geometry on the vector bundle $V$ with a partial geometry $(\nabla^{I},\bar \partial)$ is a pair $g^{V}:=(h^{V},\nabla^{II})$ consisting of
\begin{enumerate}
\item a  partial connection $\nabla^{II}$ in the $X$-direction which extends the holomorphic structure $\bar \partial$ and is compatible with the product structure,  and
\item   a hermitian metric $h^{V}$ on the bundle $V\to M\times X$ which is compatible with the product structure.
\end{enumerate}
\end{ddd}

\begin{ex}\label{dqwkjdhqwdqwdqdq}
We consider the one-dimensional trivial  complex vector bundle  $V$ on $\R\times \C$. We discuss the following  list of examples of  partial geometries:
\begin{enumerate}
\item 
  $(d,\bar \partial)$  
  \item $(d+ zdt,\bar \partial)$
  \item  $(d+\bar zdt,\bar \partial)$
  \item   $(d ,\bar \partial+td\bar z)$
  \end{enumerate}
 Here $(t,z)$ are the coordinates of $\R\times \C$, $d$ is the de Rham differential in the $\R$-direction, and $\bar \partial$
 acts in the $\C$-direction. In the first three cases we can take the geometry
 $(h^{V},\bar \partial+\partial)$ where $h^{V}$ is the standard metric of the trivial line bundle and $\partial$ acts in the $\C$-direction. In the third case $(h^{V},\bar \partial+\partial+td\bar z)$ is a geometry.
\end{ex}

\begin{rem}
Let us comment on the product structure. If $M=[0,\infty)^{n}\times N$ models a corner, then a product structure consists of a bundle $\tilde V\to N\times X$ with partial connection $\tilde \nabla^{I}$  and an isomorphism of 
$(V,\nabla^{I})$ with the pull-back of $(\tilde V,\tilde \nabla^{I})$ along the projection $M\times X\to N\times X$. The geometry $g$   and the holomorphic structure $\bar \partial$ are compatible with the product structure if they are  obtained via this isomorphism from a geometry $\tilde g$ and a holomorphic structure $\bar \partial$  on $\tilde V$. In general, the compatibility with the product structure is 
required locally at the corners. Roughly speaking, compatibility with the product structure requires that the geometry  and holomorphic structure do not depend on the normal coordinates of the corners if the bundles are trivialized along the normal directions using $\nabla^{I}$.
\end{rem}

We now assume that $X$ is a smooth 
variety  over $\C$ and consider a sheaf $V$ of locally free  and
finitely generated $\pr^{*}_{X} \cO_{X}$-modules on $M\times X$. Here  $\pr_X^*$ denotes the inverse image under the projection to $X$ in the sense of sheaves of sets. We use the same symbol $V$ in order to denote the corresponding complex vector bundle over $M\times X$.  It has an induced  partial geometry  $(\nabla^{I}, \bar \partial)$   which in addition satisfies:
\begin{ass}\label{dlqkwdqwdqwdwqd}
\label{dez2402}\mbox{}
\begin{enumerate}
\item   The partial connection $\nabla^{I}$ is flat.
\item  The  holomorphic structure $\bar \partial$ is constant w.r.t. $\nabla^{I}$, i.e.~$[\nabla^{I},\bar \partial]=0$.
\end{enumerate}\end{ass}
The original locally free sheaf of $\pr^{*}_{X} \cO_{X}$-modules can be recovered as  the sheaf of sections of $V$ annihilated by both, the  {partial} connection $\nabla^{I}$ and the holomorphic structure $\bar \partial$.

\begin{ex}
We continue Example \ref{dqwkjdhqwdqwdqdq}.
The first two examples, 
  $(d,\bar\partial)$  and $(d+zdt,\bar \partial)$, satisfy Assumption  \ref{dlqkwdqwdqwdwqd}.  
 Indeed, 
 the partial connection $d+zdt$ is flat, and  $[d+zdt,\bar \partial]=0$ since $\bar \partial z=0$.
 The  point is that $d+zdt$ is a holomorphic family of flat partial connections along the $\R$-direction.
 The third example, 
 $(d+\bar zdt,\bar \partial)$, does not satisfy the assumption
 since $[d+\bar zdt,\bar \partial]=-dt$. Indeed, $d+\bar zdt$ is a non-constant anti-holomorphic family of flat partial connections.
The last example, $(d,\bar \partial+tdz)$, also does not satisfy the assumption since
$[d,\bar \partial +td\bar z]=d\bar z$. This expresses the fact that the holomorphic structure is not parallel along the $\R$-direction.
\end{ex}

In Definition \ref{jan0210eee} we will introduce characteristic forms associated to geometries.
If  $X$ is  {not proper over $\C$}  we want  these characteristic forms to belong to the logarithmic
subcomplex $A_{\log,Mf}(M\times X)$ with controlled weights. To this end
we define a subset of geometries on $V$, called good geometries,  which behave in a controlled way at infinity.  The definition of the notion of a good geometry, the existence of local geometries explained below, and the construction of the canonical interpolation in \ref{may1701} restate corresponding  {ideas and} constructions of Burgos and Wang in \cite[\S\S 2,3]{MR1621424} in the present setting. 
\begin{ddd}\label{jan2702}
 A geometry $g$ on $V$ is called good, if every $m\in M$ has a neighbourhood $U\subseteq M$
 such that there exist the following data:
\begin{enumerate}
 \item  a good compactification
 $X\hookrightarrow \overline{ X}$ (see Subsection \ref{sec:AbsHodgeComplex}),
\item a locally free sheaf of $\pr^{*}\cO_{\overline{ X}}$-modules 
 $\overline V$ on   $U\times \overline{X}$, where
 $\pr: U\times \overline{ X}\to \overline{ X}$ is the projection,
\item \label{dez2401}  an isomorphism of $\pr_{|U\times X}^*\cO_X$-modules $ {\phi\colon} V_{|U\times X} \xrightarrow{\cong} \overline{V}_{|U \times X}$,
\item a geometry $\overline g$ on $\overline{V}$ in the sense of Definition \ref{nov1904} {such that $\phi$ is compatible with the geometries}
  (note that $\overline V$ has a natural flat partial connection  in the manifold direction and a holomorphic structure along $\overline{ X}$).
  \end{enumerate}
 \end{ddd}  
 In the situation of 1.--3. of the definition, we say that $\overline V$ is a compactification of $V_{|U\times X}$.

If $X$ is proper, then every geometry is good. In this case we can glue good geometries using a partition of unity {on $M$}. Since good geometries exist locally {on} $M$  (see below) we conclude that good geometries exist.

If $X$ is not proper, then   we do not know whether a general sheaf $V$  over $M\times X$ admits a  good geometry globally on $M$, see Example \ref{jhclkecwecwecwcc345345}. Therefore, we introduce
the notion of a local geometry. A germ of a good geometry on $V$ at a point $m\in M$ is represented by pair $(U,g)$ where $U$ is a neighbourhood of $m$ and $g$ is a good geometry on $V_{|U\times X}$, and we identify two such representatives if they coincide after restriction to a joint smaller neighbourhood of $m$.

Germs of good geometries exist at every point $m\in M$. Indeed, we can take a simply connected neighbourhood $U\subseteq M$ of $m$. 
Then we can identify $V_{|U\times X}$ with the pull-back of a bundle $\tilde V$ on $X$ along the projection $U\times X\to X$. 
By \cite[Proposition 2.2]{MR1621424} there exist a good compactification $X\hookrightarrow \overline{X}$ and an extension $\overline{\tilde V}\to \overline{X}$.  
We choose a geometry $\overline{\tilde g}$ on $\overline{\tilde V}$. Then we get a good geometry on $V_{|U\times X}$ by restricting $\overline{\tilde g}$ to $X$ and pulling it back to $U\times X$.
\begin{ddd}\label{dez2403}
A local geometry on $V$ is a family $(g_{m})_{m\in M}$ of germs of good geometries. The pair $(V,(g_{m})_{m\in M})$ will be called a   geometric bundle.  An isomorphism of geometric bundles is an isomorphism of bundles which  preserves the local geometry.
\end{ddd}
If $g$ is a good geometry on $V$, then we can set $g_{m}:=[ {M},g]$ for every $m\in M$ and thus obtain a geometric bundle.

\begin{ex}\label{jhclkecwecwecwcc345345} 
We consider $ {\bbG_{m,\C}}=\Spec(\C[\lambda,\lambda^{-1}])\in \Sm_{\C}$.
On $S^{1}\times \bbG_{m,\C} $ we consider the bundle $V$ which is trivial along the fibres of $S^{1}\times \bbG_{m,\C} \to S^{1}$ and has holonomy
$\lambda$ along $S^{1}$. Locally on $S^{1}$, we can trivialize $V$ and choose the good compactification by the trivial bundle over $ {\P^{1}_{\C}}$ (see also Example \ref{jqwhdkjqwhdwqkdhkwqdwqdwqdwqd},1.). This does not work globally on $S^{1}$ since the multiplication by $\lambda$ does not extend to the compactification. This is the typical problem which occurs when we want to construct good geometries globally. 
\end{ex}

For an object $M\times X\in \Mf\times \Sm_{\C}$, we let 
$\Vect_{Mf,\C}(M\times X)$ be  the category  of locally free and finitely generated  sheaves of $\pr_{X}^{*} \cO_{X}$-modules  on $M\times X$.   It is    symmetric monoidal with respect to the direct sum.  
From now on, the objects of $\Vect_{Mf,\C}(M\times X)$ will be called  bundles.   We let 
\begin{equation}\label{jul1401}
i\Vect_{Mf,\C}(M\times X)\subseteq \Vect_{Mf,\C}(M\times X)
\end{equation} 
be the maximal subgroupoid. Bundles can be pulled back along morphisms $M^{\prime}\times X^{\prime}\to M\times X$ in the site $\Mf\times \Sm_{\C}$. In fact, $i\Vect_{Mf,\C}$  is a symmetric monoidal stack
since bundles can be glued from local data given on covering families of the site $ \Mf\times \Sm_{\C}$.  Since bundles are locally constant in the manifold direction, this  stack is  homotopy invariant.
It can therefore be interpreted   as an object 
\begin{equation}\label{gdhjqwgdjhqwgdjqwhdgwqdwqhjdgwqdqwdqwdqwdqwd}
i\Vect_{Mf,\C}\in \Fun^{desc,h}(\bS_{Mf,\C},\CommMon(\Nerve(\Cat)[W^{-1}]))\ .
\end{equation}

We now turn to geometric bundles. 
For an object $M\times X\in  \Mf\times \Sm_{\C}$ we let 
$i\Vect_{Mf,\C}^{geom}(M\times X)$ denote the symmetric monoidal groupoid (again with respect to the direct sum) of geometric bundles  as in Definition \ref{dez2403} and geometry preserving isomorphisms. Similarly as above, we get a symmetric monoidal stack $i\Vect_{Mf,\C}^{geom}$. It    can be considered as an object  
  $$i\Vect_{Mf,\C}^{geom}\in \Fun^{desc}(\bS_{Mf,\C},\CommMon(\Nerve(\Cat)[W^{-1}]))$$ 
  which in this case is not homotopy invariant.
  We have a morphism  
\begin{equation}\label{jul0801}
i\Vect_{Mf,\C}^{geom}\to i\Vect_{Mf,\C}
\end{equation}
which maps a geometric bundle to its underlying bundle.
 
 \begin{rem}
Let $\cC$ 
 be any presentable $\infty$-category. Then the inclusion of  {homotopy invariant} into all sheaves with values in $\cC$ has a left adjoint $\cH$ called homotopification, i.e.~we have an adjunction 
$$
\cH:\Fun^{desc}(\bS_{Mf,\C},\cC)\leftrightarrows \Fun^{desc,h}(\bS_{Mf,\C},\cC) \ 
$$
 (see also Subsection \ref{mar0804}).
We expect that the morphism \eqref{jul0801} becomes an equivalence after homotopification. This should reflect the fact that the space of geometries on a bundle is contractible. In the present paper  we prove a corresponding statement in Lemma \ref{jul0803} where  we will not use $\cH$    but a concrete approximation $\bar \bs$  to the homotopification  defined in  \eqref{jul0741bn}.   We refer to \cite{Bunke:2013aa} for a detailed discussion of the relation between $\bar \bs$ and $\cH$.
\end{rem} 
 
 We consider a presentable $\infty$-category $\cC$ (e.g.~$\cC=\CommMon(\Nerve(\sSet)[W^{-1}])$). Let $\Delta^{\bullet}\in \Mf^{\Delta}$ be the cosimplicial manifold given by the standard simplices. It gives  {rise to} a functor  {$\bS_{Mf,\C} \times \Nerve(\Delta ) \to \bS_{Mf,\C}$ induced by $(M,[q]) \mapsto   \Delta^q\times M$}. Pull-back along this functor defines the functor
\begin{equation}\label{jul0802bn}
\bs:\Fun(\bS_{Mf,\C},\cC)\to \Fun(\bS_{Mf,\C}\times \Nerve(\Delta ),\cC)\simeq \Fun(\Nerve(\Delta),\Fun(\bS_{Mf,\C},\cC))\ .
\end{equation}
We define the endofunctor  {$\bar\bs$ by}
\begin{equation}\label{jul0741bn}
\bar \bs:=\colim_{\Nerve(\Delta^{op})}\circ \bs:\Fun(\bS_{Mf,\C},\cC)\to \Fun(\bS_{Mf,\C},\cC)\ .\end{equation}
 {The} projection $\Delta^{\bullet}\to *$ to the constant cosimplicial manifold given by the point
induces the transformation 
\begin{equation}\label{jul0811bn}\id\to \bar \bs\ .\end{equation}


In the following, we use the nerve functor
$$
 {\Nerve\colon} \CommMon(\Nerve(\Cat)[W^{-1}])\to\CommMon(\Nerve(\sSet)[W^{-1}])\ .
$$
\begin{lem}\label{jul0803}
The morphism \eqref{jul0801} induces an  equivalence
$$
\bar \bs (\Nerve(i\Vect_{Mf,\C}^{geom})) \xrightarrow{ {\simeq}} \bar \bs( \Nerve(i\Vect_{Mf,\C}))
$$
in $\Fun(\bS_{Mf,\C},\CommMon(\Nerve(\sSet)[W^{-1}]))$.
\end{lem}
\begin{proof}
The forgetful functor
 $\CommMon(\Nerve(\sSet)[W^{-1}])\to\Nerve(\sSet)[W^{-1}]$ reflects equivalences.
If we forget the symmetric monoidal structure and descent properties, then we can assume that
 $i\Vect_{Mf,\C}$ is realized by a strict functor $i\Vect_{Mf,\C}\in \Fun(\bS_{Mf,\C},\Cat)$.
To every object  $M\times X\in  \Mf\times \Sm_{\C}$  we can thus functorially associate a simplicial set 
  $\Nerve(i\Vect_{Mf,\C})(M\times X)\in \sSet$ and the set   
  $\Nerve_{q}(i\Vect_{Mf,\C})(M\times X)$   of  its $q$-simplices.

We show that, for every object $M\times X\in \bS_{Mf,\C}$ and every $q\in \nat$, the map of simplicial sets
 $$ 
\bs(\Nerve_{q}(i\Vect^{geom}_{Mf,\C}))(M\times X) \to \bs(\Nerve_{q}(i\Vect_{Mf,\C}))(M\times X), 
$$
 {with $\bs$ as in \eqref{jul0802bn},}
is a trivial Kan fibration. Then the induced morphism of colimits over $\Delta^{op}$ is an equivalence in $\Nerve(\sSet)[W^{-1}]$.
This implies the assertion.

Recall that trivial Kan fibrations are characterized by the right lifting property with respect to the inclusions $\partial\Delta^{p} \hookrightarrow \Delta^{p}$ for all $p\ge 0$.
 Let us consider a $p$-simplex $$x\colon \Delta^{p}\to \bs(\Nerve_{q}(i\Vect_{Mf,\C}))(M\times X)$$ and a lift $$\tilde y\colon\partial \Delta^{p} \to \bs(\Nerve_{q}(i\Vect^{geom}_{Mf,\C}))(M\times X)$$ of $x_{|\partial \Delta^{p} }$.
 Explicitly, $x$ is a sequence of isomorphisms
 $V_{0}\xrightarrow{ {\cong}} \dots \xrightarrow{ {\cong}} V_{q}$ of bundles on $\Delta^{p}\times M\times X$. 
  The lift $\tilde y$ is given by a collection of local   geometries  $(g(i)_{(u,m)})_{(u,m)\in \partial_{i} \Delta^{p}\times M}$   on $(V_{q})_{|\partial_{i}\Delta^{p}\times M\times X}$ for $i=0,\dots,p$ which are isomorphic at the corners of codimension two of the simplex $\Delta^{p}$. The germ $g(i)_{(u,m)}$ extends uniquely to a germ
of a {good} geometry on $V_{q}$ at the point $(\partial_{i}(u),m)\in \Delta^{p}\times M$ {that} is compatible with the product structure, where $\partial_{i}\colon\Delta^{p-1}\to \Delta^{p}$ is the inclusion of the $i$-th face. Using the compatibility of the $g(i)$ at the corners of codimension two, we obtain
well-defined germs of good geometries on $V_{q}$ at all points of $\partial \Delta^{p}\times M$.
We can extend this to a local geometry on $V_{q}$ by choosing germs of good geometries at all points
$(u,m)\in (\Delta^{p} \setminus \partial \Delta^{p})\times M$. In this way, we get a lift $\tilde x:\Delta^{p} \to \bs(\Nerve_{q}(i\Vect^{geom}_{Mf,\C}))(M\times X)$ of $x$.
\end{proof}

We consider a bundle $V$ on $M\times X$ and assume that we have a family of good geometries
$(g_{i})_{i=0}^{n}$ on $V$.  In the following we describe the construction of the canonical interpolation between these geometries.

 We let $[x_{0}:\dots:x_{n}]$ be homogeneous coordinates on $ {\P^{n}_{\C}}$
which will be considered as sections $x_{i}\in 
  \cO_{{\P^{n}_{\C}}}(1)( {\P^{n}_{\C}})$.
   We define an embedding
$$
\pr^{*}V  \hookrightarrow \bigoplus_{i=0}^{n}  \pr_{ {\P^{n}_{\C}}}^{*}\cO_{ {\P^{n}_{\C}}}(1) 
\otimes_{ {\C}} \pr^*V=:W\ , \quad s\mapsto \oplus_{i=0}^{n}  x_{i}\otimes s \ ,
$$
where $\pr :M \times X\times  {\P^{n}_{\C}}\to M\times X$ and $\pr_{{\P^{n}_{\C}}}:M\times X\times   {\P^{n}_{\C}}
\to  {\P^{n}_{\C}}$ are the projections.
Let $H_{ {\P^{n}_{\C}}}(1)\to  {\P^{n}_{\C}}$ denote the holomorphic   vector bundle  whose sheaf of holomorphic sections is $\cO_{ {\P^{n}_{\C}}}(1)$.
We use the standard metric $h^{H_{ {\P^{n}_{\C}}}(1)}$ and connection on
$H_{ {\P^{n}_{\C}}}(1)$ and the geometries $  \pr^{*}g_{i}$ in order to define a geometry
$ g_{ W}$ on $ W$. 
We obtain a metric $  h$ on $  \pr^{*}V$ by restricting the metric of $  W$. Using the projection $  W \to  \pr^{*} V$ given by the metric on $  W$ we obtain a partial connection $  \nabla^{II }$ on $ {\pr^*}V$. Hence we have a geometry $ g := (  h,  \nabla^{II })$ on $\pr^{*}V$.
\begin{ddd}\label{may1701}
The geometry $g$ on $\pr^{*}V$ over $M\times X\times  {\P^{n}_{\C}}$  {constructed above} is called the canonical interpolation of the family $(g_{i})_{i=0}^{n}$.
\end{ddd}
\begin{lem}\label{may1720}
The canonical interpolation of the family $(g_{i})_{i=0}^{n}$ is good.
\end{lem}
\proof
%
%

Goodness can be checked locally in $M$. We can thus assume that $M$ is simply connected. For every $i\in \{0,\dots,n\}$ the geometry 
  $g_{i} $  extends to a compactification $\overline V_{i}$ with geometry $\overline g_{i}$ on $M\times \overline{X}_{i}$.
We can 
  choose a joint good compactification $  \overline{X\times {\P^{n}_{\C}}} $ mapping to the compactifications
$    \overline{X}_{i} \times  {\P^{n}_{\C}}$ for all $i\in \{0,\dots,n\}$. 
 
The compactification $\overline V_{i}$ with geometry $\overline g_{i}$  {can be pulled back} to $M\times \overline {  X\times  {\P^{n}_{\C}}}$,  and  {this pull-back} will be denoted  by  $\hat V_{i}$ and $\hat g_{i}$.
 Since $M$ is simply connected,
the compactification
$\overline W:=\bigoplus_{i=0}^{n}  \pr_{ {\P^{n}_{\C}}}^{*}\cO_{ {\P^{n}_{\C}}}(1
)\otimes \hat V_{i}$ of  $W$
determines by \cite[Thm 2.4]{MR1621424},  {after possibly replacing $\overline{X\times  {\P^{n}_{\C}}}$ by another good compactification   mapping to $\overline{X\times  {\P^{n}_{\C}}}$}, a compactification
$\overline V$  {of $\pr^*V$} 
such that
$\overline V\hookrightarrow \overline W$ is a subbundle.
The geometries $\hat g_{i}$ together with the geometry on $H_{ {\P^{n}_{\C}}}(1)$ induce a geometry on $\overline W$.
 We obtain a metric $\overline h$ on $\overline V$ {that extends the metric $h$ of the canonical interpolation on $\pr^*V$} by restricting the metric of $\overline W$. 
 Using the projection $\overline W \to \overline V$ given by the metric on $\overline W$ we obtain a partial connection $\overline \nabla^{II}$ on $\overline V$ which extends $\nabla^{II}$. \hB  

For later use  {we record} the following fact.
\begin{fact}{\rm \label{hdwjhdkjqwhdwqdwqdwqd}
For $j\in \{0,\dots,n\}$ let
  $f_{j}:  {\P^{n-1}_{\C}} \to \ {\P^{n}_{\C}}$ denote the canonical embedding  of  the subvariety $\{x_{j}=0\}$. Then   the geometry 
 $f_{j}^{*}g$ is, by construction, the canonical interpolation of the family  $(g_{i})_{i=0, i\not=j}^{n}$. }
\end{fact}

 \subsection{Characteristic forms}\label{dez2501}

We   consider a smooth  {manifold $M$, a}  complex manifold $X$, and a 
complex vector bundle $V\to M\times X$ with partial geometry $(\nabla^{I},\bar \partial) $ (see Definition \ref{dez2601}). 

The choice of a geometry $(\nabla^{II},h^{V})$ allows us to define the connection
$\nabla^{V}:=\nabla^{I}+\nabla^{II}$,
its adjoint $\nabla^{V,*}$ with respect to $h^{V}$, and the unitarization
\begin{equation}\label{dez2602}
\nabla^{V,u}:=\frac{1}{2}(\nabla^{V}+\nabla^{V,*})\ .
\end{equation}
The component in degree $2p$ of the  unnormalized Chern form of the unitary connection $\nabla^{V,u}$ satisfies
\begin{equation}\label{2w212w2wfwefwefewfw12}
\ch_{2p}(\nabla^{V,u}):=\left[\Tr \exp(-R^{\nabla^{V,u}})\right]_{2p} \in (2\pi i)^{p}A_{\R}^{2p}(M\times X)\ .
\end{equation}

If the partial geometry $(\nabla^{I},\bar\partial)$   satisfies Assumption \ref{dez2402},  then
we have
$$R^{\nabla^{V}}\in \cF^{1}A^{2}(M\times X,\End(V))$$ so that
$$\ch_{2p}(\nabla^{V})\in  \cF^{p}A^{2p}(M\times X)\ .$$
Finally, we consider the transgression (see \eqref{dez3103} below for a definition)
$$\tilde \ch_{2p-1}(\nabla^{V,u},\nabla^{V})\in A^{2p-1}(M\times X),$$
which satisfies
$$d\tilde \ch_{2p-1}(\nabla^{V,u},\nabla^{V})=\ch_{2p}(\nabla^{V,u})-\ch_{2p}(\nabla^{V})\ .$$

We now assume that $X$ is a smooth  variety over $\C$ and that $g$ is a good geometry on the bundle $V$ in the sense of  Definition \ref{jan2702}.  Then Assumption \ref{dez2402} is fulfilled. In addition we shall see  that the Chern forms  belong to the subcomplex $A_{\log, Mf}(M\times X)$ with the correct weights. 
This can be checked locally in $M$. Thus we can assume that
the geometry $g$ is obtained from a geometry $\bar g$ on a bundle $\overline V$ over a compactification 
 $M\times X\hookrightarrow M\times  \overline X$. Since the Chern forms are natural, we conclude that $\ch_{2p}(\nabla^{V})$, $\ch_{2p}(\nabla^{V,u})$ and $\tilde \ch_{2p-1}(\nabla^{V,u},\nabla^{V})$ extend smoothly to $M\times \overline{X}$ and consequently  belong to 
 $\cW_{0}A^{*}_{\log,Mf}(M\times X)$.
 Since Chern forms are closed we  see {in addition} that
\begin{equation}\label{dez3102}
\ch_{2p}(\nabla^{V}),\:\ch_{2p}(\nabla^{V,u})\in \hat \cW_{2p} A_{\log,Mf}^{2p}(M\times X)\ .
\end{equation} 
Similarly, if $I$ is the unit interval with coordinate $t$, then we have 
\begin{equation}\label{frfrfrfrfrf98987937434jk} \ch_{2p}({(1-t)}\pr_{M\times X}^{*}\nabla^V + {t}\:\pr_{M\times X}^{*}\nabla^{V,u}) \in \hat\cW_{2p}A^{2p}_{\log,Mf}(I\times M \times X)\ .\end{equation} 
Since by the proof of Lemma \ref{dez1405} integration along $I$ preserves the subcomplex of logarithmic forms and the d\'ecalage $\hat\cW$ of the weight filtration, we conclude that
\begin{equation}\label{dez3103}
\tilde \ch_{2p-1}(\nabla^{{V},u},\nabla^{{V}}):=\int_{I}\eqref{frfrfrfrfrf98987937434jk}\in \hat \cW_{2p}A_{\log,Mf}^{2p-1}(M\times X)\ .
\end{equation}

Recall the Definition \ref{nov2031e} of the cone $\DR_{Mf,\C}(p)$ for $p\in \Z$ and the Notation \ref{nota:cone} for its elements.

\begin{ddd}\label{jan0210eee}
We define 
the characteristic form  of the bundle  $V$ with a good geometry $g $ by
\begin{equation}\label{nov1102}
\omega(p)(g ):=\left( \ch_{2p}(\nabla^{V,u})\oplus\ch_{2p}(\nabla^{V}),\tilde \ch_{2p-1}(\nabla^{V,u},\nabla^V)\right)\in  {Z^{0}(\DR_{Mf,\C}(p)(M\times X))}\ .
\end{equation} 
We further define
\begin{equation}\label{jan0403}
\omega(g ):=\prod_{p\ge 0}\omega(p)(g )\in   {Z^{0}(\DR_{Mf,\C}(M\times X))}\ . 
\end{equation}
\end{ddd}

\begin{ex} In this example we consider $X=\Spec(\C)$ in $\Sm_{\C}$ and a manifold $M\in \Mf$.
A bundle over $M\times X$ is a complex vector bundle   $V\to M$ with a flat connection $\nabla^{I}$.
A good geometry $g$ on  $V$ is just the choice of a hermitian metric $h^{V}$. 
In this case we have for $p\ge 1$
$$\omega(p)(g)=(\ch_{2p}(\nabla^{V,u})\oplus 0,\tilde \ch_{2p-1}(\nabla^{V,u},\nabla^{I}))\ .$$
The equation $d\omega(p)(g)=0$ is equivalent to the pair of equations
$$d\ch_{2p}(\nabla^{V,u})=0\ , \quad d \tilde \ch_{2p-1}(\nabla^{V,u},\nabla^{I})=\ch_{2p}(\nabla^{V,u})\ .$$
We have an isomorphism
\begin{equation}\label{hgdhjqgjhdqwdwqdwqdwqdwqd}
H^{0}( \DR_{Mf,\C}(p)(M\times X))\cong  H^{2p-1}(M;\R)\ , \quad [\alpha\oplus0,\gamma]\mapsto [\Ree(i^{p+1}\gamma)] \ .
\end{equation}
Since $\Ree(i^{p+1}\ch_{2p}(\nabla^{V,u}))=0$, we indeed have
$d\Ree(i^{p+1}\tilde \ch_{2p-1}(\nabla^{V,u},\nabla^{I}))=0$.
Therefore  the class of $\omega(p)(g)$ is mapped  {under \eqref{hgdhjqgjhdqwdwqdwqdwqdwqd}} to
$[\Ree(i^{p+1}\tilde \ch_{2p-1}(\nabla^{V,u},\nabla^{I}))]\in H^{2p-1}(M;\R)$. This is, up to normalization, exactly the 
characteristic class of the flat vector bundle $(V,\nabla^{I})$ considered in \cite[(0.2)]{MR1303026}. 
\end{ex}

Our next task is to extend the definition of the characteristic forms to geometric bundles.
A geometric bundle comes with a family of germs of good geometries $(g_{m})_{m\in M}$.
This gives a family of germs $\omega(g_{m})$ of characteristic forms. The idea is to use the canonical interpolation of geometries {from} Definition \ref{may1701} and some homotopies, to be described below, in order to extend this family to a  {zero} cycle  in the  \v{C}echification $\cL \DR_{Mf,\C}(M\times X)$ of the de Rham complex.
Here $\cL \DR_{Mf,\C}(M\times X)$ is the colimit of the  \v{C}ech complexes  of   $\DR_{Mf,\C}$ over the poset   of  coverings  of $M\times X$ of the form $(U_{m}\times X)_{m\in M}$  which are   indexed by the points of $M$ {and} such that $m\in U_{m}$. {We refer to Subsection \ref{may2201} for more details.}

We let $Q\colon A(M\times X\times \CPn)\to A(M\times X\times \CPn)$ be the projection onto the part which is harmonic in the last variable. 
Note that the harmonic forms on $\CPn$ are given by multiples of powers of the K\"ahler form $\omega \in {A}_{\R}^{2}(\CPn)$ which we normalize  such that $\int_{\CPn}\omega^{n}=1$. Explicitly we have
$$Q(\beta)=\sum_{i=0}^{n} \pr^{*}_{M\times X}\left(\int_{M\times X\times \CPn/M\times X} \beta\wedge \pr_{\CPn}^{*}\omega^{i} \right)\wedge \pr^{*}_{\CPn}\omega^{n-i}\ .$$
It easily follows from this formula that $Q$ preserves the Hodge filtration, the weight filtration, and  real forms. It therefore induces a projection operator on the level of  mapping cones
\begin{equation}\label{may2702n}
Q\colon \DR_{Mf,\C}(M\times X\times \CPn)\to  \DR_{Mf,\C}(M\times X\times \CPn).
\end{equation}
 {It} is natural in $M\times X$. 

\begin{lem}\label{fewefweweffwef5443534534234234324}
There exists a homotopy $$h:\DR_{Mf,\C}(M\times X\times \CPn)\to  \DR_{Mf,\C}(M\times X\times \CPn)[-1]$$ which is natural in $M\times X$ {and} such that
\begin{equation}\label{may1721}
dh+hd=\id-Q\ , \quad hQ=0\ .
\end{equation} 
\end{lem}
\begin{proof}
We fix $p\in \Z$ and consider the component
of $h$ for $\DR_{Mf,\C}(p)$, which we will also denote by $h$.
Assume that we have natural transformations
\begin{equation}\label{may1710}h_{\R},h_{\cF} ,\tilde h :\hat \cW_{p}  A_{\log}(M\times X\times\CPn)\to \hat \cW_{p}A_{\log}(M\times X\times \CPn)[-1]\end{equation}
and 
\begin{equation}\label{may1711}r_{\R},r_{\cF} :\hat \cW_{p}A_{\log}(M\times X\times\CPn)\to \hat \cW_{p}A_{\log}(M\times X\times \CPn)[-2]\end{equation}
such that 
 $$dh_{\R}+h_{\R}d=  {\id-Q} =dh_{\cF}+h_{\cF}d \ ,\quad d\tilde h+\tilde hd=\id-Q$$
 $$dr_{\R}-r_{\R}d=h_{\R}  {-}\tilde h\ ,\quad dr_{\cF}-r_{\cF}d=h_{\cF}  {-}\tilde h\ ,$$
 and
$h_{\R}$ preserves $A_{\log,\R}$ and $h_{\cF}$ preserves $\cF^{p}A_{\log}$.
Then we define
$$  h :\DR_{Mf,\C}{(p)}(M\times X\times \CPn)\to \DR_{Mf,\C}{(p)}(M\times X\times \CPn)[-1],$$
using the notation \eqref{wefwefewfewfefewf89798234234234234}, by 
$$
h (\omega_{\R} \oplus \omega_{\cF},\tilde \omega):=(h_{\R}\omega_{\R}\oplus h_{\cF}\omega_\cF, {-}\tilde h\tilde \omega +r_{\R}\omega_{\R}-r_{\cF}\omega_{\cF})\ .
$$
We check:
\begin{eqnarray*}
d h  (\omega_{\R} \oplus \omega_{\cF},\tilde \omega)&=&(dh_{\R}\omega_{\R} \oplus dh_{\cF}\omega_{\cF},d\tilde h\tilde \omega   {-}dr_{\R}\omega_{\R}  {+}dr_{\cF}\omega_{\cF}   {+}h_{\R}\omega_{\R} {-}h_{\cF}\omega_{\cF})\ .
\end{eqnarray*}
\begin{eqnarray*}
h d(\omega_{\R} \oplus \omega_{\cF},\tilde \omega)&=&h (d\omega_{\R} \oplus d\omega_{\cF}, {-}d\tilde\omega  {+}\omega_{\R}  {-}\omega_{\cF})\\
&=
&(h_{\R}d\omega_{\R} \oplus h_{\cF}d\omega_{\cF},\tilde hd\tilde \omega -\tilde h \omega_{\R}+\tilde h\omega_{\cF}  {+}
r_{\R}d\omega_{\R}  {-}r_{\cF} d\omega_{\cF} )\ .
\end{eqnarray*}
We get
$$\left((d h + h d)(\omega_{\R} \oplus \omega_{\cF},\tilde \omega)\right)_{\R}=dh_{\R}\omega_{\R} {+}h_{\R}d\omega_{\R}=\omega_{\R}-Q(\omega_{\R})\ ,$$
$$\left((d h + h d)(\omega_{\R} \oplus \omega_{\cF},\tilde \omega)\right)_{\cF}=dh_{\cF}\omega_{\cF} {+}h_{\cF}d\omega_{\cF}=\omega_{\cF}-Q(\omega_{\cF})\ ,$$
and
\begin{eqnarray*}
\left((d h + h d)(\omega_{\R} \oplus \omega_{\cF},\tilde \omega)\right)^{ {\sim}} &=&d\tilde h\tilde\omega  {-}dr_{\R}\omega_{\R}  {+}dr_{\cF}\omega_{\cF}  {+}h_{\R}\omega_{\R}  {-}h_{\cF}\omega_{\cF}\\&&+\tilde hd\tilde \omega -\tilde h \omega_{\R}+\tilde h\omega_{\cF}  {+}
r_{\R}d\omega_{\R}  {-}r_{\cF} d\omega_{\cF}\\&=&
\tilde \omega-Q(\tilde \omega).
\end{eqnarray*}

It remains to construct the transformations \eqref{may1710} and \eqref{may1711}. 
{This can be done using} the K\"ahler package. 
Let $Y$ be a compact K\"ahler manifold. Then we consider the operator
$G\colon A(Y)\to A(Y)$ which vanishes on the harmonic forms $\cH:= \ker(\Delta)$ and is $\Delta^{-1}$ on the orthogonal complement $\cH^{\perp}$. We define
\begin{equation}\label{mar2002}
h_\R:=d^{*}G\ ,\quad h_\cF:=2\bar \partial^{*} G\ ,\quad \tilde h:=d^{*}G\ .
\end{equation}
Using the adjoint $L^{*}$ of $L\colon\alpha\mapsto \omega\wedge \alpha$
(with $\omega$ the K\"ahler form) we define
\begin{equation}\label{mar2003}
 {r_{\cF}}:=iL^{*}G, \quad  {r_{\R}:=0}\ .
\end{equation}
{Using} that 
$$
\Delta=dd^{*}+d^{*}d=2(\bar \partial \bar\partial^{*}+\bar \partial^{*}\bar \partial), \quad {[\partial, \bar\partial^*]=0}
$$
{we see that}
$$dh_\R+h_\R d= \id-\pr_\cH\ , \quad dh_\cF+h_\cF d=\id-\pr_\cH\ .$$
Finally, using 
$$ {\tilde h}-h_\cF=(\partial^{*}-\bar \partial^{*})G\ ,\quad        [d,L^{*}]=i\partial^{*}-i\bar \partial^{*}$$ we get
$$[d,r_{ {\cF}}]=  {h_{\cF} - \tilde h} \ .$$
We apply this in the case $Y:=\CPn$.
In general, a form in $A(M\times X \times \CPn)$ can be considered as a form on $M\times X$ with values in $A(\CPn)$. In this way the homotopy operators just introduced for $\CPn$ induce the desired homotopy operators for $M\times X\times \CPn$.  
\end{proof}

We continue to use the notation introduced in Fact \ref{hdwjhdkjqwhdwqdwqdwqd}. We have natural maps $$f_{i}^{*}:  \DR_{Mf,\C}(M\times X\times \CPn)\to  \DR_{Mf,\C}(M\times X\times  {\P^{n-1}_{\C}})$$ for $i=0,\dots,n$, and we set
 $$\delta:=\sum_{i=0}^{n}(-1)^{i}f_{i}^{*}\ .$$
We observe that
$f_{i}^{*}\circ Q$ is independent of $i$.
We define $H_{{0}}:=\id$,  
and then inductively
\begin{equation}\label{may2103}
H_{{n+1}}:=\sum_{i=0}^{n+1} (-1)^{i} H_{{n}}f_{i}^{*}h:\DR_{Mf,\C}(M\times X\times {\P^{n+1}_{\C}})\to  \DR_{Mf,\C}(M\times X)[-n-1]\ ,
\end{equation}
where $h$ is the homotopy as in Lemma \ref{fewefweweffwef5443534534234234324} and $ H_{{n}}f_{i}^{*}h$ is a short-hand notation for
the composition $ H_{{n}}\circ f_{i}^{*}\circ h$.
One checks, using $\delta\circ \delta=0$ and \eqref{may1721}, that  
\begin{equation}\label{may1730}
(-1)^{n-1} dH_{n}+H_{n}d= H_{n-1}\delta\ , \quad H_{n}Q=0
\end{equation} 
for all $n\ge {1}$.

We can now construct the characteristic cocycle $\omega(g)\in Z^{0}(\cL\DR_{Mf,\C}(M\times X))$ associated to a geometric bundle $(V,g)$ on $M\times X$. Recall that $g=(g_{m})_{m\in M}$, where
$g_{m}=[U_{m},g_{(m)}]$ is the germ of a  good geometry represented by a good geometry $g_{(m)}$ on the restriction of $V$ to $U_{m}\times X$, where $U_{m}\subseteq M$
  is a neighbourhood of $m$. We get an open covering $\cU:=(U_{m})_{m\in M}$ of $M$ indexed by the points of $M$.

For $n\in \nat$ and {a} family $(m_{i})_{i=0}^{n}$ of points in $M$ we form $U_{(m_{i})}:=\bigcap_{i=0}^{n}U_{m_{i}}$. 
The geometries $g_{(m_{i})}$ induce a family of good geometries on $U_{(m_{i})}\times X$ by restriction, and we let $ {\widehat{g_{(m_{i})}}}$ denote their canonical interpolation 
(Definition \ref{may1701}).
We consider its characteristic form (Definition \ref{jan0210eee})
$$
\omega( \widehat{g_{(m_{i})}})\in \DR_{Mf,\C}^{{0}}(U_{(m_{i})}\times X\times \CPn)\ .
$$
The collection of these forms for all families $(m_{i})_{i=0}^{n}$ determines 
an element of degree $n$ in the \v{C}ech complex
$$
\omega_{n}\in  \check{C}^{n}(\cU, \DR_{Mf,\C}^{ {0}}( -\times X\times \CPn))\ .
$$
We define 
$$
\underline\omega(g)_{n} :=  H_{{n}} \omega_{n}\in  \check{C}^{n}(\cU, \DR_{Mf,\C}^{ {-n}}(-\times X))\ .
$$
Then  
\begin{equation}\label{may1740}
\underline \omega(g) :=  {(\underline\omega(g)_n)_{n\in \nat}} \in \Tot (\check{C}(\cU, \DR_{Mf,\C}(-\times X)))^{0}
\end{equation}
is a cycle. 
 To see this, we denote the 
coface maps defining the differential of the \v{C}ech complex by $\partial_{i}^{*}$.
Using
$d\omega_{n}=0$ and $\partial_{i}^{*} \omega_{n-1}=f_{i}^{*}\omega_{{n}}$, we get 
for the component of $d^{\Tot}\underline\omega(g)$ in 
$\check{C}^{n}(\cU, \DR_{Mf, \C}^{-n+1}(-\times X))$
\begin{align*}
(d^{\Tot}{\underline\omega}(g))_{ {n}} & = \sum_{i=0}^n (-1)^i\partial_i^*\underline\omega(g)_{n-1} + (-1)^n d\underline\omega(g)_n \\
& =  \sum_{i=0}^n (-1)^i \partial_i^* H_{n-1}\omega_{n-1} + (-1)^ndH_n\omega_n  \\
& = \sum_{i=0}^n (-1)^i H_{n-1} f_i^*\omega_n + (-1)^n dH_n\omega_n \\
& = H_{n-1} \delta \omega_n + (-1)^n dH_n\omega_n \\
& \overset{\eqref{may1730}}{=} H_n d\omega_n \\
& = 0.
\end{align*}
%
 {Recall that, by definition,}
the complex $\cL \DR_{Mf,\C}(M\times X)$ is the colimit of the complexes
$\Tot \check{C}(\cU, \DR_{Mf,\C}(-\times X))$ 
over the poset of open coverings of $M$ which are indexed by the points of $M$.
\begin{ddd}\label{may1760n}
The characteristic cocycle of the geometric bundle $(V,g)$ is the cycle
$$
\omega(g)\in Z^{0}(\cL \DR_{Mf,\C}(M\times X))
$$  
represented by the cycle ${\underline\omega}(g)$ constructed in \eqref{may1740}.
\end{ddd}
It is easy to see that the characteristic cocycle is well-defined, i.e.~independent of the choices of the representatives of the germs $g_{m}$, and gives rise to   a characteristic cocycle
\begin{equation}\label{may1761}
\omega:\pi_{0}(i\Vect_{Mf,\C}^{geom})\to  Z^{0}(\cL \DR_{Mf,\C})
\end{equation}
 {in the sense of Definition \ref{jan1002}.}

\subsection{Bundles and algebraic $K$-theory}\label{jul1060}

In the present subsection we generalize Definition \ref{jul0870}  of the algebraic $K$-theory sheaves to the case of a non-trivial manifold direction.    
Recall  that $i\Vect_{Mf,\C}$ is the symmetric monoidal stack of bundles on the site $\Mf\times \Sm_{\C}$, which we interpret as in 
\eqref{gdhjqwgdjhqwgdjqwhdgwqdwqhjdgwqdqwdqwdqwdqwd}.  In a similar manner,
  we consider the   symmetric monoidal stack    $i\Vect_{Mf,\Z}$ on the site  $\Mf\times {\bReg_\Z}$
    which associates to 
$M\times X$ the symmetric monoidal (with respect to  {direct sum}) groupoid  of locally free and  finitely generated  sheaves of $\pr_{X}^{*} \cO_{X}$-modules  on $M\times X$.  
We again interpret  this stack  as an object 
$$
i\Vect_{Mf,\Z}\in \Fun^{desc,h}(\bS_{Mf,\Z},\CommMon(\Nerve(\Cat)[W^{-1}])).
$$
 {As before,}   $L$ denotes the sheafification functors
\begin{equation}\label{f34f34fk34hfjk34hfi3f34f43f98897}
L:\Fun(\bS_{Mf,\C},\Sp)\leftrightarrows \Fun^{desc}(\bS_{Mf,\C},\Sp)\ , \quad L:\Fun(\bS_{Mf,\Z},\Sp)\leftrightarrows \Fun^{desc}(\bS_{Mf,\Z},\Sp) 
\end{equation} 
and $K$  is the algebraic $K$-theory functor introduced in Definition {\ref{nov1901}}.
\begin{ddd}\label{jul0820}
We define the   sheaves of algebraic $K$-theory spectra 
$$\bK_{Mf,\C} \in \Fun^{desc,h}(\bS_{Mf,\C},\Nerve(\Sp)[W^{-1}])\ , \quad  \bK_{Mf,\Z}\in \Fun^{desc,h}(\bS_{Mf,\Z},\Nerve(\Sp)[W^{-1}])$$
by
$$\bK_{Mf,\C}:=L(K(i\Vect_{Mf,\C})) ,  \quad \bK_{Mf,\Z}:=L(K(i\Vect_{Mf,\Z}))\ .$$
\end{ddd}
\begin{rem}
These sheaves are homotopy invariant in the manifold direction, since the arguments of $K$ are.
This homotopy invariance must not be confused with homotopy invariance in the  {algebraic} direction, which would lead to a version of Weibel's homotopy $K$-theory.
\end{rem}

Note that Definition \ref{jul0820} extends Definition \ref{jul0870} in view of the equivalences
\begin{equation}\label{wdqwdqwdqwdqwdd}
e^{*}\bK_{Mf,\C}\simeq \bK_{\C}\ , \quad e^{*}\bK_{Mf,\Z}\simeq \bK_{\Z}\ ,
\end{equation}
where $e$ is as in 
\eqref{jul1002}.
The units of the sheafification adjunctions  \eqref{f34f34fk34hfjk34hfi3f34f43f98897} provide morphisms
\begin{equation}\label{jul0815}
K(i\Vect_{Mf,\C})\to \bK_{Mf,\C}\ ,\quad  K(i\Vect_{Mf,\Z})\to \bK_{Mf,\Z}
\end{equation}
in $\Fun(\bS_{Mf,\C},\Sp)$, or $\Fun(\bS_{Mf,\Z},\Sp)$, respectively.

For $X$ in $\Sm_\C$ or $\Reg_\Z$, the spectra $\bK_\C(X)$ or $\bK_\Z(X)$ represent generalized cohomology theories.  By Lemma  {\ref{dkjhqwkdqwdwqdqwdwqd}}, 
 homotopy invariance of $\bK_{Mf,\C}$ and $\bK_{Mf,\Z}$, and \eqref{wdqwdqwdqwdqwdd} we have  equivalences
$$
\bK_{Mf,\C} \simeq  {\underline{\bK_\C}}\ , \quad \bK_{Mf,\Z} \simeq  {\underline{\bK_\Z}}\ .$$
By \eqref{nov2604} this immediately implies that, for all $k\in \Z$ and manifolds $M \in \Mf$, we have
$$\pi_k(\bK_{Mf,\C}(M\times X))\cong \bK_\C(X)^{-k}(M)$$
or
\begin{equation}\label{jul1080}\pi_k(\bK_{Mf,\Z}(M\times X))\cong \bK_\Z(X)^{-k}(M) \ ,\end{equation}
respectively.

We now extend the definition of a geometric bundle to the arithmetic case.  A local geometry on a bundle $V\in i\Vect_{Mf,\Z}(M\times X)$ is by definition a local geometry (Definition \ref{dez2403}) on its base change $V\otimes \C$
to 
$$
M\times (X\otimes\C) :=  M\times (X\times_{\Spec(\Z)}\Spec(\C))
$$
which in addition  is invariant under the operation of $\Gal(\C/\R)$
in the second argument.
 In the following we make this precise.
Recall that the underlying complex manifold of $\overline{X\otimes \C}$   is the smooth manifold $X\otimes \C$ with its opposite holomorphic structure. 
For a vector bundle $V$ on $M\times X$  we let
  $\overline{ V\otimes \C}$ be the bundle $V\otimes \C$  equipped with the opposite complex structure.

 Let $(\nabla^{I},\bar\partial)$ be the partial geometry (see Definition \ref{dez2601}) on $V\otimes \C$.  The same data can also be considered as a partial geometry on $\overline{V\otimes \C}$.
Let $g:=(\nabla^{II},h^{V\otimes \C})$ extend the partial geometry to a geometry (Definition \ref{nov1904}).
Then  we define the  hermitian metric on $\overline{ V\otimes \C}$ by
$$h^{\overline{ V\otimes \C}}(\phi,\psi):=\overline{h^{V\otimes \C}(\phi,\psi)}\ .$$
The partial connection $\nabla^{II}=\bar \partial+\nabla^{1,0}$ can be considered as a partial connection on $\overline{V\otimes \C}$ which extends $\bar \partial$. We can therefore  define the conjugated geometry by
$ \bar g:=(\nabla^{II}, h^{\overline{ V\otimes \C}})$.

We now have a pair of canonical isomorphisms $u_{M\times X}:M\times (X\otimes \C)\to M\times \overline{(X\otimes \C)}$ and
$U_{M\times X}:u_{M\times X}^{*} V\otimes \C\to  \overline{V\otimes \C} $.

\begin{ddd}
Let $V$ be a bundle over $M\times X$.
We say that
a geometry $g$ on $V\otimes \C$ is 
$\Gal(\C/\R)$-invariant, if
$$
(u_{X},U_{X})^{*}\bar g =g\ .
$$ 
A local geometry $(g_{m})_{m\in M}$ is called $\Gal(\C/\R)$-invariant, if $g_{m}$ is so for every $m\in M$. 
\end{ddd}

\begin{ddd}\label{jul0850}
A geometric bundle on $M\times X\in \Mf\times \bReg_{\Z}$
is a pair $(V,g)$ of a bundle $V$ on $M\times X$ and
a $\Gal(\C/\R)$-invariant local geometry $g$ on $V\otimes \C$.
\end{ddd}

We let $i\Vect_{Mf,\Z}^{geom}$ be the symmetric monoidal stack
of geometric bundles on $\Mf\times \Reg_{\Z}$, which we consider as an object 
$$i\Vect_{Mf,\Z}^{geom}\in \Fun^{desc}(\bS_{Mf,\Z},\CommMon(\Nerve(\Cat)[W^{-1}]))\ .$$
We have a forgetful transformation
\begin{equation}\label{jul0804}i\Vect_{Mf,\Z}^{geom}\to i\Vect_{Mf,\Z}\end{equation}
and 
the analogue of Lemma \ref{jul0803} with the same proof:
\begin{lem}\label{jul0805}
The transformation  \eqref{jul0804} induces an equivalence
$$\bar \bs (\Nerve(i\Vect_{Mf,\Z}^{geom})) \xrightarrow{ {\simeq}} \bar \bs (\Nerve(i\Vect_{Mf,\Z}))\ .$$
\end{lem}

\subsection{The regulators $\cbeil$ and $\beil$}
\label{sep2601}

In this subsection we apply the machinery of Subsection \ref{mar0802} in order to construct 
the complex and arithmetic versions of the regulators $\cbeil$ and $\beil$.
We use the  characteristic
cocycle $\omega$ defined in 
\eqref{may1761}. According to Definition \ref{jul0810}, we get a regulator
\[
r(\omega):K(i\Vect^{geom}_{Mf,\C})\to H(\cL\DR_{Mf,\C})\ .  
\]
In the following definition we use Lemma \ref{jul0730} saying  that for a homotopy invariant presheaf $F\in \Fun(\bS_{Mf,\C},\Sp)$
the natural morphism $F\to \bar \bs(F)$ is an equivalence.
We further use  the Fact \ref{dkjhwqjkdwqdwqdwqdwqdwqd} that for a sheaf $F\in \Fun^{desc}(\bS_{Mf,\C},\Sp)$ the natural morphism $F\to \cL(F)$ is an equivalence.
\begin{ddd}\label{jul0816}
We define the complex version of the Beilinson regulator  
$$\cbeil:\bK_{Mf,\C}\to H(\DR_{Mf,\C})$$ 
to be the map which is  essentially uniquely determined through the diagram
$$
\xymatrix@C+0.6cm{
K(i\Vect_{Mf,\C})\ar[d]^{\eqref{jul0815}} \ar[r]_{{\text{Lemma \ref{jul0730}}}}^{ {\simeq}}&\bar \bs (K(i\Vect_{Mf,\C})) &\bar \bs( K(i\Vect_{Mf,\C}^{geom}))\ar[l]^{\text{Lemma \ref{jul0803}}}_{ {\simeq}}\ar[r]^{\bar \bs{({r}(\omega))} }&\bar \bs (H(\cL\DR_{Mf,\C}))\\
\bK_{Mf,\C}\ar[rr]^{\cbeil} & & H(\DR_{Mf,\C})\ar[r]_{\text{Fact \ref{dkjhwqjkdwqdwqdwqdwqdwqd}}}^-{{\simeq}} & H(\cL\DR_{Mf,\C}) \ar[u]^{ {\simeq}}_{\text{Lemma \ref{jul0730}}}
}
$$
\end{ddd} 
Here \emph{essentially uniquely} means that the space of all factorizations is contractible.
\begin{proof}[Proof of the assertions made in the definition]
First note that $\spp$ from \eqref{jul0701} and the group completion $\Omega B$ are left adjoints and hence commute with colimits. So Lemma \ref{jul0803} implies that the second arrow in the upper row is an equivalence.
Next, the clockwise composition is a map from the presheaf $K(i\Vect_{Mf,\C})$ to the sheaf  $H(\DR_{Mf,\C})$ which then factors   essentially uniquely through the left vertical map to the sheafification.
\end{proof}
%
 
 We now consider the arithmetic situation.  
\begin{lem}\label{jan0901}
If $(V,g)\in i\Vect_{Mf,\Z}^{geom}(M\times X)$ is a  geometric bundle (Definition \ref{jul0850}), then 
the characteristic form $\omega(g)\in Z^{0}(\cL\DR_{Mf,\C}(M\times X\otimes \C))$ satisfies
$$
\overline{u_{M\times X}^{*} \omega(\bar g)}=\omega(g)\ .
$$ 
In particular, we can interpret
$$
\omega(g)\in Z^{0}(\cL \DR_{Mf,\Z}(M\times X))\ .
$$
\end{lem}
\begin{proof}
Let us first assume that $g$ is a good geometry on $V$ which is $\Gal(\C/\R)$-invariant. Then
we have
$$u_{M\times X}^{-1,*} \omega(g)=\omega (u_{M\times X}^{-1,*}g)=\omega( \bar g)=\overline{\omega(g)} \ .$$
We now use the fact that the homotopies   $H_{i}$ in \eqref{may2103} are $\Gal(\C/\R)$-equivariant in order to extend these formulas to $\Gal(\C/\R)$-invariant local geometries. 
\end{proof}

Using Lemma \ref{jan0901}, we again interpret the characteristic form
$\omega$ as a characteristic cocycle
$$\omega: \pi_{0}(i\Vect_{Mf,\Z}^{geom})\to Z^{0}(\cL\DR_{Mf,\Z})\ .$$
\begin{ddd}\label{jul08161}
We define the arithmetic version of the Beilinson regulator $$\beil:\bK_{Mf,\Z}\to H(\DR_{Mf,\Z})$$ 
to be the map which is  essentially uniquely determined through the diagram \begin{equation}\label{sep1901}
\xymatrix@C+0.6cm{
K(i\Vect_{Mf,\Z})\ar[d]^{\eqref{jul0815}}\ar[r]_{ {\text{Lemma \ref{jul0730}}}}^{ {\simeq}}&\bar \bs (K(i\Vect_{Mf,\Z})) &\bar \bs (K(i\Vect_{Mf,\Z}^{geom}))\ar[l]^{\text{Lemma \ref{jul0805}}}_{ {\simeq}}\ar[r]^{\bar \bs (r(\omega))}&\bar \bs (H(\cL\DR_{Mf,\Z}))\\
\bK_{Mf,\Z}\ar[rr]^{\beil}&&H(\DR_{Mf,\Z})\ar[r]^{ {\simeq}}_{\text{Fact \ref{dkjhwqjkdwqdwqdwqdwqdwqd}}}&H(\cL\DR_{Mf,\Z}). \ar[u]^{ {\simeq}}_{\text{Lemma \ref{jul0730}}}
}
\end{equation}
\end{ddd} 
The justification is the same as in the complex case above.

\subsection{Comparison with other constructions of the regulator} \label{nov2202}

 {We can specialize the regulators defined in \ref{jul0816} and \ref{jul08161} along the morphism of sites
$e:\bS_\C\to \bS_{Mf,\C}$ and $e:\bS_\Z\to \bS_{Mf,\Z}$, respectively. The resulting maps
\begin{equation}\label{jul0860-neu}
\cbeil:\bK_{\C}\to H(\DR_{\C})\ , \quad \beil:\bK_{\Z}\to H(\DR_{\Z})
\end{equation}
will still be denoted by the same symbol.}
%
%
%
%
The goal of the present subsection is to give an argument why  {these} regulators 
coincide  with the ones introduced originally by Beilinson \cite{MR760999, MR862628}. 

We consider the complex case.
We use    
$$ \Nerve(i\Vect_{\C})\in \Fun^{desc}(\bS_{\C},\CommMon(\Nerve(\sSet)[W^{-1}]))$$ 
(see \eqref{rh23kr23r32e3e3er32r32r3r}) and 
 $$\DR_{\C}(p)\ , \DR_{\C}\in \Fun^{desc}(\bS_{\C},\Nerve(\Ch)[W^{-1}])$$ 
(see Definition \ref{apr1701}).
The fact that the set of homomorphisms from an abelian monoid to an abelian group is itself an abelian group generalizes to the fact that for 
$$X\in \Fun(\bS_\C,\CommMon(\Nerve(\sSet)[W^{-1}]))\ , \quad Y\in \Fun(\bS_\C, \CommGroup(\Nerve(\sSet)[W^{-1}])),$$
 {the $\infty$-categorical mapping space $\map(X,Y)$ refines naturally to a mapping object 
$$\underline{\map}(X,Y)\in \CommGroup(\Nerve(\sSet)[W^{-1}])$$
in commutative groups.
}

\begin{ddd} \label{jul2601}
We define  the primitive part of the cohomology of  
$ \Nerve(i\Vect_{\C})$ with coefficients in the sheaf of chain complexes $\DR_{\C}(p)$ by
$$\Prim (\Nerve(i\Vect_{\C}),\DR_{\C}(p)):=\pi_{0}(\underline{\map}( \Nerve(i\Vect_{\C}),\Omega^{\infty} H(\DR_{\C}(p)))\ .$$
 
\end{ddd}

By Definition \ref{jul0870} we can write the algebraic $K$-theory sheaf  $\bK_{\C}\in \Fun^{desc}(\bS_{\C}, \Sp )$
 in the form
$\bK \simeq L(\spp(\Omega B( {\Nerve(i\Vect_{\C})})))$. 
Using the universal properties of sheafification and group completion and the fact that $\Omega^{\infty}H(\DR_{\C}(p))$ is a sheaf of commutative groups,
we obtain
the equivalences
\begin{eqnarray}\label{jul0822}\begin{split}
\Omega^{\infty}\Map(\bK_{\C},H(\DR_{\C}(p)))  &\simeq &\underline{\map}(\Omega B ( \Nerve(i\Vect_{\C})),\Omega^{\infty}H(\DR_{\C}(p)))\\
& \simeq &\underline{\map}( \Nerve(i\Vect_{\C}),\Omega^{\infty}H(\DR_{\C}(p))) \ .
\end{split}
\end{eqnarray}
In view of \eqref{jul0822},  the $p$-component of the regulator
$\cbeil\in \pi_0(\Omega^{\infty}\Map(\bK_{\C},H(\DR_{\C} )))$
is characterized uniquely   by the corresponding primitive cohomology class, which we will denote by 
$c_{\omega(p)}\in \Prim(\Nerve(i\Vect_{\C}),\DR_{\C}(p))$.

The goal of  {comparing} $\cbeil$ with  {Beilinson's} definition is achieved by a calculation
of $\Prim( \Nerve(i\Vect_{\C}),\DR_{\C}(p))$ {in terms of the absolute Hodge cohomology of the simplicial varieties $BGL_{\bullet}(n)$ (see below)} and the identification of $c_{\omega(p)}$ with the {$2p$-}component of the Chern character class. Due to Deligne's computations \cite[9.1.1]{HodgeIII}, the absolute Hodge cohomology of $BGL_{\bullet}(n)$ is well understood  {and basically equals its singular cohomology with real coefficients (see \eqref{apr1902} below).}

The stack of bundles
$i\Vect_{\C}$ on the site $\Sm_{\C}$ has an atlas $$\nat_{0}\to i\Vect_{\C}$$ which represents the vector bundle
$\bigsqcup_{n\in \nat_{0}} \C^{n}\to \nat_{0}$. The atlas gives rise to a groupoid  
$$GL:=\left(\bigsqcup_{n\in \nat_{0}} GL(n)\rightrightarrows \nat_{0}\right)$$
in $\Sm_\C$. We let $BGL_{{\bullet}}:=\Nerve(GL)\in (\Sm_\C)^{\Delta^{op}}$ denote its nerve. It decomposes as the disjoint union of simplicial varieties $\coprod_{n\in\nat_0} BGL_{\bullet}(n)$. In particular, we can consider its absolute Hodge cohomology (see \cite[\S 4]{MR862628})
\[
H^0(\DR_\C(p)(BGL_\bullet)) \cong H^{2p}_\mathrm{Hodge}(BGL_{\bullet}, \R(p)) \cong \prod_{n\in\nat_0} H^{2p}_\mathrm{Hodge}(BGL_{\bullet}(n), \R(p)).
\]
\begin{prop}\label{sep0405}
The primitive cohomology $\Prim(\Nerve( i\Vect_{\C}),\DR_{\C}(p))$ is naturally isomorphic to the one{-}dimensional real subspace 
\[
\R\cdot (s_{p,0},s_{p,1},s_{p,2},\dots) \subseteq \prod_{n\in\nat_0} H^{2p}_\mathrm{Hodge}(BGL_{\bullet}(n), \R(p))
\]
spanned by the degree-$2p$ component of the universal Chern character class in absolute Hodge cohomology (see the proof for notation).
Under this isomorphism $c_{\omega(p)}$ corresponds to the universal Chern character class.
\end{prop}

This proposition implies in particular that on the level of homotopy groups $\cbeil$ and similarly $\beil$ induce Beilinson's regulator introduced originally in \cite[2.3]{MR760999}.

\begin{proof}
Let
$$Y:(\Sm_\C)^{\Delta^{op}}\to \Fun(\bS_{\C},\Nerve(\sSet))$$
be the Yoneda embedding,  {let} $\kappa: \Fun(\bS_{\C}, \bC)\to \Fun(\bS_\C,  \bC[W^{-1}] )$ be the notation for
the canonical map in the cases $\bC=\Cat$ and $\bC=\sSet$,
and  {let}
$$L: \Fun(\bS_{\C},\Nerve(\sSet)[W^{-1}])\to  \Fun^{desc}(\bS_{\C},\Nerve(\sSet)[W^{-1}])$$ 
be
the sheafification functor \eqref{jan2704}. 
If we interpret the sheaf of groups $Y(GL(n))$ as the sheaf of automorphisms of the trivial $n$-dimensional bundle, then we get a  morphism  
$$\kappa(Y(GL))\to i\Vect_{Mf,\C}$$ in $\Fun(\bS_{\C},\Nerve(\Cat)[W^{-1}])$. Applying the nerve   we get a morphism in $\Fun(\bS_\C,\Nerve(\sSet)[W^{-1}])$
$$\kappa(Y(BGL_{\bullet}))=  \kappa(Y(\Nerve(GL)))\simeq \Nerve(\kappa(Y(GL)))\to \Nerve(i\Vect_{Mf,\C}) \ .$$
 {As} every bundle is locally trivial and  {as} $\Nerve(i\Vect_{Mf,\C})$ satisfies descent, this morphism induces an equivalence
\begin{equation}\label{apr180rwerewrewrewrewr324341}
L(\kappa(Y(BGL_{{\bullet}}))) \simeq  \Nerve(i\Vect_{\C})
\end{equation} 
in $\Fun^{desc}(\bS_{\C},\Nerve(\sSet)[W^{-1}])$. 
Since  $\Omega^{\infty} H(\DR_{\C}(p))$ also satisfies descent, this in turn
induces an equivalence of mapping spaces between objects of $\Fun(\bS_{\C}, \Nerve(\sSet)[W^{-1}])$
\begin{equation}\label{sep0401}
\map( \Nerve(i\Vect_{\C}),\Omega^{\infty} H(\DR_{\C}(p))) \simeq \map(\kappa(Y(BGL_{{\bullet}})),\Omega^{\infty}H(\DR_{\C}(p))).
\end{equation}

\begin{lem}
For every $k \geq 0$, we have a natural isomorphism
\begin{equation}\label{sep0402}
\pi_k\left(\map(\kappa(Y(BGL_{{\bullet}})),\Omega^{\infty}H(\DR_{\C}(p)))\right) \cong H^{-k}\left(\Tot (\DR_{\C}(p)(BGL_{{\bullet}}))\right).
\end{equation}
\end{lem} 
\begin{proof}
For $q\in \nat$
we can consider the sheaf of sets $Y(BGL_{q})$ on $\bS_{\C}$ as a sheaf of constant simplicial sets.
Then we can define $Y_{\infty}(BGL_{q}):=\kappa(Y(BGL_{q}))\in \Fun(\bS_\C, \Nerve(\sSet)[W^{-1}])$.

Let $\iota:\Nerve(\sSet)\to \Nerve(\sSet)[W^{-1}]$ be the canonical morphism.
If $S_{\bullet\bullet}$ is a bisimplicial set, then there is a natural equivalence $\iota(\operatorname{diag} S_{\bullet\bullet}) \simeq \colim_{[q]\in\Delta^{op}} \iota(S_{\bullet q})$ 
  in $\Nerve(\sSet)[W^{-1}]$. We apply this  to the diagram $Y(BGL_{{\bullet}})$ of  
   bisimplicial sets which are constant in the first simplicial direction.
Then we get the
  equivalence 
$$\kappa(Y(BGL_{{\bullet}})) \simeq \colim_{[{q}]\in \Delta^{op}} Y_{\infty}(BGL_{{q}})  $$
in  $\Fun(\bS_{\C},\Nerve(\sSet)[W^{-1}])$.
Using this and the Yoneda Lemma  \cite[Lemma 5.1.5.2]{HTT} we get the equivalences 
\begin{align*}
\map(\kappa(Y(BGL_{{\bullet}})),\Omega^{\infty}H(\DR(p))) &\simeq \lim_{[{q}]\in \Delta }
\map(Y_{\infty}(BGL_{{q}}),\Omega^{\infty}H(\DR_{\C}(p)))\\[0.3cm]
&\simeq \lim_{[{q}]\in \Delta }\Omega^{\infty}H(\DR_{\C}(p)(BGL_{{q}})) .
\end{align*}
Note that $H$ and $\Omega^{\infty}$ commute with limits since they are both right adjoints.  
  Since moreover 
$$\lim_{[q]\in\Delta} \DR_\C(p)(BGL_q) \simeq \Tot( \DR_\C(p)(BGL_{\bullet})),$$ 
we obtain the equivalence
\[
\map(\kappa(Y(BGL_{{\bullet}})),\Omega^{\infty}H(\DR(p))) \simeq \Omega^{\infty}H \left( \Tot( \DR_\C(p)(BGL_{\bullet}))\right).
\]
Since for any complex $A \in \Nerve(\Ch)[W^{-1}]$ and $k  {\in\Z}$ we have $\pi_k(H(A)) \cong H^{-k}(A)$, we get a natural isomorphism
\[
\pi_k\left(\map(\kappa(Y(BGL_{{\bullet}})),\Omega^{\infty}H(\DR(p))) \right) \cong H^{-k}\left( \Tot( \DR_\C(p)(BGL_{\bullet}))\right)
\]
which is exactly \eqref{sep0402}.
\end{proof}

Using  \eqref{jan3001}  the right-hand side of \eqref{sep0402} can be expressed in terms of the absolute Hodge cohomology of the simplicial variety $BGL_{{\bullet}}$  {as}
\begin{equation}\label{sep0403}
H^{-k}\left(\Tot( \DR_{\C}(p)(BGL_{{\bullet}}))\right) \cong \prod_{n\in\nat_{0}} H^{2p-k}_{\text{Hodge}}(BGL_{\bullet}(n),\R(p)).
\end{equation}

According to Deligne \cite[9.1.1]{HodgeIII} we have  $H^{2k+1}(BGL_{{\bullet}}(n), \R) =0$ and   
the real Hodge structure on $H^{2k}(BGL_{{\bullet}}(n),\R)$ is pure of type $(k,k)$, i.e.~the weight filtration has a single step at $2k$ and the Hodge 
filtration  {has} a single step at $k$. Using this and the definition of $\DR_{\C}(p)$ as a cone one sees that  
\begin{equation}\label{apr1902}
H^{k}_{\text{Hodge}}(BGL_{{\bullet}}(n),\R(p)) \cong H^{k}(BGL_{{\bullet}}(n),  { i^{p}\R} ) \ ,\quad  {0\le p\leq \frac{k}{2}}\ , \mbox{$k$ even,}
\end{equation}  
and all other absolute Hodge cohomology groups of $BGL_{{\bullet}}(n)$ vanish.
In particular, combining \eqref{sep0401}, \eqref{sep0402}, and \eqref{sep0403} we get
\begin{equation}\label{may1601}
\pi_0(\map(M,\Omega^{\infty} H(\DR_{\C}(p)))) \cong \prod_{n\in\nat_{0}} H^{2p}(BGL_{{\bullet}}(n), i^{p} \R). \end{equation}

The simplicial  {variety} $BGL_{\bullet}$ has the structure of a monoid in $(\Sm_\C)^{\Delta^{op}}$   {given by the block sum of matrices.}
We denote by $\mu, \pr_1, \pr_2 : BGL_{\bullet} \times BGL_{\bullet} \to BGL_{\bullet}$ the monoid structure and the projections, respectively. The subspace of primitive elements
\begin{equation}\label{sep0404}
H^*\left( \Tot( \DR_\C(p)(BGL_{\bullet}))\right)^{prim} \subseteq H^*\left( \Tot( \DR_\C(p)(BGL_{\bullet}))\right)
\end{equation}
is by definition the set of elements $x$ satisfying
\[
\mu^{*}(x)=\pr_{1}^{*}x+\pr_{2}^{*}x.
\]

It is known that
the singular cohomology of $BGL_{{\bullet}}(n)$ is a polynomial ring 
\[
H^{2*}(BGL_{{\bullet}}(n), \R) = \R[s_{1,n}, \dots, s_{n,n}]
\]
where $\frac{1}{ p!} s_{p,n}\in  H^{2p}(BGL_{{\bullet}}(n), \R)$ is the  component of the Chern character in degree $2p$  {divided by $i^{p}$}. In addition we set $s_{0,n}:=n$. Note that $s_{p,0}=0$ for $p\ge 1$. For the map $\mu:BGL_{{\bullet}}(n) \times BGL_{{\bullet}}(m) \to BGL_{{\bullet}}(n+m)$ induced by the direct sum, we have
\[
\mu^*(s_{k,n+m}) = \pr_1^*s_{k,n} + \pr_2^*s_{k,m}.
\]
This together with the fact that the space of primitive elements intersects trivially with the space of decomposable elements and an  easy computation imply that
the subgroup
of primitives 
$ H^{0}(\Tot(\DR_{\C}(p)(BGL_{{\bullet}})))^{prim}$  is the one-dimensional $\R$-vector space  spanned by  {the} class   
$$s_{p}:= (s_{p,0},s_{p,1},s_{p,2},\dots)\ , $$
where we denote the preimage of $i^{p}s_{p,n}$ in $H^{0}(\Tot(\DR_{\C}(p)(BGL_{{\bullet}}(n))))$ under the isomorphisms \eqref{sep0403} and \eqref{apr1902} by the same symbol.

We have an injection 
\begin{equation}\label{nov1910} 
\Prim( \Nerve(i\Vect_{\C}),\DR_{\C}(p))\hookrightarrow H^{0}(\Tot(\DR_{\C}(p)(BGL_{{\bullet}})))^{prim}\ .\end{equation}
It follows from the constructions of the characteristic cocycle $\omega(p)$ in Definition \ref{may1760n} and the regulator $\cbeil$  in Definition \ref{jul0816}  that the element $c_{\omega(p)}\in \Prim( \Nerve(i\Vect_{\C}),\DR_{\C}(p))$ corresponding to the $p$-component of $\cbeil$ 
  goes to the
class $\frac{1}{p!} s_{p}$ under the map (\ref{nov1910}).
This shows surjectivity of (\ref{nov1910}) {and finishes the proof of the proposition}. \end{proof}

\subsection{The relative Chern character} \label{jan1961}

In this subsection we present some additional material  which fits into the present paper, since the constructions and arguments use the techniques and ideas developed for the main part of the paper. 
Since the results shown here are not used  elsewhere in the paper this subsection {can  safely be skipped on  first reading}.

An alternative approach to regulators goes back to Karoubi \cite{KaroubiCR, KaroubiAst} using his relative Chern character. This approach was further studied and generalized in \cite{Tamme-Beil}. 
Roughly, the idea is as follows: We consider a  smooth variety $X$ over $\C$. Then we have a natural comparison map 
\begin{equation}\label{ascascsacsaee1e12e2ecsacsc}
c: \bK_\C(X) \to \bK^{top}(X)
\end{equation} from the algebraic to the topological $K$-theory spectrum of $X$.
If we compose $c$ with the usual Chern character
 $$\ch: \bK^{top}(X)\to H(\bA(X))$$
 from  {the} topological $K$-theory to {the} 
de Rham cohomology of $X$ {(cf.~\eqref{eq:def-deRham})}, then the resulting map looses a lot of information. Instead, we define the relative $K$-theory $\bK^{rel}(X)$ as the homotopy fibre of the map $c$. 
 It is the domain of the relative  Chern character 
$$\ch^{rel} :\pi_i(\bK^{rel}(X))\to \prod_p H^{2p-i-1}(X;\C)/\cF^p,$$  
which carries interesting secondary 
information.

For all $i,p\ge 0$ we have a natural map  
 $$H^{2p-i-1}(X;\C)/\cF^p\to H^{2p-i}(\bD(p)(X))\ ,$$ 
where the complex $\bD(p)$ representing weak absolute Hodge cohomology is defined similarly as the absolute Hodge complex $\DR_\C(p)$, but discarding the weight filtration (see \eqref{mai2301} for details). The weak analogue of Beilinson's regulator for smooth varieties over $\C$ induces the  maps  
$$\mathtt{r}^{\mathtt{Beil}}_{\C,\bD,p}: \pi_i(\bK_\C(X)) \to H^{2p-i}(\bD(p)(X))$$ (see \eqref{mai2401}) for all $i,p\ge 0$.
The following theorem has been shown in \cite[Thm. 3.9]{Tamme-Beil}:
\begin{theorem} \label{mai2402}
For all $p,i\geq 0$  the diagram 
\[
\xymatrix@C+0.5cm{
\pi_{i}(\bK^{rel}(X))\ar[d]\ar[r]^-{\ch_p^{rel}}& H^{2p-1-i}(X;\C)/\cF^{p}H^{2p-1-i}(X;\C)\ar[d]\\
\pi_{i}(\bK_{\C}(X))\ar[r]^{\mathtt{r}^{\mathtt{Beil}}_{\C,\bD,p}}&H^{2p-i}(\bD(p)(X)) 
}
\]
commutes. 
\end{theorem}

The goal of the present section is  to give the details of the constructions and a proof of this theorem using the techniques developed in the present  paper.  We consider the site $\Mf$ and the $\infty$-category $\bS_{Mf}:=\Nerve(\Mf^{op})$. {For technical reasons, we also need to consider the product $\bS_{Mf,Mf} := \bS_{Mf}\times \bS_{Mf}$.}
We consider the stack ${i\Vect^{top}_{Mf}}$ on ${\Mf\times \Mf}$ which associates to ${M\times X \in \Mf \times \Mf}$
the symmetric monoidal groupoid of smooth complex vector bundles on ${M\times X}$. We interpret this stack also as an object
$$
{i\Vect^{top}_{Mf}} \in \Fun^{desc}({\bS_{Mf,Mf}},\CommMon(\Nerve(\Cat)[W^{-1}]))\ .
$$
This stack does not model topological $K$-theory since it does not take the
topology of the morphism spaces into account. We can remedy this defect 
by an application of the functor 
$\bar \bs:\Fun(\bS_{Mf,Mf},\Sp)\to \Fun(\bS_{Mf,Mf},\Sp) $ defined similarly as the corresponding transformation \eqref{jul0741bn}, 
but for the site $\bS_{Mf,Mf}$ instead of $\bS_{Mf,\C}$.
We combine $\bar \bs$ and the algebraic $K$-theory functor introduced in Definition \ref{nov1901} in order to define
the sheaf of topological $K$-theory spectra
\begin{equation}\label{sep1501}
\bK^{top}_{{Mf}}:= L(\bar \bs (K({i\Vect^{top}_{Mf}})))  \in \Fun^{desc}({\bS_{Mf,Mf}}, \Sp )\ .
\end{equation}
We further define the sheaf of topological $K$-theory spectra 
$$\bK^{top} \in \Fun^{desc}(\bS_{Mf}, \Sp )$$ 
by specializing the first factor to be a point.
By  \cite[Lemma 6.3]{Bunke:2013aa} we know that $$ \bK^{top}\simeq \underline{\ku}$$
where $\underline{\ku}$ is the constant sheaf of spectra induced by the connective topological $K$-theory spectrum of $\C$.


Next, we construct 
the  comparison map \eqref{ascascsacsaee1e12e2ecsacsc}.
Let $h\colon \bS_{\C} \to {\bS_{Mf}}$ be the map given by sending a  smooth variety  over $\C$ to its associated smooth manifold. It is compatible with the topologies and induces a map $h\colon \bS_{Mf,\C} \to \bS_{Mf,Mf}$. We have a natural transformation of stacks $u\colon i\Vect_{{Mf,\C}}\to h^{*}i\Vect^{top}_{{Mf}}$ which maps a  bundle    to its underlying smooth complex vector bundle.
The  comparison map is now defined by 
\[
c:\bK_{{Mf,}\C}=L(K(i\Vect_{{Mf,\C}})) \xrightarrow{u} h^*L(K(i\Vect^{top}_{{Mf}})) \xrightarrow{\text{\eqref{jul0811bn}}}
h^{*}L(\bar \bs(K(i\Vect^{top}_{{Mf}}))) =  h^{*}\bK^{top}_{{Mf}} \ .
\]
 We
define the relative $K$-theory $\bK^{rel}_{{Mf}}$ by  forming the fibre sequence in $\Fun^{desc}(\bS_{Mf,\C},\Sp)$
$$\bK^{rel}_{{Mf}} \to \bK_{{Mf,}\C}  \xrightarrow{c} h^*\bK^{top}_{{Mf}}\ .$$
We again define $\bK^{rel} \in \Fun^{desc}(\bS_\C,  \Sp )$ by specializing the manifold in the first factor to be a point.

The relative Chern character will be induced by compatible regulator maps on $\bK_{{Mf,\C}}$ and $h^*\bK^{top}_{{Mf}}$.   We now describe the relevant choice of geometries and characteristic cocycles.
A geometry on a complex vector bundle $V$ over $M\times X$ is a pair $g^V=(h^{V},\nabla^{V})$ consisting of a hermitian metric $h^{V}$ and a connection $\nabla^{V}$.
A local geometry on the complex vector bundle $V$ over $M\times X$ is a family $(g_m)_{m\in M}$ of germs of geometries. In other words, each $g_m$ is represented by a geometry on $V_{|U\times X}$ where $U$ is a neighbourhood of $m\in M$. A geometric smooth bundle is a pair $(V,g)$ consisting of a bundle $V$ with local geometry $g$. 
We let $i\Vect^{top,geom}_{Mf}$ be the symmetric monoidal stack of
geometric smooth bundles.
 {Similarly} as in Lemma \ref{jul0803}, the forgetful transformation induces an equivalence 
\begin{equation}\label{sep1301}
\bar \bs(\Nerve(i\Vect^{top,geom}_{{Mf}})) \stackrel{\sim}{\to} \bar\bs (\Nerve(i\Vect^{top}_{{Mf}}))
\end{equation}
in $\Fun(\bS_{Mf,Mf},\CommMon(\Nerve(\sSet)[W^{-1}]))$.

We define the sheaf of complexes
\begin{equation}\label{eq:def-deRham}
\bA_{Mf}:=\prod_{p\ge 0} A[2p]\in \Fun^{desc}(\bS_{Mf,Mf},  {\Nerve(\Ch)})
\end{equation}
 where $A$ is the usual de Rham complex of smooth $\C$-valued forms. 
 Furthermore, we let $\bA\in \Fun^{desc}(\bS_{Mf}, {\Nerve(\Ch)})
$ be the restriction to the site $\bS_{Mf}$.

To a complex vector bundle $V$ on $M\times X$ with geometry $(h^V,\nabla^V)$ we associate the characteristic form
\[
\omega^{top}(V,(h^V,\nabla^V)) := \prod_{p\ge 0}\ch_{2p}(\nabla^{V})
\]
(see \eqref{2w212w2wfwefwefewfw12} for the definition of $\ch_{2p}$). By a precisely analogous construction  as in Subsections \ref{nov1101}, \ref{dez2501} (in particular using the canonical deformations) we get   the characteristic cocycle $$\omega^{top}(g)\in Z^0(\cL\bA_{Mf}(M\times X))$$ of a geometric smooth bundle $(V,g)$ in the \v Cechification (see Subsection \ref{may2201}) $\cL\bA_{Mf}$ of $\bA_{Mf}$, where $\cL$ acts on the first component in $\Mf\times \Mf$. 

Similarly to   the construction of Beilinson's regulator in Definition \ref{jul0816}, we 
define the {topological} Chern character 
$$\ch:\bK^{top} \to H( \bA)$$
through the diagram $\Fun(\bS_{Mf,Mf},\Sp)$
\[
\xymatrix@C+0.6cm{
\bar \bs (K( i\Vect^{top}_{Mf}))\ar[d]^{{\eqref{sep1501}}}  & {\bar\bs (K(i\Vect_{Mf}^{top,geom}))}\ar[l]_-{\sim}^-{\eqref{sep1301}} \ar[r]^-{\bar \bs (r(\omega^{top}))}
& \bar\bs H(\cL\bA_{Mf}) \\
\bK^{top}_{Mf}\ar[r]^{\ch}  & H(\bA_{Mf})  \ar[r]^{\sim}_{\text{Fact \ref{dkjhwqjkdwqdwqdwqdwqdwqd}}} & H(\cL\bA_{Mf}). \ar[u]^{ \sim}_{\text{Lemma \ref{jul0730}}}
}
\]

We can also consider its real variant $\ch_{\R}$ by replacing $\bA_{Mf}$ with the complex of real valued differential forms 
$$
\bA_{\R,Mf}:=\prod_{p\ge 0} (2\pi i)^{p} A_\R[2p] \in  \Fun^{{desc}}(\bS_{Mf,Mf}, {\Nerve(\Ch)})
$$
and using the characteristic {form (cf. \eqref{dez2602}) $$\omega^{top}_{\R}(V,(h^V,\nabla^V)) := \prod_{p\geq 0} \ch_{2p}(\nabla^{V,u})\ .$$}

Geometries on algebraic vector bundles are defined in such a way that the corresponding Chern character forms live in the appropriate step of the Hodge filtration. We thus
consider  
\[
\cF^0\bA_{\log,Mf} := \prod_{p\geq 0} \cF^p A_{\log,Mf}[2p]\in \Fun^{ }(\bS_{Mf,\C},{\Nerve(\Ch)})\ .
\]
 {By Lemma \ref{lem:descentSmfC}, we can consider $\cF^{0}\bA_{\log,Mf}$ as an object in $\Fun^{desc}(\bS_{Mf,\C},\Nerve(\Ch)[W^{-1}])$.}
There is a natural morphism $\cF^0\bA_{\log,Mf} \to h^* \bA_{Mf}$ in $ \Fun  (\bS_{Mf,\C},\Ch)$, which gives rise to a fibre sequence in $\Fun^{{desc}}(\bS_{Mf,\C},\Nerve(\Ch)[W^{-1}])$
\[
\bA_{rel,Mf} \to \cF^0\bA_{\log,Mf} \to h^* \bA_{Mf}
\]
defining $\bA_{rel,Mf}\in \Fun^{{desc}}(\bS_{Mf,\C},\Nerve(\Ch)[W^{-1}])$.
We denote its specialization to $\bS_{\C}$ by $\bA_{rel}$.
 We will actually  use the explicit model
\begin{equation}\label{hfkjwefhkewfewfewff}
\bA_{rel,Mf}\simeq \Cone(\cF^0\bA_{\log,Mf} \to h^* \bA_{Mf})[-1].
\end{equation}

From a good geometry $g^V=(h^{V},\nabla^{II})$ (see {Definition \ref{jan2702}}) we obtain a geometry 
$(h^{V},\nabla^{V})$ with $\nabla^{V}=\nabla^{II}+\nabla^{I}$  on the underlying complex vector bundle.
 {Using t}he characteristic form  
$$\omega_{\cF}(g^V):=\prod_{p\geq 0} \ch_{2p}(\nabla^V)$$
and the construction of Subsections \ref{nov1101}, \ref{dez2501}, we get for any geometric bundle $(V,g)$ on $M\times X\in \bS_{Mf,\C}$, i.e.~a bundle $V$ with a local geometry $g=(g_m)_{m\in M}$, a characteristic cocycle $$\omega_\cF(V,g) \in Z^0(\cL\cF^0\bA_{\log,Mf}(M\times X))\ .$$
Again, similarly as in the construction in Definition \ref{jul0816}, we get
a regulator map
$$\ch_{\cF}:\bK_{Mf,\C}\to H(\cF^0\bA_{\log,Mf})\ .$$
Clearly, the diagram  of characteristic cocycles
\[
\xymatrix{
 \pi_0(i\Vect^{geom}_{{Mf,\C}})\ar[r]^-{\omega_{\cF}} \ar[d] & Z^{0}({\cL} \cF^0\bA_{\log,Mf}) \ar[d]\\
h^*\pi_0(i\Vect^{top,geom}_{{Mf}}) \ar[r]^-{ {h^{*}(\omega^{top})}}& h^*Z^{0}({\cL}\bA_{Mf})
}
\]
commutes. 
 This gives us the commutativity of the lower square in the following diagram in $\Fun^{desc}(\bS_{Mf,\C},{\Sp})$.  We define the relative Chern character $\ch^{rel}$ to be the map induced on   fibres:
\begin{equation}\label{ejdlkqdqwdqwdwqd}
\xymatrix@C+0.5cm{
\bK^{rel}_{ {Mf}} \ar@{..>}[r]^-{\ch^{rel}} \ar[d] & H( {\bA_{rel,Mf}}) \ar[d]\\
\bK_{ {Mf},\C} \ar[r]^-{ \ch_\cF} \ar[d] & H(\cF^0\bA_{\log,Mf})\ar[d]\\
h^*\bK^{top}_{ {Mf}} \ar[r]^-{{h^{*}(\ch)}} & H(h^*\bA_{ {Mf}})\ .
}
\end{equation}
Specializing the manifold to a point, we get the map 
$$\ch^{rel}\colon \bK^{rel}\to H(\bA_{rel})$$ in $\Fun^{desc}(\bS_{\C}, \Sp)$.
Note that we have natural isomorphisms 
\begin{align*}
\pi_i(H(\bA_{rel}(X))) &\cong  H^{-i}(\Cone(\cF^0\bA_{\log}(X) \to h^*\bA(X))[-1]) \\
&\cong \prod_{p\geq 0} H^{2p-1-i}(X;\C)/\cF^{p}H^{2p-1-i}(X;\C)\ .
\end{align*}
This finishes the construction of the relative Chern character.

In order to compare it with the Beilinson regulator, we introduce the  Deligne-Beilinson or  weak absolute Hodge complex
\begin{equation}\label{mai2301}
\bD_{ {Mf,\C}}(p) := \Cone\left((2\pi i)^{p}A_{ {Mf,}\R} \oplus \cF^{p}A_{\log, {Mf}} \to A\right)[2p-1] \in  \Fun (\bS_{Mf, {\C}},\Nerve(\Ch))
\end{equation}
and set $\bD_{ {Mf,\C}} := \prod_{p\geq 0} \bD_{Mf,\C}(p)$.
By inspection of the Definition  {\ref{nov2031e} of $\DR_{Mf,\C}$}
 we have a natural map 
\begin{equation}\label{dcscs2342342432dedewd}
\DR_{ {Mf,\C}} \to \bD_{ {Mf,\C}}
\end{equation} 
 {which simply forgets} the weight filtration
and the logarithmic growth condition at the appropriate places.

The image under \eqref{dcscs2342342432dedewd} of the characteristic ${\cL}\DR_{ {Mf,}\C}$-valued cocycle $\omega$ defined in  {Definition~\ref{may1760n}}  is an $\cL\bD_{ {Mf,\C}}$-valued characteristic cocycle 
 which we  denote by $\omega_{\bD}$. 
 {As in Definition~\ref{jul0816}, it induces a morphism}
\begin{equation}\label{mai2401}
\mathtt{r}^{\mathtt{Beil}}_{\C,\bD}  : \bK_{ {Mf,\C}} \to H(\bD_{ {Mf,\C}})
\end{equation}
in $\Fun^{desc}(\bS_{Mf,\C},\Sp)$, which we call the  weak Beilinson regulator.
It follows from the constructions that the composition of  the characteristic cocycle $\omega_{\bD}$ with the  projection to the second component of the cone $Z^{0}(\cL\bD_{ {Mf,\C}})\to Z^{0}(\cL\cF^{0} \bA_{\log, {Mf}})$
coincides with $\omega_{\cF}$
and moreover that the square 
\[
\xymatrix{
\pi_{0}(i\Vect_{ {Mf,}\C}^{geom} )\ar[r]^-{\omega_{\bD}} \ar[d] & Z^{0}(\cL\bD_{ {Mf,\C}}) \ar[d]\\
h^*\pi_{0}(i\Vect^{top,geom}_{ {Mf}})\ar[r]^-{\omega^{top}_{\R}}& Z^{0}(h^*\cL\bA_{\R, {Mf}})}
\]
commutes.

It follows from \eqref{hfkjwefhkewfewfewff} and \eqref{mai2301} that the fibre of the map $\bD_{ {Mf,\C}} \to h^*\bA_{\R, {Mf}}$,  given by the projection to the first component of the cone, is naturally equivalent to 
$ {\bA_{rel,Mf}}$ {(see also \eqref{sep0601} below)}.
 {Using the diagram} 
\begin{equation}\label{ejdlkqdqwe3e3eee33dqwdwqd}
\xymatrix@C+0.5cm{
\bK^{rel}_{ {Mf}} \ar@{..>}[r]^-{ } \ar[d] & H( {\bA_{rel,Mf}}) \ar[d]\\
\bK_{ {Mf},\C} \ar[r]^-{\mathtt{r}^{\mathtt{Beil}}_{\C,\bD}} \ar[d] & H( \bD_{ {Mf,\C}})\ar[d]\\
h^*\bK^{top}_{ {Mf}} \ar[r]^-{h^{*}(\ch_{\R})} & H(h^*\bA_{\R, {Mf}})
}
\end{equation}
 {the} weak Beilinson regulator
  $\mathtt{r}^{\mathtt{Beil}}_{\C,\bD}$ also induces a map
\begin{equation}\label{feb1004}
\bK^{rel}_{ {Mf}} \to  H(\bA_{rel,Mf}).
\end{equation}

We claim that \eqref{feb1004} coincides with $\ch^{rel}$ constructed in \eqref{ejdlkqdqwdqwdwqd}.  The assertion of Theorem \ref{mai2402} then follows by {specializing the manifold to a point and} taking homotopy groups.

We start with the following commutative diagram  in $\Fun^{desc}(\bS_{Mf,\C},\Nerve(\Ch)[W^{-1}])$
\begin{equation}\label{sep0601}
\xymatrix{
\Cone(\bD_{Mf,\C} \to h^*\bA_{\R, {Mf}})[-1] \ar[r] & \bD_{Mf,\C} \ar[r] & h^*\bA_{\R,Mf} \\
\Cone(\bD_{Mf,\C} \to \Cone(\cdots))[-1] \ar[u]^{\text{induced}}_{ {\simeq}} \ar[r] \ar[d]_{\text{induced}}^{\simeq} & 
	  \bD_{ {Mf,\C}} \ar@{=}[u] \ar[r] \ar[d] & \Cone(h^*\bA_{\R,{Mf}} \oplus h^*\bA_{ {Mf}} \to h^*\bA_{ {Mf}})[-1] \ar[u]_{ {\simeq}} \ar[d]     \\
 {\bA_{rel,Mf}}\ar[r] & \cF^0\bA_{\log, {Mf}} \ar[r] & h^*\bA_{ {Mf}}.
}
\end{equation}
Here the map from $\bD_{ {Mf,\C}}$ to the cone in the right column is induced by $\cF^0\bA_{\log, {Mf}} \to h^*\bA_{ {Mf}}$, the maps in the right column are the obvious projections, and the maps in the left column are the induced ones.  The lower left vertical arrow is an equivalence since the lower right square is a pull-back square.

To prove the claim, we introduce the auxiliary characteristic cocycle
\[
\omega_{\Cone}: \pi_0(i\Vect^{top,geom}_{ {Mf}}) \to Z^0\left({\cL}\Cone(\bA_{\R, {Mf}} \oplus \bA_{ {Mf}} \to \bA_{ {Mf}})[-1]\right)
\]
which, similarly as before, is induced by sending a complex vector bundle with geometry 
$(V,(h^V, \nabla^V))$ to $$\prod_{p\geq 0} (\ch_{2p}(\nabla^{V,u})\oplus\ch_{2p}(\nabla^V), \widetilde\ch_{2p-1}(\nabla^{V,u},\nabla^V))$$ (see Subsection \ref{dez2501} for the definition of the transgression $\widetilde\ch$ and compare with \eqref{nov1102}).
 { 
It follows directly from the definitions that we have the compatibility of characteristic cocycles defined on $\pi_0(i\Vect_{Mf,\C}^{geom})$, (or on 
 $h^*\pi_0(i\Vect^{top,geom}_{ {Mf}})$, respectively) with values in the groups of $0$-cocycles in the \v Cechifications of the middle respectively right column of diagram \eqref{sep0601},
summarized schematically in the diagram
\begin{equation}\label{nocheindiagramm}
\xymatrix@C+0.5cm@R-0.3cm{
\omega_{\bD} \ar@{|->}[r] & \omega^{top}_{\R}\\
\omega_{\bD} \ar@{=}[u] \ar@{|->}[r] \ar@{|->}[d] & \omega_{\Cone} \ar@{|->}[u] \ar@{|->}[d] \\
\omega_{\cF} \ar@{|->}[r] & \omega^{top}.
}
\end{equation}
As in the diagrams \eqref{ejdlkqdqwdqwdwqd} and \eqref{ejdlkqdqwe3e3eee33dqwdwqd}, the characteristic
cocycles $\omega_{\bD}$ and $\omega_{\Cone}$ thus give rise to a third map $\bK^{rel}_{Mf} \to H(\bA_{rel,Mf})$. Moreover, by the commutativity of \eqref{sep0601} and the compatibilities in \eqref{nocheindiagramm}, this third map coincides with $\ch^{rel}$ from \eqref{ejdlkqdqwdqwdwqd}  -- corresponding to the bottom row of \eqref{nocheindiagramm} --   and with \eqref{feb1004} -- corresponding to the top row of \eqref{nocheindiagramm}. This implies our claim. \hB
%
%
%
}


\begin{rem}
Whereas with this construction of the relative Chern character the relation to Beilinson's regulator follows quite easily, it is not a priori clear that this new construction really extends Karoubi's original one. This is done in \cite{Tamme-Beil} and used there to establish also the comparison with Borel's regulator in the number field case.
\end{rem}

\section{Differential algebraic $K$-theory}\label{jan1964}

\subsection{Definition of differential algebraic $K$-theory}\label{dlqkwdjlqdwqdwqdwqdwqdw24234234}

A spectrum $E\in \Sp$ represents a cohomology theory $E^{*}$ on the category of topological spaces. 
A differential extension $\widehat E^{*}$ combines the restriction of the functor $E^{*}$ to the category  of smooth manifolds with    information about characteristic forms.
  In the following, we describe the Hopkins-Singer version of differential cohomology introduced in the ground-breaking paper \cite{MR2192936}.   
Recall that $H:\Nerve(\Ch)[W^{-1}]\to \Sp$ denotes the Eilenberg-MacLane correspondence.
In order to define a Hopkins-Singer differential cohomology theory, we must choose differential data $(E, {C},c)$.  It consists of the spectrum $E$ as above, a chain complex $ {C}\in \Ch$ of real vector spaces,
and a morphism of spectra $c:E\to H( {C})$. The differential data is called strict  if this morphism induces isomorphisms
$\pi_{*}(E)\otimes \R\cong H^{-*}( {C})$.

A standard choice for $ {C}$ is the  $\Z$-graded vector space $ {C}:=\pi_{*}(E)\otimes \R$ with trivial differentials. 
 We then have a unique\footnote{For uniqueness we use the fact that $C$ consists of $\Q$-vector spaces.} equivalence class of  morphisms $c:E\to H( {C})$ which induces the canonical map $$\pi_{*}(E)\to \pi_{*}(H( {C}))\cong\pi_{*}(E)\otimes \R$$ in homotopy groups. 
The triple $(E,{C},c)$ is  called the canonical differential data associated to the spectrum $E$. It is strict.

For a chain complex of real vector spaces $ {C}$, we let $\Omega {C}\in \PSh_{\Ch}(\Mf)$ be the de Rham complex with coefficients in $C$.
Note that it actually is a sheaf, i.e.~$\Omega C\in \Fun^{desc}(\bS_{Mf},\Nerve(\Ch))$.

\begin{rem}At this point we must consider $\Omega  {C}$ as a sheaf with values in $\Nerve(\Ch)$ and not in $\Nerve(\Ch)[W^{-1}]$, since we want to form the stupid  {truncation} $\sigma^{\ge 0}\Omega {C}$ below, which  is only well-defined before inverting quasi-isomorphisms.
\end{rem} 

 For a spectrum $F$, let  
 $\underline{F}\in \Fun^{desc,h}(\bS_{Mf},\Sp)$ be the constant sheaf of spectra induced by $F$ {(cf.~Subsection~\ref{mar0804})}.   We have a version of the de Rham Lemma  {(Lemma~\ref{efef234234fwefewfewfewfwfew})}, which provides the equivalence 
  $$\underline{H( {C})}\simeq H(\Omega  {C})$$ 
  in $\Fun^{desc,h}(\bS_{Mf},\Sp)$.
Given differential data $(E,C,c)$, we thus have a canonical equivalence class of morphisms
\begin{equation}\label{t54t4t54t4t4t456rth}
\rat:\underline{ {E}}\xrightarrow{c} \underline{H( {C})}\simeq H(\Omega  {C})\ .
\end{equation}
For every integer $n\in \Z$, the sheaf  $\Diff^{n}(E)\in \Fun^{desc}(\bS_{Mf},\Sp)$  of  differential function spectra is defined by the pull-back
\begin{equation}\label{r23rr2r23jkr3hkjh23r}
\xymatrix{
\Diff^{n}(E)\ar[d]^{I}\ar[r]^-{R}&H(\sigma^{\ge n} \Omega  {C})\ar[d]\\
\underline{E}\ar[r]^-{\rat}&H(\Omega {C})
}
\end{equation}
in $\Fun^{desc}(\bS_{Mf},\Sp)$,
where $\sigma^{\ge n}\Omega A$ is the stupid 
truncation. Finally, 
the differential cohomology groups of  {a manifold} $M$ are  defined by   $$\widehat{E}^{n}(M):=\pi_{-n}(\Diff^{n}(E)(M))\ .$$ 
This construction of differential cohomology is called the Hopkins-Singer construction because of the foundational work \cite{MR2192936}. 
In this streamlined form it has been introduced in \cite{bg},  \cite[Sec.~4.4]{Bunke:2013aa}.


In the present paper we consider the differential extension of the algebraic $K$-theory of arithmetic schemes  {(see Subsection \ref{jul1001})}. 
Its construction is modelled on the Hopkins-Singer construction explained above, but modifications are needed since we work on a product site $\Mf\times  {\Reg_{\Z}}$.
Our choice of differential forms is the complex $\cL\DR_{Mf,\Z}$. This complex is connected with algebraic $K$-theory by  the
regulator 
\[
 \bK_{Mf,\Z} \xrightarrow{\beil} H(\DR_{Mf,\Z}) \stackrel{\sim}{\to} H(\cL\DR_{Mf,\Z})
\]
given in  Definition~\ref{jul08161}. This map replaces the map $\rat$ in \eqref{t54t4t54t4t4t456rth}. 
Usually, a differential extension of a cohomology theory is considered as a functor on the category of manifolds. 
Here it is a functor $$M\times X\mapsto \widehat{\bK}(X)^{0}(M)$$ of two variables, {an arithmetic} scheme $X$ 
and a manifold $M$. Furthermore, 
the differential data  
are far from being strict. In order to explain this,
let us fix $X$ and write 
\begin{equation}\label{apr2450}\DR_\Z(X)(M):=\DR_{Mf,\Z}(M\times X)\ .\end{equation}
Let us assume that $X$ is proper  {over $\Spec(\Z)$} in order to avoid the complications with compactifications.
Then $\DR_\Z(X)(M)$ is a topological completion of the de Rham complex of $M$ with coefficients in
$\DR_\Z(X)(*)$. Even if $X$ is zero-dimensional, $\DR_\Z(X)(*)$ is a complex with non-trivial differential.
Beilinson's conjectures, described in Subsection \ref{nov2602}, predict that, for {$X\to\Spec(\Z)$ proper and flat}  with potentially good reduction everywhere, the regulator induces an isomorphism $\bK_p(X)\otimes \R\to H^{-p}(\DR_\Z(X)(*))$ for $p\ge 2$, and its kernels and cokernels for other $p$ 
are well understood and non-trivial in general. 
Hence even if Beilinson's conjectures are assumed  {to be true}, the map 
$$\bK_*(X)\otimes \R\to H^{-*}(\DR_\Z(X)(*))$$   
is not an isomorphism.

We now turn  to the construction of 
the differential  algebraic $K$-theory functor  $\widehat{\bK}^{0}$
and the structure maps $R$, $I$, $a$ of a differential cohomology theory. 
The  differential   algebraic $K$-theory functor    will be defined in terms of the differential algebraic $K$-theory spectrum.
We denote the stupid truncation of a complex to the non-negative degree part by $\sigma^{\geq 0}$.

\begin{ddd}\label{jan0503}
We define the differential algebraic $K$-theory spectrum 
$$\Diff(\bK)\in \Fun (\bS_{Mf,\Z}, \Sp )$$ by the pull-back in $\Fun (\bS_{Mf,\Z}, \Sp )$
\begin{equation}\label{nov2601}
\begin{split}
\xymatrix{\Diff(\bK)\ar[r]^-{R}\ar[d]^{I}&H( {\sigma}^{\ge 0}\cL\DR_{Mf,\Z})\ar[d]^{i} \\
 \bK_{Mf,\Z} \ar[r]^-{\beil}&H(\cL\DR_{Mf,\Z})\ .
}
\end{split}
\end{equation}
Furthermore, we define the differential algebraic $K$-theory (in degree zero) by
$$\widehat{\bK}^{0}:= \pi_0(\Diff(\bK))\in \Fun(\bS_{Mf,\Z},\Nerve(\Ab))\ .$$
\end{ddd}
{Note that a pull-back in $\Fun(\bS_{Mf,\C},\Sp)$ can be computed object-wise.} 

While $ \bK_{Mf,\Z}$ and $H(\cL\DR_{Mf,\Z})$ are homotopy invariant in the manifold direction   (see Lemma \ref{dez1405} for the latter) the complex  $H( {\sigma}^{\ge 0}\cL\DR_{Mf,\Z})$ is no longer homotopy invariant.
Therefore $\Diff(\bK)$ is not homotopy invariant as well. Hence the functor $\widehat{\bK}^{0}$ is not homotopy invariant in the manifold direction. The deviation from homotopy invariance can be described by a
homotopy formula, which will be formulated  in  Lemma \ref{jan0605}.

 The maps $R$ and $I$ in the diagram \eqref{nov2601} induce the structure maps (denoted by the same symbols)
$$R:\widehat{\bK}^{0}\to Z^{0}(\cL\DR_{Mf,\Z})\ , \quad  I:\widehat{\bK}^{0}\to \pi_{0}( \bK_{Mf,\Z})\ .$$
The inclusion $i$ of the  non-negative 
part into the full de Rham complex fits into a fibre sequence 
\begin{equation}\label{wedewdewdd3455345wef}
\dots\to  {\sigma}^{\le -1} \cL\DR_{Mf,\Z}[-1]\rightarrow  {\sigma}^{\ge 0} \cL\DR_{Mf,\Z} \xrightarrow{i} \cL \DR_{Mf,\Z}\to  {\sigma}^{\le -1} \cL\DR_{Mf,\Z}\to \dots \end{equation}
in $\Fun(\bS_{Mf,\Z},\Nerve(\Ch)[W^{-1}])$.
Since  \eqref{nov2601} is a pull-back square, the fibre sequence \eqref{wedewdewdd3455345wef} induces a fibre sequence
\begin{equation}\label{wedewdewdd3455345wefe}
\dots\to  H({\sigma}^{\le -1} \cL\DR_{Mf,\Z}[-1])\xrightarrow{a}  \Diff(\bK) \xrightarrow{I} \bK_{Mf,\Z}\to  H({\sigma}^{\le -1} \cL\DR_{Mf,\Z})\to \dots \ . 
\end{equation} 
The map marked by $a$ in this sequence 
induces the map
\begin{equation}\label{mar1401}
a: H^{-1}( {\sigma}^{\le -1}\cL\DR_{Mf,\Z})\to \widehat{\bK}^{0}\ .
\end{equation}  Applying $\pi_{0}$ to \eqref{wedewdewdd3455345wefe} we get the exact sequence
\begin{equation}\label{mar1402}
\pi_1( \bK_{Mf,\Z} ) \to H^{-1}(\sigma^{\leq -1}\cL\DR_{Mf,\Z}) \xrightarrow{a} \widehat\bK^0 \xrightarrow{I} \pi_0( \bK_{Mf,\Z} ) \to 0\ ,
\end{equation}
 where the surjectivity of the last map follows from  $H^0(\sigma^{\leq -1}\cL\DR_{Mf,\Z})=0$. 
\begin{ddd}\label{apr2321}
We define  the flat part of differential algebraic $K$-theory by
$$\widehat{\bK}^{0}_{flat}:=\ker\left(R:\widehat{\bK}^{0}\to Z^{0}(\cL\DR_{Mf,\Z})\right)\ .$$
\end{ddd}

\subsection{The homotopy of $\Diff(\bK)$}\label{djqlwdqwldwqdwqdqd}

In this subsection we calculate the homotopy groups of $\Diff(\bK)$. We fix $X\in  {\bReg}_{\Z}$ 
and write 
\[
\Diff(\bK(X))  (M):=\Diff(\bK)(M\times X)\ ,
\]
\begin{equation}\label{jul1101}
\widehat{\bK}(X)^{0}(M):=\widehat{\bK}^{0}(M\times X)\ ,
\end{equation} 
and 
\begin{equation}\label{jul1120}
\widehat{\bK}(X)_{flat}^{0}(M):=\widehat{\bK}_{flat}^{0}(M\times X)\ . 
\end{equation}
Furthermore, we let ${\bK}_{\Z}(X)^{*}$ and ${\bK}_{\Z}(X)\R/\Z^{*}$ be the  generalized cohomology 
theories represented by the spectra ${\bK}_{\Z}(X)$ and its $\R/\Z$-version ${\bK}_{\Z}(X)\R/\Z:={\bK}_{\Z}(X)\wedge M\R/\Z$, respectively (compare with \eqref{jul1080}). {Here $M\R/\Z$ denotes the Moore spectrum of $\R/\Z$.}
The calculation will partially depend on the validity of Beilinson's   conjectures for $X$.
We will indicate this dependence precisely at the corresponding places.
\begin{prop}
The homotopy groups of $\Diff(\bK(X))$ can be described as follows:
\begin{enumerate}
\item For $i\le -1$ we have 
$$\pi_{i}(\Diff(\bK(X)) )\cong  \bK_\Z(X)^{-i}\ , \quad i\le -1\ .$$
\item For $i\ge 1$ we have a natural map
\begin{equation}\label{jan0610} 
\bK_\Z(X)\R/\Z^{-i-1}  \to \pi_{i}(\Diff(\bK(X)))\ .
\end{equation}
It is an isomorphism if Beilinson's  conjectures hold true for $X$.
\item We have a sequence 
\begin{eqnarray}
\begin{split}
\lefteqn{ \bK_\Z(X)\R/\Z^{-2} \to  \bK_\Z(X)^{-1} \xrightarrow{\beil} H (\DR_\Z(X))^{-1} \xrightarrow{a}
\widehat{\bK}(X)^{0}\xrightarrow{(I,R)}}&&\label{jan0102}\\[0.3cm]
&&\hspace{3cm}\to   \bK_\Z(X)^{0}  \times_{H (\DR_{\Z}(X))^{0}} Z^{0}(\cL\DR_\Z(X))    \to 0 \label{may2101}
\end{split}
\end{eqnarray}
which is exact except possibly at $ \bK_\Z(X)^{-1}$. 
It is exact   if Beilinson's  conjectures hold true for $X$.
\item We have an exact sequence
\begin{equation}\label{apr2320} 
\bK_\Z(X)^{-1} \to H (\DR_\Z(X))^{-1} \xrightarrow{a} \widehat{\bK}(X)^{0}_{flat} \xrightarrow{I}  \bK_\Z(X)^{0}  \xrightarrow{\beil} H (\DR_\Z(X))^{0}\ .
\end{equation}
\end{enumerate}
\end{prop}
\begin{proof}
Since the inclusion map $i$ in (\ref{nov2601})  induces an isomorphism  on 
$\pi_{i}$ for $i\le -1$  and a surjection on $\pi_0$, we obtain from the long exact sequence
associated to the pull-back that
$$\pi_{i}(\Diff(\bK(X)) )\cong  \bK_\Z(X)^{-i} \ , \quad i\le -1\ .$$

We let ${e}_X:\bS_{Mf}\to \bS_{Mf,\Z}$ be given by 
\begin{equation}\label{jul1103} 
 {e}_X(M):=M\times X \ .
\end{equation}
 {Write $\bK_{Mf,\Z} \R:=\bK_{Mf,\Z} \wedge M\R$.}
The map
\begin{equation}\label{mar0901}
\bK_{Mf,\Z} \R\to H(\DR_{Mf,\Z} )
\end{equation}
induced by the regulator $\beil$
induces a map of pull-back diagrams in $\Fun(\bS_{Mf},\Sp)$ from 
\begin{equation}\label{mar0901dedd2ed22334}
\begin{split}
\xymatrix{  {e}_X^{*}\Sigma^{-1}\bK_{Mf,\Z}\R/\Z   \ar[r] \ar[d] &0\ar[d]  \\
  {e}_X^{*}\bK_{Mf,\Z} \ar[r] &  {e}_X^{*}\bK_{Mf,\Z}\R 
}
\end{split}
\end{equation}
to 
\begin{equation}\label{mar0903}
\begin{split}
\xymatrix{
\Diff(\bK(X)) \ar[r]\ar[d]& {e}^{*}_XH( {\sigma}^{\ge 0}\cL\DR_{Mf,\Z} )\ar[d] \\
 {e}_X^{*}\bK_{Mf,\Z} \ar[r]& {e}_X^{*}H(\cL\DR_{Mf,\Z} )\ .
}
\end{split}
\end{equation} 
The resulting map of left upper corners  gives the map \eqref{jan0610} on homotopy groups.

 {We now assume that $X$ satisfies Beilinson's   conjecture \ref{mar0902}. We want to prove that \eqref{jan0610} is an isomorphism  for $i\geq 1$.}
We first study the map between the lower right corners  {of diagrams \eqref{mar0901dedd2ed22334} and \eqref{mar0903}}. Since 
the corresponding sheaves of spectra are homotopy invariant, we can  use the equivalences
$$
 {e}_X^{*}\bK_{Mf,\Z}\R\simeq \underline{\bK_{\Z}(X)\R}\ , \quad 
 {e}_{X}^{*}H(\cL\DR_{Mf,\Z})\simeq \underline{H(\DR_{\Z}(X))}\ .
$$
For a manifold $M$, we consider the morphism between the  {(right half plane)} Atiyah-Hirzebruch spectral sequences for $\bK_\Z(X)\R^*(M)$ and $H(\DR_\Z(X))^{*}(M)$    induced by \eqref{mar0901}.
 {By our assumption,} the map of $E_2$-terms is an isomorphism on $E_2^{pq}$ for all $q\leq -2$ and injective on $E_2^{0,-1}$.
Here is a schematic picture of these spectral sequences:
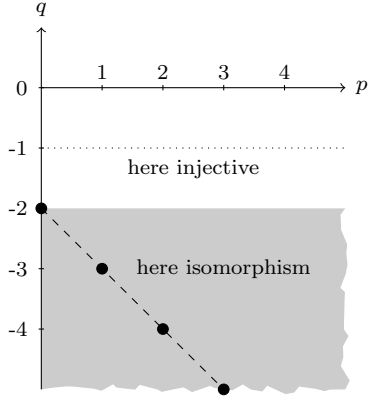
\begin{center}
\begin{tikzpicture}[y=.8cm, x=.8cm,font=\scriptsize]
	\draw[dotted](0,-1) -- (5,-1);
	\node[below] at (2.5,-1) {here injective};
	\fill [fill=black!20!white,decoration={random steps,segment length=4pt,amplitude=2pt}] decorate { (5,-2) -- (5,-5) -- (0,-5)} -- (0,-2)-- cycle;
	\node at (3,-3) {here  isomorphism};
	\draw[dashed](0,-2)--(3,-5) ;
	\foreach \z in {0,...,3} 
		\draw[fill] (\z,-2-\z) circle (2pt);
	\draw[thin,->] (0,0) -- coordinate (x axis mid) (5,0) node[right]{$p$};
    	\draw[->] (0,-5) -- coordinate (y axis mid) (0,1) node[above]{$q$};
    	\foreach \x in {1,...,4}
     		\draw (\x,1pt) -- (\x,-1pt)
			node[anchor=south] {\x};
    	\foreach \y in {-4,...,0}
     		\draw (1pt,\y) -- (-1pt,\y) 
     			node[anchor=east] {\y}; 
\end{tikzpicture}
\end{center}
The induced map is injective  {on} the dotted {line} and an isomorphism  {in the shaded area}. The dashed diagonal indicates the contributions in cohomological degree $-2$.
 {The terms $E_{2}^{pq}$ of both spectral sequences vanish for $p>\dim(M)$. Hence, if $M$ is finite dimensional, the spectral sequences converge, and we deduce that} 
$$ \bK_\Z(X)\R^{-i}(M)  \to  H(\DR_{\Z}(X))^{-i}(M)$$
is an isomorphism for $i\geq 2$ and injective for $i=1$.
 {If $M$ is infinite dimensional, we get the same result arguing component-wise.}

The map from to  \eqref{mar0901dedd2ed22334} to \eqref{mar0903} induces a morphism between the associated long exact sequences in homotopy.
By an application of the Five Lemma we see that \eqref{jan0610} is an isomorphism  for $i\geq 1$.
 
The exactness of \eqref{jan0102} is simply a part of the long exact sequence of homotopy groups associated with the pull-back diagram \eqref{mar0903} where we 
{have used the equivalence $\DR_\Z(X) \xrightarrow{\sim} \cL\DR_\Z(X)$ and we}
have replaced $\pi_1(\Diff(\bK(X))) $ at the left end by $ \bK_\Z(X)\R/\Z^{-2}$ using \eqref{jan0610}. Hence it is exact except possibly at $ \bK_\Z(X)^{-1}$. By the above, it is exact  everywhere if we assume  Beilinson's conjecture for $X$.

The exactness of \eqref{apr2320} follows from the unconditional part of exactness of \eqref{jan0102} by pull-back along the 
 {injective} 
transformation
\[
\ker\left( \bK_\Z(X)^{0} \xrightarrow{\beil} H (\DR_\Z(X))^{0}\right) \xrightarrow{x\mapsto (x,0)}  \bK_{\Z}(X)^{0} \times_{H (\DR_{\Z}(X))^{0}} Z^{0}(\cL\DR_{\Z}(X)).
\]
\end{proof}

\begin{ex}\label{dklqwdqwdqwdqwdqwd}
 We specialize the sequence \eqref{may2101} to
 the case $M=*$ and $X=\Spec(R)$ for a number ring $R$.  Let $\operatorname{Cl}(R)$ denote the class group of $R$. There is an isomorphism $K_{0}(R) \xrightarrow{\rk\oplus \operatorname{cl}} \Z \oplus \operatorname{Cl}(R)$ whose first component is the rank homomorphism.
Combining this and the Examples \ref{wklqdkjkqwdjqwldwqdqdwqdwqd} and \ref{dkqjwdqwdqwdwqdwqdqdwqd}, 
 we get an exact sequence
\begin{equation}\label{hdjhqwkdwqdwdwqdwqdwq}
  0\to \Z^{r_{1}+r_{2}-1}\xrightarrow{i} \R^{r_{1}+r_{2}}\stackrel{a}{\to}\widehat{\bK}(X)^{0}(*)\xrightarrow{ {\rk}\oplus \mathrm{cl}}  \Z\oplus \mathrm{Cl}(R)\to 0\ . 
 \end{equation}
Here $r_{1}$ and $r_{2}$ are as in Example \ref{wklqdkjkqwdjqwldwqdqdwqdwqd}. 
The group $\Z^{r_{1}+r_{2}-1}$ is the image of the group of units $R^{*}$  {in} $R$ under the  regulator 
$$R^{*}\cong \ K_{1}(R)\xrightarrow{\beil} H^{-1}(\DR_{\Z}(X))\cong \R^{r_{1}+r_{2}}\ .$$
Its rank is determined by Dirichlet's theorem. 
In particular,  \eqref{apr2320} gives an exact sequence
\[
0 \to \R^{r_{1}+r_{2}}/{i(\Z^{r_{1}+r_{2}-1})} \to \widehat{\bK}(X)_{flat}^{0}(*) \to \operatorname{Cl}(R) \to 0.
\]
\end{ex}

\subsection{The geometric cycle map}\label{jan1970}
Let $M$ be a smooth manifold and $X$   a  scheme in $\bReg_{\Z}$. 
A bundle $V\in i\Vect_{Mf,\Z}(M\times X)$ 
gives rise to a class $$\cycl(V)\in \pi_0(\bK_{Mf,\Z}(M\times X))\cong {\bK}_\Z(X)^{0}(M)\ .$$ 
The natural transformation $\cycl$ is called the topological cycle map, see \eqref{jan0660} for details.  

If $g^{V}$ is a local geometry on $V$, so that $(V,g^{V})$ is a geometric bundle (see Definition~\ref{jul0850}), then we have
the characteristic form $\omega(g^{V})\in Z^{0}(\cL\DR_{Mf,\Z}(M\times X))$,  defined in \ref{may1760n}, 
which represents the class of the regulator $$\beil(\cycl(V))\in H^{0}(\cL\DR_{Mf,\Z} (M\times X))\cong H(\DR_\Z(X))^{0}(M).$$  
The main result   of the present  subsection is the construction of a geometric cycle map  
$\hcycl$ which sends $(V,g^{V})$ to the differential algebraic $K$-theory class $$\hcycl(V,g^{V})\in \widehat{\bK}(X)^{0}(M)$$
 such that
  \begin{equation}\label{jul1075}R(\hcycl(V,g^{V}))=\omega(g^{V})\ , \quad  I(\hcycl(V,g^{V}))=\cycl(V)\ . \end{equation}

 \begin{rem}
 Note that in order to define  $\hcycl(V,g^{V})\in \widehat{\bK}(X)^{0}(M)$, it is not sufficient to give a pair of classes $\omega(g^{V})\in Z^{0}(\cL\DR_{Mf,\Z}(M\times X))$ and
 $\cycl(V)\in {\bK}_\Z(X)^{0}(M)$ such that their images in $H^{0}(\cL\DR_{Mf,\Z}(M\times X))$
 coincide. The square \eqref{nov2601} does not induce a pull-back square on the level of $\pi_{0}$.
 \end{rem}

  \begin{rem}
 A first version of topological and geometric cycle maps for number rings  has  been introduced in \cite[Sec.~3.6]{bg}.  The approach presented in the present paper is simpler.
 Moreover, we get a geometric cycle map which is functorial in the scheme $X$. Restricted to number
 rings, we get a functorial geometric cycle map for number rings. This improves the result of \cite{bg}. See, for example, the discussion in \cite[Sec.~3.11]{bg}.
\end{rem}   
    
We start with  a precise description of the topological cycle map. We consider the functor
$$ \pi_{0}( i\Vect_{Mf,\Z}): \Mf^{op}\times  {\bReg}_{\Z}^{op}\to \Mon(\Set)\ ,$$
 see Subsection \ref{jul1060}. 
 \begin{ddd}
The natural map of monoids
$$\Nerve( i\Vect_{Mf,\Z}) \to \Omega B(\Nerve(i\Vect_{Mf,\Z})) \simeq \Omega^{\infty}   K(\Vect_{Mf,\Z}) \to \Omega^{\infty}  \bK_{Mf,\Z} $$  induces the topological cycle map
\begin{equation}\label{jan0660}\cycl: \pi_{0}( i\Vect_{Mf,\Z}) \to \pi_{0}( \Omega^{\infty}  \bK_{Mf,\Z})\cong \pi_0(\bK_{Mf,\Z})\end{equation}
as a map in $\Fun(\bS_{Mf,\Z},\Nerve(\Mon(\Set)))$.
\end{ddd}
In order to define the geometric cycle map, we extend diagram  \eqref{sep1901}, used in the construction of the regulator {in} Definition \ref{jul08161}, to the commuting diagram in $\Fun(\bS_{Mf,\Z},\Sp)$
$$
\xymatrix@R+0.2cm{
K(i\Vect_{Mf,\Z}^{geom})\ar[d]\ar[drr]\ar[r]^-{K(\omega)}_-{{\text{\eqref{jul0710}}}}  & K(Z^{0}(\cL\DR_{Mf,\Z}))\ar[r]^{ {\simeq}}_{{\text{\eqref{jul0720}}}} &
H(Z^{0}(\cL\DR_{Mf,\Z}))\ar[r] &H(\sigma^{\ge 0}\cL\DR_{Mf,\Z})\ar[d] \\
K(i\Vect_{Mf,\Z})\ar[d]\ar[r]& \bar \bs K(i\Vect_{Mf,\Z}) &\bar \bs K(i\Vect_{Mf,\Z}^{geom})\ar[l]_{ {\simeq}}\ar[r]^{\bar \bs(r(\omega))}&\bar \bs H(\cL\DR_{Mf,\Z})\\
\bK_{Mf,\Z}\ar[rrr]^{\beil}&&&H(\DR_{Mf,\Z}) .\ar[u]^{ {\simeq}}
}
$$
Note that  the datum of a commuting square in an $\infty$-category includes  that of fillers. 
By the universal property of the pull-back defining $\Diff(\bK)$, the outer square defines a canonical map
\begin{equation}\label{jul1070}K(i\Vect_{Mf,\Z}^{geom})\to  \Diff(\bK)\end{equation} 
of diagrams of spectra.
\begin{ddd}\label{nov2050neu}

We define the geometric cycle map
$$\hcycl:\pi_0(i\Vect^{geom}_{Mf,\Z})\to \widehat \bK^{0}$$
as the transformation in $\Fun(\bS_{Mf,\Z},\Nerve(\Mon(\Set)))$ given by  {the} composition
$$
\pi_0(i\Vect^{geom}_{Mf,\Z})\to \pi_0(K(i\Vect_{Mf,\Z}^{geom})) \xrightarrow{\eqref{jul1070}}  \pi_0(\Diff(\bK))\ .
$$
\end{ddd}

It is clear from the construction that this geometric cycle map satisfies the conditions \eqref{jul1075}.
 \begin{ex}\label{diklqwdqwdqwdqwdq}
 We continue Example \ref{dklqwdqwdqwdqwdqwd}. We consider the bundle  $V:=\cO_{X}$ on $X:=\Spec(R)$.
 The standard metric $\|.\|$ on $\C$ induces a geometry $g_{\|.\|}^{V}$. We therefore get a class
 $$\hcycl(V,g_{\|.\|}^{V})\in \widehat \bK^{0}(X)(*)\ .$$
 We have $ {\rk}(\hcycl(V,g_{\|.\|}^{V}))=1$ and $\mathrm{cl}(\hcycl(V,g_{\|.\|}^{V}))=0$. 
 For $\lambda\in (0,\infty)$ we can define a new scaled metric $\lambda g^{V}_{\|.\|}$.
 Then we get a class
 \begin{equation}\label{gdjhwegdhjwedwedwedewd}
\hcycl(V,\lambda g_{\|.\|}^{V})-\hcycl(V,g_{\|.\|}^{V})\in \widehat \bK^{0}_{flat}(X)(*).
\end{equation}
{We will describe this class} more explicitly in \ref{dqwdlqwdqwdwqdwdwqdwqdwqdqwdwqdqwdqweqe}.
  \end{ex}

\subsection{Homotopy formulas }\label{jan1965}

\subsubsection{The manifold direction}
\label{sec:homotopy-mfd}
 
Let $I:=[0,1]$ be the unit interval. We consider a  smooth manifold $M\in \Mf$  and   a smooth   variety  $X\in \Sm_{\C}$   over $\C$.
As observed in the proof of Lemma \ref{dez1405}, the integral
$$\int_{I}:A^{*}(I\times M\times X)\to A^{*-1}(M\times X)$$   
induces a morphism 
 $$\int_{I}:\DR_{Mf,\C}(I\times M\times X)\to \DR_{Mf,\C}(M\times X)[-1]\ .$$
Explicitly, if $(\omega_{\R}\oplus\omega, \tilde\omega)$ is an element in the cone, then 
\begin{equation}\label{may2102}
\int_I(\omega_{\R}\oplus\omega, \tilde\omega) = (\int_I\omega_{\R}\oplus\int_I\omega, -\int_I\tilde\omega)\ .
\end{equation}
Let
$i_{0},i_{1}:M\to {I}\times M$ be the inclusions corresponding to the end points of the interval. By Stokes' theorem,  for $\omega\in \DR_{Mf,\C}(I\times M\times X)$ we then have the identity
\begin{equation}\label{jan0101}
\int_{I} d\omega = i_{1}^{*}\omega -i_{0}^{*} \omega - d \int_{I}\omega
\end{equation}
in $\DR_{Mf,\C}(M\times X)$.
A similar formula holds true for $X\in  {\bReg}_{\Z}$ and $\omega\in \DR_{Mf,\Z}(I\times M\times X)$.



A typical feature of differential cohomology is a homotopy formula in the manifold direction \cite[Eq.~(1)]{MR2608479}, \cite[Thm.~3.6]{Bunke:2013aa}. The case of differential algebraic $K$-theory is slightly more complicated since the curvature takes values in a \v{C}echification $\cL\DR_{Mf,\Z}$. We only state a homotopy formula in the arithmetic case and for classes whose curvature belongs to the image of the inclusion $\DR_{Mf,\Z}\to \cL \DR_{Mf,\Z}$.

Assume that $M\in \Mf$ is a smooth manifold and  $X$ is a scheme in $\bReg_{\Z}$. Recall that we denote by $a$ the natural map  $\cL\DR^{-1}_{Mf,\Z}(M\times X) /\im(d) \xrightarrow{\eqref{mar1401}} \widehat\bK^0(M\times X)$. \begin{lem}\label{jan0605}
If $\hat x\in \widehat{\bK}^{0}( {I}\times M\times X)$ and $R(\hat x)$ belongs to the image of the canonical map 
$$Z^{0}(\DR_{Mf,\Z}(I\times M\times X))\to  Z^{0}(\cL\DR_{Mf,\Z}(I\times M\times X))\ ,$$ 
then  we have
$$i_{1}^{*}(\hat x)-i_{0}^{*}(\hat x)=a(\int_{I} R(\hat x))\ .$$ 
\end{lem}
\begin{proof}
 {Let $\pr_{M\times X}\colon I\times M \times X \to M\times X$ denote the projection.}
 Since $\bK^{0}$ is homotopy invariant in the manifold direction, 
there exists a class $y\in \bK^{0}(M\times X)$ such that $\pr_{M\times X}^{*}(y)=I(x)$.
We can choose a form $\beta\in Z^{0}(\DR_{Mf,\Z}(M\times X))$ such that
$[\beta]=\beil(y)$. By the exactness of \eqref{may2101} at the right end, we can choose a class
$\hat y\in \widehat{\bK}^{0}(M\times X)$ such that $I(\hat y)=y$ and $R(\hat y)=\beta$.
{As $H(\DR_{Mf,\Z})^{0}$ is also homotopy invariant in the manifold direction, the cocycle} 
$${R(\hat x)-\pr^{*}_{M\times X}(R(\hat y))=} R(\hat x)-\pr^{*}_{M\times X}(\beta)\in Z^{0}(\DR_{Mf,\Z}(I\times M\times X))$$
 is exact.  Thus there exists a form
$\alpha\in \DR_{Mf,\Z}^{-1} (I\times M\times X) $ such that
$d\alpha=R(\hat x)-\pr^{*}_{M\times X}(R(\hat y))$. {Using $R\circ a = d$ and \eqref{may2101} again, there is some $\gamma\in Z^{-1}(\DR_{Mf,\Z}(I\times M\times X))$ satisfying}
\begin{equation}\label{may2701n}
\hat x-\pr^{*}_{M\times X}(\hat y)-a(\alpha)=a(\gamma).
\end{equation}
{Combining}
 $$d(\alpha+\gamma)= {d\alpha=}R(\hat x)-\pr_{M\times X}R(\hat y)\ , $$  
\eqref{may2701n},  and \eqref{jan0101}, we get
\[
i_{1}^{*}(\hat x)-i_{0}^{*} (\hat x) \overset{\eqref{may2701n}}{=} a(i_{1}^{*}(\alpha+\gamma)-i_{0}^{*}(\alpha+\gamma))\overset{\eqref{jan0101}}{=}a(\int_{I} R(\hat x) ) \ . \qedhere
\]
\end{proof}


\begin{ex}\label{wqlidkqwdqwdqwdwqd}  
As an application, we discuss how the geometric cycle class changes under scaling the metric.
   Thus we let $V$ be a locally free sheaf of $\pr_{X}^{*}\cO_{X}$-modules of finite rank on $M\times X$.   {We assume that $V$ admits a good geometry} $g^{V}=(\nabla^{II},h)$
 such that the total connection $\nabla$ is flat.
Let $\lambda\in (0,\infty)\subset \R$.
\begin{lem}\label{ejdkhj23d} We have
\begin{equation}\label{apr2410}
\hcycl(V,(\nabla^{II},\lambda h)) - \hcycl(V,(\nabla^{II},h)) = a(\beta)
\end{equation} 
where $\beta\in\DR_{Mf, {\Z}}^{-1} (M\times X)$ has {only one non-trivial factor, namely} 
$$(0 \oplus 0,\frac{1}{2}\dim(V)\log(\lambda)) \in \DR_{Mf,{\Z}} {(1)}^{-1} (M\times X) \ .$$
\end{lem}
\begin{proof}
 Consider the bundle $\pr_{M\times X}^{*}V$ on $I\times M \times X$.
We choose a smooth positive function $\rho$ on ${I}$ with $\rho(0)=1$, $ \rho(1)=\lambda$. Then    $\hat g^V= (\pr_{M\times X}^{*}\nabla^{II},\rho \cdot\pr_{M\times X}^{*}h)$ is a good geometry on $\pr_{M\times X}^{*}V$. We let $\hat x:=\hcycl(\pr_{M\times X}^{*}V,\hat g^{V})$.
Denoting the total connection on $\pr_{M\times X}^*(V)$ by $\hat\nabla$, we find that its adjoint is simply given by 
\[
\hat\nabla^*=\pr_{M\times X}^*(\nabla^{*_h}) + d\log(\rho).
\]
The curvature of $\hat\nabla^u$ is the pull-back of the curvature of $\nabla^{u_h}$, hence
\begin{equation*}
\int_I \ch_{2p}(\hat\nabla^u) = 0.
\end{equation*}
Write $\omega := \nabla^{u_h} - \nabla$, $\hat\omega := \hat\nabla^u - \hat\nabla = \pr_{M\times X}^*\omega + \frac{1}{2}d\log(\rho)$. Since $\hat \nabla$ is flat,  we obtain the relation 
\[
\tilde \ch_{2p-1}(\hat \nabla^{u},\hat \nabla) =-\frac{4^pp!}{2(2p)!}\Tr (\hat \omega^{2p-1}).
\] We have
$$\hat \omega^{2p-1}=\pr_{M\times X}^{*}\omega^{2p-1}+\frac{1}{2} d\log(\rho) \wedge \pr_{M\times X}^{*}\omega^{2p-2}\ .$$ Since $\Tr(\omega^{2p-2})=0$ for $p>1$, we get $\tilde \ch_{2p-1}(\hat \nabla^{u},\hat \nabla)=\pr^{*}_{M\times X} \tilde \ch_{2p-1}(\nabla^{u},\nabla)$ for $p> 1$. 
Furthermore, $\tilde \ch_{1}(\hat \nabla^{u},\hat \nabla)=\pr^{*}_{M\times X} \tilde \ch_{1}(\nabla^{u},\nabla)-\frac{\dim(V)}{2} d\log (\rho)$.
We finally get
\begin{equation*}
\int_I \tilde\ch_{2p-1}(\hat\nabla^u,\hat\nabla) = 
\begin{cases}
-\frac{1}{2}\log(\lambda) \dim V & \text{if } p=1\\
0 & \text{else.}
\end{cases}
\end{equation*}
Using \eqref{may2102}
we see that $\int_{I}R(\hat x)(p)=0$ for $p\not=1$ and
$\int_{I}R(\hat x)(1)=(0  \oplus  0,\frac{1}{2}\dim(V)\log(\lambda))$.
By Lemma \ref{jan0605} this implies \eqref{apr2410}.
\end{proof}
\end{ex}

\begin{ex}\label{weflkwelfwef}
The homotopy formula in the manifold direction can be used in order to calculate the Beilinson regulator explicitly in some cases. Let $X:=\Spec(R)$ for a number ring $R$. Then $X(\C)$ is a zero-dimensional manifold whose points are the embeddings $\sigma:R\hookrightarrow \C$. For every $\sigma\in X(\C)$ we have an evaluation map $\ev_{\sigma}:\DR_{\Z}(X)\to \DR_{\C}(\Spec(\C))$. By Example \ref{kdjqlkwdqwdqwdwqdwqdwqdwqd}
we have an isomorphism 
\begin{equation}\label{swsqwsqwsqsss12s21s12s}
H^{-1}(\DR_{\C}( \Spec(\C)))\cong H^{-1}(\DR_{\C}{(1)}( \Spec(\C))) \cong \C/i\R\cong \R\ .
\end{equation} 
\begin{lem}\label{likwefefwefewfewfewfwf}
  The composition of the evaluation and the Beilinson regulator
  $$ R^{*}\cong  K_{1}(R)\xrightarrow{\beil} H^{-1}(\DR_{\Z}(X))  \xrightarrow{\ev_{\sigma}}  H^{-1}(\DR_{\C}(\Spec(\C)))    \stackrel{\eqref{swsqwsqwsqsss12s21s12s}}{\cong} \R$$
maps
$u\in R^{*}$ to
$-\log\|\sigma(u)\|\in \R$.
\end{lem}
\begin{proof}
We consider  the trivial bundle $V:=\cO_{X}$.
It has a canonical metric $h^{V}_{\|.\|}$ induced by the standard metric of $\C$.
It further has a canonical connection $\nabla^{II}$ and therefore a geometry $(\nabla^{II},h^{V}_{\|.\|})$.  Multiplication by $u$ induces an isomorphism
between geometric bundles $$\phi_{u}:(V,(  \nabla^{II},h^{V}_{\|.\|}))\xrightarrow{\cong} (V,(\nabla^{II},|u|^{-2}h^{V}_{ \|.\|}))\ ,$$
where the function $|u|: X(\C)\to (0,\infty)$ is given by  $ \sigma \mapsto \|\sigma(u)\|$.
Let $\rho:I\times X(\C)\to (0,\infty)$ be a smooth  function such that $\rho(t,\sigma)=1$ for $t$ near $0$ and
$\rho(t,\sigma)=\|\sigma(u)\|^{-2}$ near $t=1$. We define the metric $\hat h^{V}:=\rho\ \pr^{*}_{X}h^{V}_{\|.\|}$
and get the geometry $\hat g^{V}:=(\pr_{X}^{*}\nabla^{II},\hat h^{V})$ on $\pr_{X}^{*}V$.
 Using $\phi_{u}$ we can descend the geometric  bundle
$(\pr_{X}^{*}V,\hat g^{V})$     to a geometric bundle
$(\tilde V_{u},\tilde g_{u})$ on $S^{1}\times X$.

The difference 
$\cycl(\tilde V_{u})-\cycl(\tilde V_{1})\in \bK_{\C}^{0}(X)(S^{1})$
corresponds to $u\in R^{*}\cong   K_{1}(R)\cong \bK_{\Z}(X)^{-1}(*)$ under suspension.
Hence, by {a calculation as in} Example \ref{wqlidkqwdqwdqwdwqd}, we get 
$$\ev_{\sigma}(\beil(u))=\left[\int_{S^{1}\times X/X} R(\hcycl(\tilde V_{u},\tilde g_{u})-\hcycl(\tilde V_{1},\tilde g_{1}))(1)\right]= -\log\|\sigma(u)\| \ .$$ 
\end{proof}
\end{ex}
 
\begin{ex}\label{dqwdlqwdqwdwqdwdwqdwqdwqdqwdwqdqwdqweqe}
We continue with Example \ref{diklqwdqwdqwdqwdq}. Note that the identification
$H^{-1}(\DR_{\Z}(X))\cong \R^{r_{1}+r_{2}}$ is given by the tuple of evaluations
$(\ev_{\sigma})_{ \sigma\in X(\C)/\Gal(\C/\R)}$.

 The flat class \eqref{gdjhwegdhjwedwedwedewd}
is given by $a({\operatorname{diag}}(\lambda))$ where
$$\operatorname{diag}(\lambda) :=\frac{1}{2}\log(\lambda)\underbrace{(1,\dots,1)}_{r_{1}+r_{2}}\in \R^{r_{1}+r_{2}}\ .$$
In this picture the image of the map $i$ in \eqref{hdjhqwkdwqdwdwqdwqdwq} is given by the vectors
$$(\log \|\sigma^{\R}_{1}(u)\|,\dots,\log \|\sigma^{\R}_{r_{1}}(u)\|,\dots,\log \|\sigma^{\C}_{ 1}(u)\|,\dots, \log \|\sigma_{r_{2}}^{\C}( u)\|)\ , \quad u\in R^{*}\ ,$$
where $(\sigma^{\R}_{i})_{i}$ and $(\sigma^{\C}_{j})_{j}$ run over the real embeddings $R\hookrightarrow \R$ and
representatives of pairs of  complex conjugate complex embeddings $R\hookrightarrow \C$. 
\end{ex}

\subsubsection{The algebraic direction}\label{jan0110}

  {A}bsolute Hodge cohomology is homotopy invariant in the algebraic direction   in the sense that the natural projection induces a 
  quasi-isomorphism
\[
\DR_{Mf,\C}(M\times X) \xrightarrow[ {\simeq}]{\pr_{M\times X}^*} \DR_{Mf,\C}(M\times X\times \bbA^1_{\C})\ :
\]  
 By descent in the manifold direction, it suffices to prove this for $M$ a point. From the homotopy invariance of de Rham cohomology we know that the map $A_{\log,\R}(X) \to A_{\log,\R}(X\times \bbA^{1}_{\C})$ is a quasi-isomorphism. By Deligne \cite{HodgeII} it is strict with respect to the d\'ecalage of the weight filtration. Thus $\hat \cW_{k}A_{\log,\R}(X) \to \hat\cW_{k}A_{\log,\R}(X\times \bbA^{1}_{\C})$ is a quasi-isomorphism, too. Similarly, one handles the other components of the cone defining $\DR_{Mf,\C}$.
 
In particular, the two inclusions $ M \times X \hookrightarrow M \times X \times \bbA^1_{\C}$ given by the points $0,1\in \bbA^1_{\C}(\C)$ induce the same map in absolute Hodge cohomology. We expect that there exists an integration operator $\int_{II}$ in the algebraic direction which satisfies a formula similar to \eqref{jan0101}. However, the actual construction of such an operator seems to be quite complicated, and we content ourselves with the analogous formula for $\P^1_{\C}$. This is enough for our purposes in the current paper.

Again we state the homotopy formula in the arithmetic situation. We consider $M\in \Mf$ and $X\in  {\bReg}_{\Z}$.
 We let $[x_{0}:x_{1}]$ be homogeneous coordinates of $\P_{\Z}^{1}$ and  $f_{i}:M\times X  \hookrightarrow M\times X \times \P_{\Z}^{1}$ be the inclusions  determined by $x_{i}=0$, $i=0,1$.
%
%
Using the homotopy  {$H_{1}$ from} \eqref{may2103} we define  
\begin{equation}\label{mar2001}
\int_{II} \omega:= \cL(H_{{1}}) \colon \cL\DR_{Mf,\Z}^{ {0}}(M\times X\times \P^{1}_{\Z})\to \cL\DR^{-1}_{Mf,\Z}(M\times X) \ .
\end{equation} 
Then \eqref{may1730} becomes
\begin{equation}\label{mar1902}
 f_0^*\omega - f_1^*\omega = \int_{II} d\omega + d\int_{II}\omega\ . 
\end{equation}
\begin{prop}\label{jan0111}
We consider a class $\hat x\in \widehat{\bK}^{0}(M\times X\times  {\P^1_\Z}) $.
There exists a class $\hat y\in \widehat{\bK}^0(M\times X\times\P^1_{\Z})$ such that $I(\hat y)=I(\hat x)$ and $\int_{II}R(\hat y) =0$. For any such class $\hat y$,
we have
$$
(f_{0}^{*}-f_{1}^{*})\hat x   - (f_{0}^{*}-f_{1}^{*})\hat y=a(\int_{II} R(\hat x))\ .
$$
\end{prop}
\begin{proof}
Let $\eta\in Z^0(\DR(M\times X\times \P^1_\Z))$ be a form representing $\beil(I(\hat x))$. 
Recall the $\Gal(\C/\R)$-equivariant operator $Q$ from \eqref{may2702n}. 
{According to \eqref{may1730},} we have   $\int_{II}Q(\eta) =0$. 
By the exactness of \eqref{may2101} there exists $\hat y \in \widehat{\bK}^0(M\times X\times\P^1_{\Z}) $ such that $I(\hat y)=I(\hat x)$ and $R(\hat y) = Q(\eta)$. Hence $\hat y$ has the desired properties.

By the exactness 
of \eqref{mar1402} at $\widehat{\bK}^{0}$
we can then find
an element $\omega\in \cL\DR^{-1}_{Mf,\Z}(M\times X\times  {\P^{1}_\Z})$ such that
$ \hat x-\hat y=a(\omega)$.  
It follows that
\begin{equation*}
{(f_{0}^{*}-f_{1}^{*})}\hat x-{(f_{0}^{*}-f_{1}^{*})}\hat y = a({f_{0}^{*}\omega-f_{1}^{*}}\omega)  {\overset{\eqref{mar1902}}{=} a(\int_{II} d\omega)} = a(\int_{II}R(\hat x)) 
\end{equation*}
since  $R(\hat y)$ is annihilated by $\int_{II}$.  \end{proof}

\subsection{Relation to arithmetic $K$-theory}
\label{sec:Relation-to-arithmetic-K}
 
 Predating the development of differential generalized cohomology theories in the case of smooth manifolds,
analogues have been considered in the arithemtic context. The first examples were the so called arithmetic Chow groups and an arithmetic $K_{0}$-group introduced by 
Gillet and Soul\'e \cite{MR1038362,GS2} for schemes in $\Reg_{\Z}$  which are flat and quasi-projective over $\Z$ 
(or, more generally, over an arithmetic ring). Let $X$ be such a scheme. Gillet and Soul\'e define the arithmetic 
$K_{0}$-group $\widehat{K}^{GS}_{0}(X)$ \cite[6.1]{GS2} in terms of cycles and relations: A hermitian vector bundle 
$\bar E$ on $X$ is a pair $(E,h)$ consisting of a vector bundle $E$ on $X$ and a hermitian metric $h$ on the complex
 vector bundle $E(\C)\to X(\C)$ which is invariant under the complex conjugation. Generators for $\widehat{K}^{GS}_{0}(X)$
 are triples $(E,h,\eta)$ consisting of a hermitian vector bundle $(E,h)$ and a differential form 
$ {\eta=(\eta(p))_{p\geq 0} \in \prod_{p\geq 0}} \widetilde{A}^{p,p}(X)$ where 
\[
\widetilde{A}^{p,p}(X) := \left(\left(A^{p,p}(X(\C)) \cap A^{2p}_{\R}(X(\C))\right)/(\im(\partial) + \im(\bar\partial))\right)^{\Gal(\C/\R)}
\]
Relations come from short exact sequences and involve a Bott-Chern secondary form.

Recall that, given a hermitian vector bundle $(E,h)$ on $X$, there is a unique unitary connection $\nabla_{h}$ on $E(\C)$ which is compatible with the holomorphic structure $\bar\partial$. 
Assume for simplicity that $X(\C)$ is compact. Then $(E,h,\nabla_{h})$ is a bundle on $\ast\times X$ with a good geometry in our sense.   
 
To every generator $(E,h,\eta)$ one can associate a class 
\(
\Psi(E,h,\eta) \in \widehat{\bK}^{0}(\ast\times X)
\)
as follows: For $\eta_{p} \in \widetilde{A}^{p,p}(X)$, the class
\[
 ( i^{p}(\partial - \bar\partial)\eta_{p} \oplus 2 i^{p}\partial\eta_{p}, -  i^{p}\eta_{p}) \in \DR_{\Z}(p+1)^{-1}(X)/\im(d)
\]
is well defined.
We set
\[
\Psi(E,h,\eta) := \hcycl(E,h,\nabla_{h}) + a\left(   \left(i^{p-1}(\partial - \bar\partial)\eta_{p-1} \oplus 2i^{p-1} \partial\eta_{p-1}, - i^{p-1} \eta_{p-1}\right)_{p\geq 0}\right) \in \widehat{\bK}^{0}(\ast\times X).
\]
Using the axiomatic characterization of Bott-Chern forms and a deformation over $\P^{1}$, similar as we will use it in the proof of Lott's relation,
 one can prove that $\Psi$ indeed  gives a homomorphism 
\[
\widehat{K}^{GS}_{0}(X) \to \widehat{\bK}^{0}(\ast\times X).
\]

 Recall from \eqref{ddhqwdhqwkjhdwkqdwqd} that the higher $K$-groups of $X$ are given by
 \[
\bK_{n}(X):=\pi_{n}(\bK_{\Z}(X)).
 \]
 Beilinson's regulator induces
 \[
\bK_{n}(X) \xrightarrow{\beil} H^{{-n}}(\DR_{\Z}(X))\cong \bigoplus_{p\geq 0} H^{2p-n}_{\mathrm{Hodge}}(X,\R(p)).
 \]
It was suggested by Deligne \cite[Rem.~5.4]{MR902592} and Soul\'e \cite[III.2.3.4]{MR1208731} that higher arithmetic $K$-groups $\widehat{K}_{n}^{arith}(X)$ for $n\geq 0$ should fit into long exact sequences
\[
\bK_{n+1}(X) \xrightarrow{\beil} H^{{-n-1}}(\DR_{\Z}(X)) \to \widehat{K}_{n}^{arith}(X)  \to \bK_{n}(X) \xrightarrow{\beil} H^{{-n}}(\DR_{\Z}(X)).
\]
Candidates for such groups have been constructed by Takeda \cite{MR2153537}  and Scholbach \cite{2012arXiv1205.3890S}. We now explain, how our differential algebraic $K$-theory provides another candidate for these groups:

We use  the inclusion $j:*\to S^{n}$  of the north pole  into the $n$-dimensional sphere in order to define reduced $n$-th cohomology groups.
We have isomorphisms 
\[
\bK_{n}(X)\cong \ker\left(j^{*}:\bK_{\Z}(X)^{0}(S^{n})\to \bK_{\Z}(X)^{0}(*)\right)
\]
and 
\[
H^{ {-}n}(\DR_{\Z}(X))\cong \ker\left(j^{*}:
H(\DR_{\Z}(X))^{0}(S^{n}) \to H(\DR_{\Z}(X))^{0}(*)\right)\ .
\]
If we set
$$\overline \bK_{n}(X):=\ker\left(j^{*}:\widehat{\bK}(X)\oben{0}_{flat}(S^{n})\to \widehat{\bK}(X)\oben{0}_{flat}(*)\right)\ ,$$
then this group fits into an exact sequence  
\begin{equation}\label{apr1401}
\bK_{n+1}(X)\stackrel{\beil}{\to } H^{{-n-1}}(\DR_{\Z}(X))\stackrel{a}{\to} \overline \bK_{n}(X) \stackrel{I}{\to}
\bK_{n}(X)\stackrel{\beil}{\to} H^{{-n}}(\DR_{\Z}(X))
\end{equation}
 {as desired.}

{We comment briefly on the relation to Takeda's construction.}
Based on the differential form level construction of the Beilinson regulator by {Burgos and Wang} \cite{MR1621424}, Takeda defines in \cite{MR2153537}
a group $\widehat{\bK}^{Takeda}_{n}(X)$ for proper $X$ together with a characteristic form map $\ch$.
The group
$\widehat{\bK}^{Takeda}_{n}(X)$ is the analogue of our
$$\widehat{\bK}^{BT}_{n}(X):=\ker\left(j^{*}:\widehat{\bK}(X)\oben{0}(S^{n})\to \widehat{\bK}(X)\oben{0}(\pt)\right)\ .$$
It is not the same because of a different choice of differential form data computing absolute Hodge  {cohomology}. Takeda's version of arithmetic algebraic $K$-theory is 
$$\overline{\bK}_{n}^{Takeda}(X):=\ker({\ch})\ .$$
It fits into an exact sequence which is  the analogue of \eqref{apr1401}. We expect  that there is a natural isomorphism
$$ \overline \bK_{n}(X)\cong \overline{\bK}_{n}^{Takeda}(X)\ .$$ 
For a proof one must construct a map relating these groups which is compatible with the exact sequences. By the Five Lemma this map is then automatically an isomorphism. At the moment, however, it is not obvious how to construct such a map.

\subsection{Rings of integers}\label{mar18001}

\subsubsection{A review of \cite{bg}}\label{jan0601}

In this subsection we review the construction of differential algebraic $K$-theory $\widehat{KR}^{0}$  as introduced in \cite[Sec. 2.2]{bg} by a specialization of the Hopkins-Singer construction.

Let $R$ be a number ring, i.e.~the ring of integers in a finite field extension $\Q\subseteq k$.
We consider 
the scheme $X= {\Spec}(R) \in \bReg_{\Z}$, which is of relative dimension 0 over $\Spec(\Z)$, and its connective algebraic $K$-theory spectrum  $KR:=\bK_\Z(X)\in
 \Sp$.  In order to define the differential algebraic $K$-theory $\widehat{KR}^{0}$ by the Hopkins-Singer construction  {as} explained in Subsection \ref{dlqkwdjlqdwqdwqdwqdwqdw24234234}, we choose the canonical differential data
 $(KR, {C},c)$. Here $ {C}^{*}:=\pi_{-*}(KR)\otimes \R$ is a chain complex of real vector spaces with trivial differentials  and
 $c:KR\to H( {C})$ is {up to equivalence} uniquely determined  by
 the property that it induces the canonical map 
 $\pi_{*}(KR)\to \pi_{*}(H({C}))\cong \pi_{*}(KR)\otimes \R$. Note the explicit calculation of $ {C}$ given in Example~\ref{dkqjwdqwdqwdwqdwqdqdwqd}.

The maps denoted by $R$ and $I$ in the diagram \eqref{r23rr2r23jkr3hkjh23r}
induce maps
$$R:\widehat{KR}^{0}\to Z^{0}(\Omega  {C})\ ,\quad I:\widehat{KR}^{0}\to  KR^{0}\ ,$$
and we have a map
$$a:\Omega  {C}^{-1}/\im(d)\to \widehat{KR}^{0}\ .$$
We shall use the exact sequence 
\begin{equation}\label{jan0507} 
KR^{-1} \to \Omega {C}^{-1}/\im(d) \stackrel{a}{\to} 
\widehat{KR}^{0} \stackrel{I}{\to}   KR^{0}\to 0
 \end{equation}
 {of \cite[Def.~2.2 (iii)]{bg}, which is  the analogue of \eqref{mar1402}.}

\subsubsection{Comparison of the two versions of differential algebraic $K$-theory}
\label{sec:comparison-of-diff-K}

We want to give the precise relation between the differential algebraic $K$-theory $\widehat{KR}^{0}$ introduced in \cite{bg} or Subsection \ref{jan0601}  and $\widehat{\bK}(X)^{0}$    in defined in  \eqref{jul1101}. We further use the notation 
 $ {e}^{*}_X\DR_{Mf,\Z} \in  {\PSh}_{\Ch}(\Mf)$ with $ {e}_X$ as in \eqref{jul1103}.

We must relate the two versions of forms $\Omega {C}$ and $ {e}_X^{*}\DR_{Mf,\Z}$ corresponding to $\widehat{KR}^{0}$ and $\widehat{\bK}(X)^{0 }$. We write elements 
$\omega \in  \DR_{Mf,\Z} (M\times X)$ in the form 
 \begin{equation}\label{jan0130}\omega=(\omega_{\R}(p)\oplus\omega(p),\tilde \omega(p))_{p\ge 0}\in \prod_{p\ge 0} \DR_{Mf,\Z} (p)(M\times X)\ .\end{equation}
 Here 
\begin{eqnarray*}
 \omega_{\R}(p)&\in& (2\pi i)^{p} A_{\R}(M\times X)[2p]\ ,\\
 \omega(p)&\in&  \cF^{p} A(M\times X)[2p]\ , \\
 \tilde \omega(p)&\in&  A(M\times X)[2p-1]\ ,\end{eqnarray*}
where we use the notation
 $ A(M\times X):=[ A(M\times X(\C))]^{{\Gal(\C/\R)}}$
 where ${\Gal(\C/\R)}$ acts on the set $X(\C)$ of complex points of $X$ and the differential  forms by complex conjugation. Furthermore,
 $\cF^{p} A(M\times X)=0$ for $p\ge 1$ and $\cF^{0} A(M\times X)= A(M\times X)$.
We define
\begin{equation}\label{sep2501}\Ree(\tilde \omega(p)):=(2\pi i)^{p}  {\mathfrak{Re}}\left((2\pi i)^{-p} \tilde \omega(p)\right)\ , \quad 
 \Imm(\tilde \omega(p)):=(2\pi i)^{p+1}  {\mathfrak{Re}}\left((2\pi i)^{-p-1} \tilde \omega(p)\right) ,
\end{equation}
where $\mathfrak{Re}$ denotes the usual real part.  
Note that for every complex point $\sigma\in X(\C)$, we have an evaluation
$$\ev_{\sigma}: A(M\times X)\to  A(M)\ .$$

 We will use the following results about Beilinson's regulator for $R$,  {which have already been stated in Example~\ref{dkqjwdqwdqwdwqdwqdqdwqd}}: Combining Borel's results \cite{MR0387496} concerning the ranks of the $K_i(R)$, $i\geq 0$, and the comparison between the regulators of Borel and Beilinson \cite{MR760999, MR944994, MR1869655},  {the map}
\[
\beil\colon  \pi_*(\bK_\Z(X)) \otimes \R \to H^{-*}(\DR_{\Z}(X))
\]
is an isomorphism in degrees $*\geq 2$, and injective (with well known image) in degrees 0 and 1.
In fact, the conditions 1.~-- 3.~below determine a subcomplex of $\DR_\Z(X)$  with trivial differential such that $C^{*}=\pi_{-*}(\bK_\Z(X)) \otimes \R $ is isomorphic to (the cohomology of) this subcomplex via $\beil$. The cohomology of $\DR_\Z(X)$ has been calculated explicitly in Example \ref{wklqdkjkqwdjqwldwqdqdwqdwqd}.

It follows that for any smooth manifold $M$, we get
a natural isomorphism 
\begin{equation}\label{jan0420}
\Psi:\Omega  {C}(M)\hookrightarrow   \DR_{Mf,\Z}(M\times X)\end{equation} onto  the subcomplex of forms
$\omega\in  \DR_{Mf,\Z}(M\times X)$ 
satisfying (using the components introduced in \eqref{jan0130})
\begin{enumerate}
\item $\omega(p)=0$, $\omega_{\R}(p)=0$, $\Ree(\tilde \omega(p))=0$ for $p\ge 2$,
\item  $\omega(1)=0$, $\omega_{\R}(1)=0$, $\Ree(\tilde \omega(1))=0$ and 
\begin{equation}\label{jan0510} 
\frac{1}{|X(\C)|}\sum_{\sigma\in X(\C)} \ev_{\sigma}\tilde \omega(1)=0\ ,
\end{equation}
\item $\tilde \omega(0)=0$ and there exists a form $x\in A_{\R}(M)$ such that 
$\ev_{\sigma}\omega(0)=\ev_{\sigma}\omega_{\R}(0)=x$ for all $\sigma\in X(\C)$.
\end{enumerate}
The isomorphism is normalized in the unique manner  such that
\begin{equation}\label{jan0210}
\begin{split}
\xymatrix{\underline{KR}\ar[d]^{{\simeq}}\ar[r]^{\rat} & H(\Omega  {C})\ar[d]^{\Psi}\\
 {e}_X^{*}\bK_{Mf,\Z}\ar[r]^-{\beil}&  {e}_X^{*}H(\DR_{Mf,\Z})}
\end{split}
\end{equation}
commutes. In fact, $\Psi$ is fixed by the evaluation of this diagram at $M=*$ using the fact that $\rat$  induces an isomorphism after {tensoring} its domain by $\R$.

\begin{ex}
{To motivate condition \eqref{jan0510}, consider a number ring $R$ and let $X:=\Spec(R)$.}
If $u\in R^{*}\cong K_{1}(R)$, then we have the well-known relation
$$\sum_{\sigma\in X(\C)} \ev_{\sigma}(\beil(u)) \overset{\text{Lemma \ref{likwefefwefewfewfewfwf}}}{=}-\sum_{\sigma\in X(\C)} \log\|\sigma(u)\|=0\ .$$
\end{ex}

\begin{lem}\label{jan0508}
The map
$$\Psi:H^{*}(\Omega  {C})\to  {e}_X^{*}H^{*}(\DR_{Mf,\Z})$$ is injective.
\end{lem}
\begin{proof}
Let $M$ be a smooth manifold. 
Let $\alpha\in Z^{*}(\Omega  {C}(M))$ and assume  {that}
$\Psi(\alpha)=d\omega$. We must show that there exists some $\hat \beta \in \Omega  {C}^{*-1}(M)$ such that 
$d\hat \beta=\alpha$.
We define
$\beta\in \DR_{Mf,\Z}^{*-1}(M\times X)$ as follows:
\begin{enumerate}
\item
For $p\ge 2$ we set $\tilde \beta(p):=\Imm(\tilde \omega(p))$ and $\beta(p):=0$, $\beta_{\R}(p):=0$.
\item For $p=1$ we define  $\beta(1):=0$, $\beta_{\R}(1):=0$ and 
$\tilde \beta(1)$ by 
$$\ev_{\sigma}\tilde \beta(1):= \ev_{\sigma}\Imm(\tilde \omega(1))-\frac{1}{|X(\C)|}\sum_{\sigma\in X(\C)}\ev_{\sigma}\Imm(\tilde \omega(1))$$
for all $\sigma\in X(\C)$.
\item 
For $p=0$ we set
$\beta_{\R}(0):=\omega_{\R}(0)$, $\beta(0):=\omega_{\R}(0)$ and $\tilde \beta(0):=0$.
\end{enumerate}
Then we have $d\beta=\Psi(\alpha)$ and  {since $\beta$ satisfies 1.~-- 3.~above}
$\beta=\Psi(\hat \beta)$ for some uniquely determined $\hat \beta \in \Omega  {C}^{*-1}(M)$. We conclude that
$d\hat \beta=\alpha$.
\end{proof}
In view of the Definition \ref{jan0503} of $\Diff(\bK)$, the   map of  pull-back  diagrams
$$
\xymatrix{
\underline{KR}\ar[r]^{\rat}\ar[dd]^{{\simeq} } & \ar[dd]^{\Psi}H(\Omega  {C}) & & H( {\sigma}^{\ge 0} \Omega  {C})\ar[ll]\ar[d]^{\simeq} \\
& & & H( \sigma^{\ge 0}\cL\Omega  {C})\ar[d]^{ \sigma^{\ge 0}\cL\Psi}\\
 {e}_X^{*}\bK_{Mf,\Z}\ar[r]^-{\beil} &  {e}_X^{* }H(\DR_{Mf,\Z})\ar[r]^{\simeq} &  {e}_X^{* }H(\cL\DR_{Mf,\Z}) &  {e}_X^{* }H(\sigma^{\ge 0}\cL\DR_{Mf,\Z})\ar[l]
}
$$
gives a morphism $\Diff^{0}(KR) \to \Diff(\bK(X)) $  in $\Fun(\bS_{Mf},\Sp)$ and hence a natural transformation of differential cohomology functors
\begin{equation}\label{mar1420}
\psi: \widehat{KR}^{0} \to \widehat{\bK}(X)^{0}\ .
\end{equation}
By construction, the diagram
\begin{equation}\label{jan0506}
\begin{split}
\xymatrix{
\Omega  {C}^{-1}\ar[r]^-{a}\ar[d]&\widehat{KR}^{0}\ar[d]^{\psi}\ar[r]^-{R}\ar@/^1cm/[rr]^{I}&Z^{0}(\Omega  {C})\ar[d]& KR^{0} \ar[d]^{\cong}\\
 {e}_X^{*}\cL\DR^{-1}_{Mf,\Z}\ar[r]^-{a} & \widehat{\bK}(X)^{0}\ar[r]^-{R}\ar@/_{1cm}/[rr]^{I}&Z^{0}( {e}_X^{*}\cL\DR_{Mf,\Z})&   \bK_\Z(X)^{0}, 
}
\end{split} 
\end{equation}
commutes, 
where the unlabelled vertical maps are induced by $\Psi$. 

\begin{lem}\label{jan0440}
For every manifold $M\in \Mf$ the map
$$\psi: \widehat{KR}^{0}(M) \to \widehat{\bK}(X)^{0}(M)$$
is injective. 
\end{lem}
\begin{proof}
 Compare the exact sequence \eqref{jan0102} with the analogous exact sequence for $\widehat{KR}^0$. The assertion follows by a simple diagram chase using Lemma \ref{jan0508}.
\end{proof}

\subsubsection{Comparison of cycle maps}\label{jan0415}

 {As in the previous subsections} we fix $X=\Spec(R)$ for a number ring $R$.
A sheaf  of locally free, locally finitely generated $\pr_{X}^{*}\cO_{X}$-modules $V$ on $M\times X$ is the same thing as a locally constant sheaf of finitely generated projective $R$-modules on $M$. 
It gives rise to complex vector bundles $V_{\sigma}\to M$ for all $\sigma\in X(\C)$.
Since $X(\C)$  is zero-dimensional, the datum of a good geometry
$g^{V}$  on $V$ reduces to the choice of  hermitian metrics $h^{V_{\sigma}}$ on the complex vector bundles $V_{\sigma}$, for all $\sigma\in X(\C)$, {which satisfy the compatibility} 
$h^{V_{\sigma}}=\bar h^{V_{\bar\sigma}}$. 
This is the same thing as a geometry on the  
locally constant sheaf of finitely generated projective $R$-modules on $M$ as considered in \cite[Def.~3.8]{bg}. We let $\bloc^{proj}_{geom}(M)$ denote the monoid of isomorphism classes of  locally constant sheaves of finitely generated projective $R$-modules on $M$ with geometry in the sense of \cite[Def.~3.8]{bg}.  Recall that $i\Vect^{geom}_{Mf,\Z}$ is the stack of geometric bundles, where a geometric bundle is a bundle with the choice of a local geometry, see Definition \ref{jul0850}. Since a good geometry induces a local geometry, we have a natural map
\begin{equation}\label{may2101n}
c\colon \bloc^{proj}_{geom}(M)\to  \pi_0(i\Vect^{geom}_{Mf,\Z}(M\times X)) \ .
\end{equation}
 A major result in \cite{bg} was the construction of the cycle map
$$
\hat \cycl\colon \bloc^{proj}_{geom}(M)\to \widehat{KR}^{0}(M)\ .
$$
In the present subsection we will compare it with the  cycle map 
$$
\hcycl\colon \pi_0(i\Vect^{geom}_{Mf,\Z}(M\times X))\to \widehat{\bK}^{0}(M\times X)
$$
constructed in  {Subsection \ref{jan1970}.} 
As the comparison in Lemma \ref{jan0441} shows, these are not equal, but differ by a natural correction term.

In the present paper, we let $$\beta(g^{V})\in Z^{0}(\Omega  {C}(M))$$ denote the characteristic 
form  of the  bundle  with  geometry $(V,g)$ as introduced in \cite[Def.~3.10]{bg}, and
$$\omega(g^{V})\in Z^{0}(\DR_{Mf,\Z}(M\times X))$$ be the form introduced in Definition \ref{jan0210eee}
{and} Lemma \ref{jan0901}.
In order to simplify formulas, we use a normalization adapted to the conventions of the present paper.
For $\Psi$ as in \eqref{jan0420} we write ({using again the notation from \eqref{jan0130}})
\begin{equation}\label{jan0406}\Psi(\beta(g^{V}))=\left(\beta(g^{V})_{\R}(p)\oplus \beta(g^{V})(p),\tilde \beta(g^{V})(p)\right)_{p\ge 0}\ .\end{equation}
Since $\Psi$ is injective the characteristic form $\beta(g^{V})$ is uniquely determined  by:
\begin{enumerate}
\item
For $p\ge 2$ we have
$\beta(g^{V})_{\R}(p)=0=\beta(g^{V})(p)$, and  for all $\sigma\in X(\C)$
$$\ev_{\sigma}\tilde \beta(g^{V})(p)=\frac{1}{2}\tilde \ch_{2p-1}(\nabla^{V_{\sigma},*},\nabla^{V_{\sigma}}) \ .$$
\item If $p=1$, then $\beta(g^{V})_{\R}(1)=0=\beta(g^{V})(1)$, and  for all $\sigma\in X(\C)$
$$\ev_{\sigma}\tilde \beta(g^{V})(1) =  \frac{1}{2} \tilde \ch_{1}(\nabla^{V_{\sigma},*},\nabla^{V_{\sigma}})-\kappa \ ,$$
where
$$\kappa:=\frac{1 }{2|X(\C)|}\sum_{\sigma\in X(\C)}  \tilde \ch_{1}(\nabla^{V_{\sigma},*},\nabla^{V_{\sigma}})\ .$$
\item
We have $\tilde \beta(g^{V})(0)=0$, and
for all $\sigma\in X(\C)$
$$\ev_{\sigma}(\beta(g^{V})_{\R}(0))=\ev_{\sigma}(\beta(g^{V})(0)):=\dim(V)\ .$$
\end{enumerate}

The following two lemmas prepare the definition of the correction term in the comparison of  {the} cycle maps.
 \begin{lem}\label{jan0540}
There exists a 
smooth map $g^{V}\mapsto \lambda(g^{V})\in  A_{\R}^{0}(M)$  {which is natural with respect to pull-back along maps between manifolds and} such that 
 $$\sum_{\sigma\in X(\C)} \tilde \ch_{1}(\nabla^{V_{\sigma},*},\nabla^{V_{\sigma}})=d \lambda(g^{V})\ .$$
   \end{lem}
\begin{proof}
Recall that
$$\tilde \ch_{1}(\nabla^{V_{\sigma},*},\nabla^{V_{\sigma}}) = \int_{I\times M/M} \ch_{2}(\nabla^{\widetilde V_{\sigma}})$$
where $\widetilde V_{\sigma}\to I\times M$ is the bundle $\pr_{M}^{*} V_{\sigma}$ with the connection
$\nabla^{\widetilde V_{{\sigma}}}$ obtained by linear interpolation between $\nabla^{V_{\sigma}}$ and $\nabla^{V_{\sigma},*}$. We consider the line
bundle $$\widetilde \bV:=\bigotimes_{\sigma\in X(\C)}\det(\widetilde V_{\sigma})$$
with the connection $\nabla^{\widetilde \bV}$ induced by the connections $\nabla^{\widetilde V_{\sigma}}$.
Then
$$ \sum_{\sigma\in X(\C)} \tilde \ch_{1}(\nabla^{V_{\sigma},*},\nabla^{V_{\sigma}})=
\int_{I\times M/M} \ch_{2}(\nabla^{\widetilde \bV}) \ .$$

We further consider the complex line bundle
$\bV:=\bigotimes_{\sigma\in X(\C)}\det(V_{\sigma})$ on $M$ with the induced flat connection $\nabla^{\bV}$.
We let $(s,t)\in I\times I$ denote the parameters.
On $I\times I\times  M$ we consider the bundle
$\widehat \bV:=\pr_{M}^{*}\bV$ with the connection $\nabla^{\widehat \bV}$ which along the $s$-direction linearly interpolates between the connection 
$\nabla^{\widetilde \bV}$  and the linear path (parametrized by $t$) between $\pr_{M}^{*}\nabla^{\bV}$ and $\pr_{M}^{*}\nabla^{\bV,*}$.
We set $$\lambda_{0}(g^{V}):=-\int_{I\times I\times M/M} \ch_{2}(\nabla^{\widehat \bV})$$
and observe that
$$
d\lambda_{0}(g^{V})=
\sum_{\sigma\in X(\C)} \tilde \ch_{1}(\nabla^{V_{\sigma},*},\nabla^{V_{\sigma}})-\tilde \ch_{1}(\nabla^{\bV,*},\nabla^{\bV}) \ .$$
We now observe that $\bV$ has a $\Z$-structure and therefore has a uniquely determined hermitian metric $h^{\bV}_{0}$ such that the $\Z$-basis vectors have length one.
On the bundle $\widehat \bV \to  
I\times I\times M$ we consider the metric 
$$\widehat h^{\widehat \bV}:=(1-s)\ \pr^{*}_{M}h^{\bV}+ s\  \pr^{*}_{M}h^{\bV}_{0}\ .$$
We consider the connection $\widehat{\nabla}^{\widehat \bV}$ which linearly (in $t$) interpolates between the 
connection $\pr_{M}^{*}\nabla^{\bV}$ and its adjoint with respect to  $\widehat h^{\widehat \bV}$.
We define $$\lambda_{1}(g^{V}):=-\int_{I\times I\times M/M} \ch_{2}(\hat \nabla^{\widehat \bV})$$
and observe that
 $$d \lambda_{1}(g^{V})=\tilde \ch_{1}(\nabla^{\bV,*},\nabla^{\bV}) \ .$$
Since $\Imm \sum_{\sigma\in X(\C)} \tilde \ch_{1}(\nabla^{V_{\sigma},*},\nabla^{V_{\sigma}})=0$ (see \eqref{sep2501} for  {the definition of} $\Imm$ and $\Ree$),
\begin{equation}\label{jan0520}\lambda(g^{V}):=\Ree(\lambda_{0}(g^{V})+\lambda_{1}(g^{V}))\end{equation} 
has the required properties.
\end{proof}

Let $V\to M$ be a complex vector bundle and $\nabla_{i}$, $i=0,1,2$, be three connections on $V$. Then we have the following well-known  relation
$$\tilde \ch(\nabla_{0},\nabla_{1})+\tilde \ch(\nabla_{1},\nabla_{2})\equiv \tilde \ch(\nabla_{0},\nabla_{2})$$ modulo exact forms. In fact, we have a stronger result.
\begin{lem}
There exists a 
smooth map
$$(\nabla_{0},\nabla_{1},\nabla_{2})\mapsto \tilde \ch(\nabla_{0},\nabla_{1},\nabla_{2})\in A(M)\ ,$$
 {natural with respect to pull-back along maps between manifolds,} such that
$$\tilde \ch(\nabla_{0},\nabla_{1})+\tilde \ch(\nabla_{1},\nabla_{2})- \tilde \ch(\nabla_{0},\nabla_{2})
= d\tilde \ch(\nabla_{0},\nabla_{1},\nabla_{2})\ .$$
\end{lem}
\begin{proof}
We consider the bundle $\widehat V:=\pr_{M}^{*}V\to \Delta^{2}\times M$  and define the connection
$\nabla^{\widehat V}$ as the convex linear interpolation
between the connections $\nabla_{i}$ at the corners of the simplex.
Then
$$\tilde \ch(\nabla_{0},\nabla_{1},\nabla_{2}):=\int_{\Delta^{2}\times M/M} \ch(\nabla^{\widehat V})$$
does the job.
\end{proof}

We now define   the natural transformation \begin{equation}\label{jan0340}\alpha:\bloc^{proj}_{geom}(M)\to \DR^{-1}_{Mf,\Z}(M\times X)\ ,\quad  (V,g^{V})\mapsto \alpha(g^{V})\end{equation}  {that will show up in the above mentioned correction term}.
We use the notation  \eqref{jan0130} in order  to describe the form $\alpha(g^{V})$:
\begin{enumerate}
\item
For $p\ge 2$ we set
$$\ev_{\sigma}\alpha(g^{V})_{\R}(p):= \frac{1}{2}\left(\tilde \ch_{2p-1}(\nabla^{V_{\sigma},u},\nabla^{V_{\sigma}})+\tilde \ch_{2p-1}(\nabla^{V_{\sigma},u},\nabla^{V_{\sigma},*})\right)\ ,$$
$$\ev_{\sigma}\tilde \alpha(g^{V})(p):= {-}\frac{1}{2}\tilde \ch_{2p-2}(\nabla^{V_{\sigma},*},\nabla^{V_{\sigma},u},\nabla^{V_{\sigma}})\ ,$$
and $\alpha(g^{V})(p):=0$. 
\item 
For $p=1$ we set (using \eqref{jan0520})
 $$\ev_{\sigma}\alpha(g^{V})_{\R}(1):= \frac{1}{2}\left(\tilde \ch_{1}(\nabla^{V_{\sigma},u},\nabla^{V_{\sigma}})+\tilde \ch_{1}(\nabla^{V_{\sigma},u},\nabla^{V_{\sigma},*})\right)\ ,$$
   $$\ev_\sigma\tilde \alpha(g^{V})(1):= {-}\frac{1}{2}\tilde \ch_0(\nabla^{V_{\sigma},*},\nabla^{V_{\sigma},u},\nabla^{V_{\sigma}}) - \frac{1}{2|X(\C)|}\lambda(g^{V})\ ,$$
 and $\alpha(g^{V})(1):=0$.

\item
We set $\tilde \alpha(g^{V})(0):=0$, $\alpha(g^{V})_{\R}(0):=0$ and $\alpha(g^{V})(0):=0$.
\end{enumerate}
A simple calculation shows that 
\begin{equation}\label{jan0408}
\omega(g^{V})=\Psi(\beta(g^{V}))+d\alpha(g^{V}) \ .
\end{equation}
Recall the transformations $\psi$ from \eqref{mar1420} and $c$ from \eqref{may2101n}. 
 {In the next lemma, we} compare the two cycle maps
\[
\hcycl\circ c,\,  \psi\circ\hat\cycl\colon \bloc^{proj}_{geom}  \to   \widehat{\bK}(X)^0.
\] 
 {In its formulation, we use the natural map} 
\[
V\colon \bloc^{proj}_{geom} \to  {e}_X^{*}\pi_0(i\Vect_{Mf,\Z} ) 
\]
which forgets the geometry  and the flat part $\widehat{\bK}(X)_{flat}^0 $ of differential algebraic $K$-theory  defined in \eqref{jul1120}.
\begin{lem}\label{jan0441}
There exists a natural transformation 
\[
\epsilon: {e}_X^{*}\pi_0(i\Vect_{Mf,\Z})  \to \widehat{\bK}(X)_{ flat}^0 
\]
such that 
\begin{equation*}
\psi\circ\hat\cycl - \hcycl\circ c + a\circ\alpha = \epsilon\circ V.
\end{equation*}
Moreover, $\epsilon$ is additive on short exact sequences, i.e.~if 
\begin{equation}\label{mar2701}
0\to V_0 \to V_1 \to V_2 \to 0
\end{equation}
is an exact sequence, then we have the equality
$$\epsilon(V_1) = \epsilon(V_0) + \epsilon(V_2)\ .$$
\end{lem}
\begin{proof}
We consider the natural transformation
\[
\delta:= \psi\circ\hat\cycl - \hcycl \circ c+ a\circ\alpha\colon  \bloc^{proj}_{geom} \to \widehat{\bK}(X)^0.
\]
The equality of characteristic forms \eqref{jan0408} implies that $R\circ \delta = 0$. Since $X$ is proper over $\Spec(\Z)$, any geometry is good and we can interpolate between any two  geometries on a given bundle. The homotopy formula  {from Lemma} \ref{jan0605} then implies that $\delta$ factors as a composition
\[
\bloc^{proj}_{geom}\xrightarrow{V}  {e}_X^{*}\pi_0(i\Vect_{Mf,\Z} ) \xrightarrow{\epsilon} \widehat{\bK}(X)_{flat}^0 \subseteq \widehat{\bK}(X)^0.
\]
This proves the first part of the lemma.

The functor 
  $\widehat{\bK}(X)_{flat}^{0}$ is, in contrast to $\widehat{\bK}(X)^0 $, part of  a cohomology theory.
It is represented by the homotopy fibre of the map of spectra $\beil\colon \bK_\Z(X) \to H(\DR_{\Z}(X))$.
In particular, it makes sense to evaluate $\widehat{\bK}(X)^0_{flat}$ on any topological space.

Note that $\delta$ and hence $\epsilon$ are additive for direct sums and vanish on a trivial bundle. Therefore, $\epsilon$  corresponds 
to a class  {in} $\widehat{\bK}(X)^0_{ flat}(BGL(R))$, where $BGL(R)$ is the classifying space of the discrete group $GL(R)$. Since $\widehat{\bK}(X)^0_{flat}$ is {part of} a cohomology theory,  passing to the plus-construction induces an isomorphism
 $$\widehat{\bK}(X)^0_{flat}(BGL(R)^+) \xrightarrow{\cong} \widehat{\bK}(X)^0_{ flat}(BGL(R))\ .$$

We denote by $c_V\colon M\to BGL(R)$ the map that classifies the stable class of $V\in{ {e}_{X}^{*}\pi_{0}(i\Vect_{Mf,\Z})(M)}$. Given  an extension \eqref{mar2701} one knows that $c_{V_1}$ and $c_{V_0\oplus V_2}$ become homotopic to each other when  we compose them with $BGL(R) \to BGL(R)^+$. This implies the  second assertion of the lemma.
\end{proof}

\subsection{Extensions}\label{mar18002}

The main result of this subsection is the proof (\ref{mar2702}) of {Theorem} \ref{jan2104}, which identifies the alternating sum of cycle classes for an extension of geometric vector bundles with a version of the Bismut-Lott higher torsion form $\cT$ (Lott's relation). We refer to \ref{mar2703} for the precise statement. The main ingredient is a deformation to a split extension over $ {\P^{1}_{\C}}$. Using integration in the algebraic direction (\ref{jan0110}) we construct in \ref{jan0470} a form $\cT^{\prime}$ and show that it coincides with the Bismut-Lott form $\cT$ up to exact forms. For this we use the axiomatic characterization of the latter, to be recalled in \ref{jan0341} below.  Lott's relation then follows essentially from the homotopy formula  {in} Proposition  \ref{jan0111}.

\subsubsection{The axiomatic characterization of the torsion form}\label{jan0341}


In this section we fix $X= \Spec(\C)$. Then a bundle over $M\times X$ is a
 locally constant sheaf $V$ of finite dimensional $\C$-vector spaces on the manifold $M$. It can be presented as the sheaf of parallel sections of a uniquely determined flat complex vector bundle $(V,\nabla^{V})$ on $M$. The connection $\nabla^{V}$ is a partial geometry on that bundle  in the sense of  Definition \ref{dez2601}. A geometry on $V$ in the sense of Definition  \ref{nov1904} is  a hermitian metric
$h^{V}$ on the vector bundle $V\to M$.  {Given a geometry,} $\nabla^{V,*}$ denotes the adjoint connection.

Let $$\cV: 0\to V_{0}\to \dots \to V_{n}\to 0$$
be an exact sequence of locally constant sheaves of finite dimensional $\C$-vector spaces
on a manifold $M$. A geometry on $\cV$ is  {by definition} a collection $h^{\cV}=(h^{V_{i}})_{i=0,\dots,n}$  of metrics on the bundles $V_{i}$.
In this situation the Bismut-Lott torsion form
$$\cT(\cV,h^{\cV})\in A(M)$$
is defined such that
\begin{equation}\label{apr2701}d\cT(\cV,h^{\cV})=\frac{1}{2}\sum_{i=0}^{n} (-1)^{n} \tilde \ch(\nabla^{V_{i},*},\nabla^{V_{i}}) \end{equation}
 \cite[Def.~2.20]{MR1303026}. Note that our normalization of the torsion form differs from the one adopted in  \cite{MR1303026}.  In the present paper we will not use the precise construction of the torsion form, but rather its axiomatic characterization. We restrict to the case $n=2$ and
consider a  transformation
$$(\cV,h^{\cV})\mapsto \cT^{\prime}(\cV,h^{\cV})$$
which associates to every exact sequence
 \begin{equation}\label{jan0402}\cV:0\to V_{0}\to V_{1}\to V_{2}\to 0 \end{equation}
of locally constant sheaves of finite dimensional $\C$-vector spaces on a manifold $M$,
equipped with a geometry $h^{\cV}$, a form
$\cT^{\prime}(\cV,h^{\cV})\in A(M)$.
 We have the following characterization, due to Bismut-Lott:
\begin{theorem}[{\cite[Thm. A1.2]{MR1303026}}]\label{jan0401neu}
Assume that $(\cV,h^{\cV})\mapsto \cT^{\prime}(\cV,h^{\cV})$ satisfies
\begin{enumerate}
\item $ {\Ree(\cT^{\prime}(\cV,h^{\cV})_{2p-2})=0}$ for all $p\ge 1$  (see \eqref{sep2501} for  {the notation} $\Ree$) ,
\item $d\cT^{\prime}(\cV,h^{\cV})=\frac{1}{2}\sum_{i=0}^{2}(-1)^{i}\tilde \ch(\nabla^{V_{i},*},\nabla^{V_{i}})$,
\item $\cT^{\prime}$ is natural, i.e.~$\cT^{\prime}(f^{*}\cV,f^{*}h^{\cV})=f^{*}\cT^{\prime}(\cV,h^{\cV})$ for a smooth map $f:M^{\prime}\to M$,
\item $\cT^{\prime}(\cV,h^{\cV})=0$ in the split case, i.e.~when $V_{1}\cong V_{0}\oplus V_{2}$ as flat bundles 
and $h^{V_{1}}=h^{V_{0}}\oplus h^{V_{2}}$, 
\item $\cT^{\prime}$ depends smoothly on $\cV$ and $h^{\cV}$.
\end{enumerate}
Then $\cT^{\prime}(\cV,h^{\cV})$ equals the Bismut-Lott torsion form $\cT(\cV,h^{\cV})$ modulo exact forms:
$$\cT^{\prime}(\cV,h^{\cV})\equiv \cT(\cV,h^{\cV})\in A(M)/\im(d)\ .$$
\end{theorem}

\subsubsection{A deformation}\label{jan0470}

In this  {subsection} we give a holomorphic construction of a version of the analytic torsion form.
Let $(\cV,h^{\cV})$ be an exact sequence \eqref{jan0402} with geometry.
  We construct a form $\cT^{\prime}(\cV,h^{\cV})$ which satisfies the {assumptions in Theorem \ref{jan0401neu}.} 
The main point of the  construction below is that it is canonical.   This ensures the properties 
3.~and 5.~in Theorem \ref{jan0401neu}.

The metric  $h^{V_{1}}$ provides an orthogonal decomposition of vector bundles $V_{1}\cong V_{0}\oplus V_{2}$. With respect to this decomposition the connection on $V_{1}$ can be written as
$$\nabla^{V_{1}}=\nabla^{V_{0}}\oplus \nabla^{V_{2}}+ E\ ,\quad  E\in A^{1}(M,\Hom(V_{2},V_{0}))\ .$$
 Furthermore,
$$h^{V_{1}}=h^{V_{1}}_{|V_{0}}\oplus h^{V_{1}}_{|V_{2}}\ .$$

We let $Y:= {\P^{1}_{\C}}$ with homogeneous coordinates $[x:y]$. 
We consider $x,y$  as global sections  of the sheaf $\cO_{Y}(1) $ of holomorphic sections of the    holomorphic line bundle $H_{Y}(1)$. This bundle carries a natural $U(2)$-invariant metric  $h^{H_Y(1)}$ and an $U(2)$-invariant connection.
We normalize the metric $h^{H_Y(1)}$ in such a way that its value on the section $y$ evaluated at $[0:1]$ is 1.

On $M\times Y$ we  consider the complex vector bundle
$$W:=\pr_{M}^{*} V_{0}\otimes \pr_{Y}^*{H_{Y}(1)}\oplus \pr_{M}^{*}V_{2}\ .$$
The connections $\nabla^{V_{0}}$ and $\nabla^{V_{2}}$ together with the connection on
$H_{Y}(1)$ induce a connection $\nabla^{W}_{0}$ on $W$. We consider
$$
\pr_{M}^{*}E\otimes \pr_{Y}^{*}y\in A^{1}(M\times Y,\Hom(\pr_{M}^{*}V_{2},\pr_{M}^{*}V_{0}\otimes \pr_{Y}^{*} {H_Y(1)}))
$$ 
and define the connection
$$
\nabla^{W}:=\nabla^{W}_{0}+ \pr^{*}_{M}(E)\otimes \pr_{Y}^{*}y\ .
$$
Note that, for every $s\in \C$, the connection $\nabla^{V_{0}}\oplus \nabla^{V_{2}}+sE$  on $V_{1}$ is flat. Therefore,
the partial connection $\nabla^{I}$ on $W\to M\times Y$  induced by $\nabla^{W}$ in the $M$-direction is flat.
Furthermore, the $(0,1)$-part $\bar \partial$ of $\nabla^{W}$ in the $Y$-direction is a holomorphic structure. The pair $(\nabla^{I},\bar \partial)$ is a partial geometry on the complex vector bundle $W\to M\times Y$  {in the sense of} Definition \ref{dez2601}.

 We fix a  cut-off function $\theta\in C^{\infty}(\R)$ such that
 $\theta(t)\in [0,1]$ for all $t\in \R$, 
$\theta(t)\equiv 0$ for $t\le 1$ and $\theta(t)\equiv 1$ for
$t\ge 2$.
The bundle $W$ has two metrics
$h^{W}_{02}$ and $h^{W}_{1}$, where the first one is induced from $h^{V_{0}}$, $h^{V_{2}}$ and
$h^{H_{Y}(1)}$, while the second one  is induced in the same manner from $h^{V_{1}}_{|V_{0}}$ and $h^{V_{1}}_{|V_{2}}$ and $h^{H_{Y}(1)}$.
We define
the metric
$$
h^{W}([x:y]):=(1-\theta(|\frac{y}{x}|))  h^{W}_{02} {([x:y])} +\theta(|\frac{y}{x}|) h^{W}_{1} {([x:y])}
$$
(note that $h^{W}$ extends smoothly to the point $[0:1]\in Y$).

The pair $(\nabla^{II},h^{W})$  consisting of the partial connection $\nabla^{II}$ in the $Y$-direction induced by $\nabla^{W}$ 
and the metric $h^{W}$ is a good geometry $g^{W}$ in the sense of Definitions \ref{nov1904}, \ref{jan2702}. In particular, we have the characteristic form \eqref{jan0403}
$$\omega(g^{W})\in Z^{0}(\DR_{Mf,\C}(M\times Y))\ .$$ 
We abbreviate $ \DR_{Mf,\C}(M):= \DR_{Mf,\C}(M\times \Spec(\C))$.
Using \eqref{mar2001}, 
we define the form
\begin{equation}\label{jan0480}
\cU(\cV,h^{\cV}):=\int_{II} \omega(g^{W}) \in \DR_{Mf,\C}^{-1}(M)\ .
\end{equation}

Let $f_{0},f_{1}\colon M \hookrightarrow M\times {\P^{1}_{\C}}$ be the inclusions given by the points ${[0:1],[1:0]}\in  {\P^{1}_{\C}}$.
Then, by construction  and using $y$ as a basis of $\cO_Y(1)_{|\{y\not=0\}}$, we have isomorphisms
$$
{f_{0}^{*}} (W,g^{W})\cong (V_{1},h^{V_{1}})\ , \quad {f_{1}^{*}}  (W,g^{W})\cong (V_{0}\oplus V_{2},h^{V_{0}}\oplus h^{V_{2}})\ .
$$
%
%
Since $\omega(g^{W})$ is closed, \eqref{mar1902} implies that
\begin{equation}\label{jan0410}
-d\cU(\cV,h^{\cV})=   \sum_{i=0}^2 (-1)^i \omega(h^{V_{i}}) =: \: 
  \omega(h^{\cV})\  .
\end{equation}
We define a form
 $\alpha(h^{V}) \in \DR_{Mf,\C}^{-1}(M)$  using the notation \eqref{jan0130}  {similarly to the form defined in \eqref{jan0340}}:
\begin{enumerate}
\item
For $p\ge 1$ we set 
$$\alpha(h^{V})_{\R}(p):= \frac{1}{2}\left(\tilde \ch_{2p-1}(\nabla^{V,u},\nabla^{V})+\tilde \ch_{2p-1}(\nabla^{V,u},\nabla^{V,*})\right)\ ,$$
$$\tilde \alpha(h^{V})(p):= {-}\frac{1}{2} \tilde \ch_{2p-2}(\nabla^{V,*},\nabla^{V,u},\nabla^{V})\ ,$$
and $\alpha(h^{V})(p):=0$.
 \item
We set $\tilde \alpha(h^{V})(0):=0$, $\alpha(h^{V})_{\R}(0):=0$ and $\alpha(h^{V})(0):=0$.
 \end{enumerate}
 We set 
 \begin{equation}\label{jan0450}  \alpha(h^{\cV}):=\sum_{i=0}^{2} (-1)^{i} \alpha(h^{V_{i}}) \in  \DR^{-1}_{Mf,\C}(M)\ .\end{equation}
Furthermore, we define
$$\gamma(h^{\cV}):=\sum_{i=0}^{2} (-1)^{i} \gamma(h^{V_{i}})\in Z^{0}(\DR_{Mf,\C}(M))\ ,$$ where
$\gamma(h^{V})\in Z^{0}(\DR_{Mf,\C}(M))$
 is given as follows:
 \begin{enumerate}
\item
For $p\ge 1$ we set 
$\gamma(h^{V})_{\R}(p):=0$, $
\gamma(h^{V})(p):=0$ and 
$$\tilde \gamma(h^{V})(p):=\frac{1}{2}\tilde \ch_{2p-1}(\nabla^{V,*},\nabla^{V})\ .$$ 
 \item
We set $\tilde \gamma(h^{V})(0):=0$, $\gamma(h^{V})_{\R}(0)=\gamma(h^{V})(0):=\dim(V)$.
 \end{enumerate}
Then we
have (compare with  \eqref{jan0408})
\begin{equation}\label{jan0411}\omega(h^{\cV})-d\alpha(h^{\cV})=\gamma(h^{\cV})\ .
 \end{equation}

 We define
\begin{equation}\label{jan0466}\cZ(\cV,h^{\cV}):=  \cU(\cV,h^{\cV}) {+}\alpha(h^{\cV})\in \DR_{Mf,\C}^{-1}(M)\ .\end{equation}
The  {equalities} \eqref{jan0410} and \eqref{jan0411} imply that the components of $\cZ(\cV,h^{\cV})$ (as in 
\eqref{jan0130}) satisfy  
$\cZ(\cV,h^{\cV})(p)=0$ and 
\begin{equation}\label{jan0491} d\cZ_{\R}(\cV,h^{\cV})(p)=0\ ,\quad  {-}d\tilde \cZ(\cV,h^{\cV})(p) {+}\cZ_{\R}(\cV,h^{\cV})(p)=  {-}\tilde\gamma(h^{\cV})(p)\end{equation} for all $p\ge 0$.
We now separate the real and imaginary parts  (see \eqref{sep2501})  of this equality.
Note that 
$$\Imm(\cZ_{\R}(\cV,h^{\cV})(p))=0  \ ,\quad \Ree(\tilde \ch_{2p-1}(\nabla^{V_{i},*},\nabla^{V_{i}}))=0\ \ .$$ 
We define
$$\cT^{\prime}(\cV,h^{\cV})_{2p-2}:= \Imm ( \tilde \cZ(\cV,h^{\cV})(p))\ ,\quad  \cT^{\prime}(\cV,h^{\cV}):=\sum_{p\ge 1} \cT^{\prime}(\cV,h^{\cV})_{2p-2}\ .$$
Then the imaginary part of \eqref{jan0491} gives
\begin{equation}\label{mar2101}
d\cT^{\prime}(\cV,h^{\cV})=\frac{1}{2} \sum_{i=0}^{2}(-1)^{i}  \tilde \ch(\nabla^{V_{i},*},\nabla^{V_{i}})\ .
\end{equation}

Finally, we check that $\cT^{\prime}(\cV, h^{\cV})$ vanishes in the split case. 
If $(\cV,h^{\cV})$ is split, then the metric $h^W = h^W_{02} = h^W_1$, the connection $\nabla^W = \nabla^W_0$, and hence the characteristic forms are $U(2)$-invariant. In particular, the components of $\omega(g^W) \in \DR_{Mf,\C}(M\times Y)$ are harmonic in the $Y$-direction, and, by the second equality in \eqref{may1721}  {and the definition of $\int_{II}$ in \eqref{mar2001}}, we get $\int_{II} \omega(g^W)=0$. One sees directly that also $\tilde\alpha(h^{\cV})$ vanishes.

Putting everything together, we see that 
$$(\cV,h^{\cV})\mapsto \cT^{\prime}(\cV,h^{\cV})$$ satisfies
 the assumptions in Theorem \ref{jan0401neu} characterizing  the Bismut-Lott torsion form, and hence 
\begin{equation}\label{jan0460}\cT^{\prime}(\cV,h^{\cV})\equiv \cT(\cV,h^{\cV}) \end{equation}
modulo exact forms.

\bigskip

For later use we introduce the following notation:
\begin{equation}\label{jan0461}\cZ^{\Imm}(\cV,h^{\cV}):=\left(0\oplus 0, \cT^{\prime}(\cV,h^{\cV})_{2p-2}\right)_{p\ge 0}\end{equation}
and
$$\cZ^{\Ree}(\cV,h^{\cV}):=\left(\cZ_{\R}(\cV,h^{\cV})(p) \oplus0,   \Ree(\tilde \cZ(\cV,h^{\cV})(p))\right)_{p\ge 0}$$
so that
\begin{equation}\label{jan0465}\cZ(\cV,h^{\cV})=\cZ^{\Ree}(\cV,h^{\cV})+\cZ^{\Imm}(\cV,h^{\cV})
\ .\end{equation}
Further note that \eqref{jan0491} implies
\begin{equation}\label{jan0490}\cZ^{\Ree}(\cV,h^{\cV})=d\left( \Ree(\tilde \cZ(\cV,h^{\cV})(p))\oplus 0,0\right)_{p\ge 0}\ .\end{equation}

\subsubsection{Lott's relation}
\label{mar2703} 

We fix a number ring $R$ and set $X:= {\Spec}(R)\in  {\bReg}_{\Z}$.   We further consider a smooth manifold $M$ and an exact sequence
\begin{equation}\label{mar2601}
\cV:0\to V_{0}\to V_{1}\to V_{2}\to 0
\end{equation}
of locally free{, locally} finitely generated $\pr_{X}^{*}\cO_{X}$-modules on $M\times X$. 
     
  Recall from Subsection \ref{jan0415} that for every $\sigma\in X(\C)$ we get an exact sequence 
 \begin{equation}\label{jan1701}\cV_{\sigma}:0\to V_{0,\sigma}\to V_{1,\sigma}\to V_{2,\sigma}\to 0\ \end{equation} of locally constant sheaves of finite dimensional $\C$-vector spaces.
 A  geometry $g^{\cV}=(g^{V_{i}})_{i=0,1,2}$ is the same as a collection of geometries $(h^{\cV_{\sigma}})_{\sigma\in X(\C)}$     such that $\bar h^{\cV_{\sigma}}= h^{\cV_{\bar\sigma}}$  {(cf.~the introductory paragraphs of \ref{jan0415} and \ref{jan0341})}.

We  use the Bismut-Lott torsion forms $\cT(\cV_{\sigma},h^{\cV_{\sigma}})$ in order to
define the analytic torsion form 
$$\cT_{\Z}(\cV,g^{\cV}) \in \Omega {C}^{-1}(M),$$
which coincides up to normalization with the form \cite[(104)]{bg}.
Note that the map $\Psi$ in \eqref{jan0420} is injective. Hence we can describe $\cT_{\Z}(\cV,g^{\cV})$ by giving the components  \eqref{jan0130} of its image under $\Psi$. We write
$$\Psi(\cT_{\Z}(\cV,g^{\cV}))=( \cT_{\Z}(\cV,g^{\cV})_{\R}(p)\oplus  \cT_{\Z}(\cV,g^{\cV}) (p), \tilde  \cT_{\Z}(\cV,g^{\cV}) (p))_{p\ge 0}\ .$$ 
Then we have
 $$
 \cT_{\Z}(\cV,g^{\cV})_{\R}(p):=0\ ,\quad  \cT_{\Z}(\cV,g^{\cV})(p):=0 \quad  {\text{for all } p\geq 0}
$$ 
and
\begin{equation}\label{apr2702}
\ev_{\sigma} \tilde\cT_{ {\Z}}(\cV,g^{\cV})(p):=\left\{\begin{array}{cc}  {-}\cT(\cV_{\sigma},h^{\cV_{\sigma}})_{2p-2}, &\text{if } p\ge 2,\\
 {-}\cT(\cV_{\sigma},h^{\cV_{\sigma}})_{0} {+}\bar\tau(\cV,g^{\cV}),
 &\text{if } p=1,\\
0,&\text{if } p=0,
\end{array}
\right.
  \end{equation}
  with
\begin{equation}\label{jan0453}
\bar \tau(\cV,g^{\cV}):=
\frac{1}{|X(\C) |}\sum_{\sigma\in X(\C)}   \cT(\cV_{\sigma},h^{\cV_{\sigma}})_{0}.
\end{equation}  
Recall the characteristic form $\beta(g^{V_i})$ introduced in \eqref{jan0406}, and write
$$\beta(g^{\cV}):=\sum_{i=0}^{2}(-1)^{i} \beta(g^{V_{i}}).$$
By the fundamental property \eqref{apr2701} of the Bismut-Lott torsion form, we then have 
$$
 d\Psi(\cT_{\Z}(\cV,g^{\cV}))=\Psi(\beta(g^{\cV})) \quad \text{ and hence }\quad d\cT_{\Z}(\cV,g^{\cV})=\beta(g^{\cV}). 
$$ 
  In \cite[Conj.~5.7]{bg} we asked whether
$$  \hat \cycl(V_{0},g^{V_{0}})-\hat \cycl(V_{1},g^{V_{1}})+\hat \cycl(V_{2},g^{V_{2}})=a(\cT_{\Z}(\cV,g^{\cV})) $$ 
holds true    in  $\widehat{KR}^{0}(M)$.
We call this Lott's relation. For more motivation we refer to \cite[Sec.~5.4.1]{bg}.
\begin{theorem}\label{jan0430}
Lott's relation holds true, i.e.~we have
$$  \hat \cycl(V_{0},g^{V_{0}})-\hat \cycl(V_{1},g^{V_{1}})+\hat \cycl(V_{2},g^{V_{2}})=a(\cT_{\Z}(\cV,g^{\cV})) $$
  in $\widehat{KR}^{0}(M)$.
\end{theorem}

\subsubsection{Proof of Theorem \ref{jan0430}}  
\label{mar2702}

Recall the functions \eqref{jan0453} and  \eqref{jan0520}.
\begin{lem}\label{jan0570}
We have the relation
$${\bar \tau(\cV,g^{\cV})=}\frac{1}{2|X(\C)|}\sum_{i=0}^{2}(-1)^{i}\lambda( {g^{V_{i}}})\ .$$
\end{lem}
\begin{proof}
We write
$$\sigma(\cV,g^{\cV}):=\bar \tau(\cV,g^{\cV})-\frac{1}{2|X(\C)|}\sum_{i=0}^{2}(-1)^{i}\lambda( {g^{V_{i}}})\ .$$
We first observe that, as a consequence of Lemma \ref{jan0540} and  \eqref{mar2101},  
we have $d\sigma(\cV,g^{\cV})=0$.
This implies that $\sigma(\cV,g^{\cV})$ is independent of the choice of the metric $g^{\cV}$.
Moreover it suffices to check the equality $\sigma(\cV,g^{\cV})=0$ in the case that $M$ is a point. If $M$ is a point, then we can  assume that $(\cV,g^{\cV})$ splits, in which case the relation $\sigma(\cV,g^{\cV})=0$ is clear. 
\end{proof}

To prove the theorem, by
Lemma \ref{jan0440}  and \eqref{jan0506} it suffices to show that
$$\psi(\hat \cycl(V_{0},g^{V_{0}}))-\psi(\hat \cycl(V_{1},g^{V_{1}}))+\psi(\hat \cycl(V_{2},g^{V_{2}}))=a(\Psi(\cT_{\Z}(\cV,g^{\cV})))$$
in $\widehat{\bK}^{0}(M\times X)$. By Lemma \ref{jan0441}, this is equivalent to
$$\hcycl(V_{0},g^{V_{0}})-\hcycl(V_{1},g^{V_{1}})+ \hcycl(V_{2},g^{V_{2}})=a\left(\Psi(\cT_{\Z}(\cV,g^{\cV})) + \alpha(g^{\cV})\right)\ ,$$
where $\alpha(g^{\cV})$ is defined using the form $\alpha$ from \eqref{jan0340} by 
$$\alpha(g^{\cV}):=\sum_{i=0}^{2}(-1)^{i} \alpha(g^{V_{i}})$$
and we suppress the comparison map $c$ given in \eqref{may2101n}.
We write
$$\Psi(\cT_{\Z}(\cV,g^{\cV})) + \alpha(g^{\cV})=:\delta(\cV,g^{\cV})\in \DR^{-1}_{Mf,\Z}(M\times X))\ .$$  

If we compare \eqref{jan0450} with the definition of the map \eqref{jan0340}, then we get the relations
$$\ev_{\sigma}\alpha(g^{\cV})(p)=\alpha(g^{\cV_{\sigma}})(p)$$   for $p\not=1$ and
\begin{eqnarray}
\ev_{\sigma}\alpha(g^{\cV})(1)&=&\alpha(g^{\cV_{\sigma}})(1)-(0\oplus0,\frac{1}{2|X(\C)|}\sum_{i=0}^{2}(-1)^{i}\lambda(g^{V_{i}}))\nonumber\\
&\overset{\text{Lemma \ref{jan0570}}}{=}&\alpha(g^{\cV_{\sigma}})(1)-(0\oplus0,\bar \tau(\cV,g^{\cV}) )\label{dkwqdhkqwddwqdwqd}
\end{eqnarray}  
for the components in  $\DR^{-1}_{Mf,\C}(p)(M)$. 
We use  {in particular} \eqref{dkwqdhkqwddwqdwqd}
in the equality \eqref{lkjwqldqwjdlqwdjwqdqde12e2e12e21} of the calculation below in order to cancel out the contribution of $\bar \tau(\cV,g^{\cV})$.
We calculate  
\begin{eqnarray}
\ev_{\sigma}\delta(\cV,g^{\cV}) &\overset{\eqref{apr2702}}{=}& (0\oplus 0,-\cT(\cV_{\sigma},h^{\cV_{\sigma}})_{2p-2})_{p\ge 0} +\alpha(g^{\cV_{\sigma}})\label{lkjwqldqwjdlqwdjwqdqde12e2e12e21}
\\&\overset{\eqref{jan0460}}{\equiv}& (0\oplus 0,-\cT^{\prime}(\cV_{\sigma},h^{\cV_{\sigma}})_{2p-2})_{p\ge 0}+  \alpha(g^{\cV_{\sigma}})\nonumber\\
&  \overset{  \eqref{jan0461}}{=}&
-\cZ^{\Imm}(\cV_{\sigma},g^{\cV_{\sigma}})+  \alpha(g^{\cV_{\sigma}}) \nonumber \\
& {\overset{\eqref{jan0465}}{=}}&
 {-}\cZ(\cV_{\sigma},h^{\cV_{\sigma}}) {+}\alpha(g^{\cV_{\sigma}})   {+}\cZ^{\Ree}(\cV_{\sigma},h^{\cV_{\sigma}})\nonumber \\
& {\overset{\eqref{jan0466}}{=}}& {-}\cU(\cV_{\sigma},h^{\cV_{\sigma}})  {+} \cZ^{\Ree}(\cV_{\sigma},h^{\cV_{\sigma}})  \nonumber\\
&\stackrel{\eqref{jan0490}}{\equiv}&  {-}\cU(\cV_{\sigma},h^{\cV_{\sigma}})  {\quad \pmod{\im(d)}} \label{jan0571}\ .
\end{eqnarray} 
The idea is now to refine the deformation constructed in Subsection \ref{jan0470}  {in} such  {a way} that it is compatible with the $R$-module structure.
  We can assume that $M$ is connected and  fix a base point $m\in M$. We let
$U_{i}\in \Mod(R)$ be the fibres of $V_{i}$ at $m$.
The sheaves $V_{i}$ are determined by the holonomy representations $\rho_{i}$ of $\pi_{1}(M,m)$ on the finitely generated projective $R$-modules
$U_{i}$. We can choose  a decomposition \begin{equation}\label{jan1610}U_{1}=U_{0}\oplus U_{2}\end{equation} as $R$-modules. There exists a uniquely determined map  $\nu\colon \pi_{1}(M,m)\to \Hom_{\Mod(R)}(U_{2},U_{0})$ such that
$$\rho_{1}=\left(\begin{array}{cc} \rho_{0}&\nu\\ 0&\rho_{2}\end{array}\right)\ .$$

We let $Y:=\P^{1}_{\Z}$ with homogeneous coordinates $[x:y]$ and  consider $y\in \cO_{Y}(1)(Y)$.
We define the vector bundle
$$\widetilde U_{1}:=U_{0}\otimes_{R}\cO_{Y}(1)\oplus U_{2}\otimes_{R}\cO_{Y}$$ on
$X\times Y$ and set
$$\widetilde \rho_{1}:=\left(\begin{array}{cc} \rho_{0}\otimes 1&\nu\otimes y\\ 0&\rho_{2}\end{array}\right):\pi_{1}(M,m)\to \Aut(\widetilde U_{1}) \ .$$
The representation $\widetilde \rho_{1}$ determines a  sheaf $\widetilde V_{1}$
of locally free, finitely generated  $\pr_{X\times Y}^{*}\cO_{X\times Y}$-modules on
$M\times X\times Y$. Its base-change along the inclusion $\Z\to \C$ is a sheaf $\widetilde V_{1,\C}$ over
$M\times X(\C)\times {\P^{1}_{\C}}$.
 For every $\sigma\in X(\C)$, we get  a  complex vector bundle 
$\widetilde W_{\sigma}\to M\times  {\P^{1}_{\C}}$ with partial geometry $(\nabla^{I,\widetilde W_{\sigma}},\bar\partial^{\widetilde W_{\sigma}})$ such that the sheaf of joint kernels of $\nabla^{I,\widetilde W_{\sigma}}$ and $\bar\partial^{\widetilde W_{\sigma}}$  is canonically  isomorphic to $(\widetilde V_{1,\C})_{|M\times \{\sigma\}\times  {\P^{1}_{\C}}}$. 

If we apply the construction of Subsection  \ref{jan0470} to the sequence \eqref{jan1701},
we get a bundle $W_{\sigma}\to M\times {\P^{1}_{\C}}$ with partial geometry $(\nabla^{W_{\sigma}},\bar\partial^{W_{\sigma}})$.
By construction we have a natural identification 
$$\phi_{\sigma,m}:(W_{\sigma})_{|\{m\}\times  {\P^{1}_{\C}}}\xrightarrow{ {\cong}} (\widetilde W_{\sigma})_{|\{m\}\times  {\P^{1}_{\C}}}\ .$$
Indeed, both sides are canonically isomorphic to  the vector bundle on ${\P^{1}_{\C}}$ corresponding to
$$(U_{0}\otimes_{\sigma}\C)\otimes \cO_{ {\P^{1}_{\C}}}(1)\oplus (U_{2}\otimes_{\sigma}\C)\otimes \cO_{ {\P^{1}_{\C}}}\ .$$
We use parallel transport along curves with constant projection to $ {\P^{1}_{\C}}$ in order to extend this identification
to an isomorphism 
$$\phi_{\sigma}:W_{\sigma} \xrightarrow{ {\cong}} \widetilde W_{\sigma}$$ 
which is compatible with
the partial connections $\nabla^{I,W_{\sigma}}$ and $\nabla^{I,\widetilde W_{\sigma}}$.
This works, since by construction the isomorphism $\phi_{\sigma,m}$ maps the holonomy of $(\nabla^{I,W_{\sigma}})_{M\times \{z\}}$
to the holonomy of $(\nabla^{I,\widetilde W_{\sigma}})_{M\times \{z\}}$ for all $z\in {\P^{1}_{\C}}$.
Note that $\phi_{\sigma}$ is also compatible with the holomorphic structures
$\bar \partial^{W_{\sigma}}$, $\bar\partial^{\widetilde W_{\sigma}}$.
We now use $\phi_{\sigma}$ in order to transport the geometry $g^{W_{\sigma}} $ defined in \ref{jan0470}  for the sequence \eqref{jan1701} to a geometry
$g^{\widetilde W_{\sigma}}$.

We define the geometry $g^{\widetilde V_{1}}$ such that it evaluates at all $\sigma\in X(\C)$ to the geometry $g^{\widetilde W_{\sigma}}$.

We let $f_{0},f_{1}\colon M\times X\to M\times X\times Y$ be the maps given by  the points ${[0:1],[1:0]}\in \P^{1}_{\Z}(\Z)$.
By our normalization of the metric $h^{H_{Y}(1)}$, we have a natural isomorphism
\[
{f_{0}^{*}}(\widetilde V_{1},g^{\widetilde V_{1}})\cong (V_{1},g^{V_{1}})\ .
\]
We also consider the split variant $(\widetilde V_1^{\mathrm{split}}, g^{\widetilde V_1^{\mathrm{split}}})$ obtained by replacing the exact sequence \eqref{mar2601} in the constructions above by 
\begin{equation*}
\cV^{\mathrm{split}}: 0 \to V_0 \to V_0\oplus V_2 \to V_2 \to 0
\end{equation*}
and using the geometry induced by $g^{V_0}$ and $g^{V_2}$ on $V_0\oplus V_2$. By construction, we have  natural isomorphisms
\[
{f_1^*}(\widetilde V_1^{\mathrm{split}}, g^{\widetilde V_1^{\mathrm{split}}}) \cong {f_1^*}(\widetilde V_1, g^{\widetilde V_1}), \quad
{f_0^*}(\widetilde V_1^{\mathrm{split}}, g^{\widetilde V_1^{\mathrm{split}}}) \cong (V_{0},g^{V_{0}})\oplus (V_{2},g^{V_{2}})\ .
\]
We write
\[
\hat x:= \hcycl(\widetilde V_1,g^{\widetilde V_1}), \quad \hat y:= \hcycl(\widetilde V_1^{\mathrm{split}}, g^{\widetilde V_1^{\mathrm{split}}})\ .
\]
Since the topological cycle map \eqref{jan0660} is additive on short exact sequences, we have $I(\hat x) = I(\hat y)$. Furthermore, $\int_{II} R(\hat y) = 0$ (cf.~the argument in Subsection \ref{jan0470}). By Proposition \ref{jan0111} we have
\begin{equation*}
{ f_0^*(\hat x) - f_0^*(\hat y) = f_0^*(\hat x-\hat y) - f_1^*(\hat x-\hat y)} = a(\int_{II} R(\hat x)) = a(\int_{II} \omega(g^{\widetilde V_1}))\ .
\end{equation*}
In other words,
$$\hcycl(V_{0},g^{ {V_{0}}})-\hcycl(V_{1},g^{{V_{1}}})+\hcycl(V_{2},g^{ {V_{2}}})=-a(\int_{II} \omega(g^{\widetilde V_1}))\ .$$
Note that by \eqref{jan0480}
$$\ev_{\sigma}\int_{II} \omega(g^{\widetilde V_1})=\cU(\cV_{\sigma},g^{\cV_{\sigma}})$$
for all $\sigma\in X(\C)$. In view of \eqref{jan0571} this implies
$$
-a(\int_{II} \omega(g^{\widetilde V_1}))=a(\delta(\cV,g^{\cV}))\ , 
$$ 
and  thus  Theorem \ref{jan0430}.
\hB

\appendix
\section{Some technical results}
\label{apr0502}

\subsection{Sheaves  and \v{C}echification}
\label{dez2701}
\label{may2201} 
 
 {Here we assemble some technical facts about sheaves, used in the main text. The basic reference is \cite[\S 6.2.2]{HTT}. 
Note, however, that the situation in the present paper is simpler than that of {\it loc.~cit.}, since the sites we consider, i.e.~the domains of the sheaves, are always 1-categories. 
Here the $\infty$-categorical notion of a site reduces to the classical one \cite[Rem.~6.2.2.3]{HTT}.}

We consider a site $S$, 
i.e.~a category with a  {Grothendieck} topology determined by a collection of covering families.
Let $\mathbf{S}:=\mathtt{N}(S^{\mathit{op}})$.  
We define the \v{C}ech nerve of a covering family  $\mathcal{U}=(U_{\alpha}\to M)_{\alpha\in A}$ as the presheaf of simplicial sets
\[
\mathcal{U}_{\bullet} \in \mathbf{Fun}(\mathbf{S}, \mathtt{N}(\mathbf{sSet}))
\]
whose presheaf of $n$-simplices is given by 
\[
\mathcal{U}_{n} := \left(\coprod_{\alpha\in A} Y(U_{\alpha}) \right) \times_{Y(M)} \dots \times_{Y(M)} \left(\coprod_{\alpha\in A} Y(U_{\alpha}) \right)
\]
($n+1$ factors) with the obvious simplicial structure. Here $Y$ denotes the Yoneda embedding.
The \v{C}ech nerve has a natural augmentation towards $Y(M)$ considered as a presheaf of constant simplicial sets.
Using the natural map $\iota:\mathtt{N}(\mathbf{sSet}) \to \mathtt{N}(\mathbf{sSet})[W^{-1}]$ we form the presheaf of spaces $\iota(\mathcal{U}_{\bullet}) \in \mathbf{Fun}(\mathbf{S}, \mathtt{N}(\mathbf{sSet})[W^{-1}])$.

{Let} $\mathbf{C}$  {be} a presentable $\infty$-category  {\cite[Ch.~5]{HTT}}.
In the following definition we use the left Kan extension
\begin{equation}\label{eq:left-Kan}
\mathbf{Fun}(\mathbf{S},\mathbf{C}) \to \mathbf{Fun}(\mathbf{Fun}(\mathbf{S}, \mathtt{N}(\mathbf{sSet})[W^{-1}])^{\mathit{op}}, \mathbf{C}).
\end{equation}
We denote the image of a presheaf $F\in \mathbf{Fun}(\mathbf{S},\mathbf{C})$ under \eqref{eq:left-Kan} by the same symbol.

\begin{ddd}\label{jan0702}
We let
$$\mathbf{Fun}^{\mathit{desc}}(\mathbf{S},\mathbf{C})\subseteq \mathbf{Fun}(\mathbf{S},\mathbf{C})$$ 
denote the full subcategory
of sheaves, i.e.~objects $F\in  \mathbf{Fun}(\mathbf{S},\mathbf{C})$ which satisfy the descent condition that
$$
F(M) \simeq F(\iota(Y(M))) \xrightarrow{\simeq}  F(\iota({\mathcal{U}_{\bullet}}))
$$
is an equivalence in $\mathbf{C}$ for all ${M \in S}$ and  every covering family $\mathcal{U}$ of $M$.
\end{ddd}
\begin{rem}\label{rem:sheaves}
Assume that  the site $S$ satisfies the following conditions. 
\begin{enumerate}
\item Coproducts in $S$ are disjoint.
\item For every covering family $(U_{\alpha} \to M)_{\alpha \in A}$,
the coproduct $\coprod_{\alpha\in A} U_{\alpha}$ exists in $S$.
\item \label{Bedingung2}
If $(U_{\alpha})_{\alpha\in A}$ is a family of objects in $S$ whose coproduct $U := \coprod_{\alpha\in A} U_{\alpha}$ exists, then $(U_{\alpha} \to U)_{\alpha\in A}$ is a covering family.
\end{enumerate}
Then the descent condition for $F$ is equivalent to:
\begin{enumerate}
\item[(A)] For every family $(U_{\alpha})_{\alpha\in A}$ as in \ref{Bedingung2}.~above the natural map
\[
F(\coprod_{\alpha\in A} U_{\alpha}) \to \prod_{\alpha\in A} F(U_{\alpha})
\]
is an equivalence.
\item[(B)] For every covering family $(U_{\alpha}\to M)_{\alpha\in A}$, the natural map 
\[
F(M) \to \lim_{\mathtt{N}(\Delta)} F(U_{\bullet})
\]
is an equivalence, where $U_{\bullet}\in \mathbf{Fun}(\Delta^{\mathit{op}},S)$ is the \v{C}ech nerve of the morphism $U \to M$ with $U := \coprod_{\alpha\in A} U_{\alpha}$.
\end{enumerate}
See \cite[Section 6.3]{bg} for details.

This remark applies to all examples we consider in this paper.
\end{rem}

\color{black}

\begin{rem}
In the present paper, we define the notion of descent using \v{C}ech nerves. An alternative option would be to use hypercovers. The latter choice would be relevant if we would try to detect equivalences between sheaves on stalks what we never do in the present paper.
 For smooth manifolds  with the open  coverings topology  both options give the same notion of descent, anyway.
 \end{rem}

We have an adjunction
\begin{equation}\label{jan2704}
L:\Fun(\bS,\bC) 
 {\rightleftarrows} \Fun^{desc}(\bS,\bC): {\mathrm{inclusion}}\ ,
\end{equation}
where $L$ is called the sheafification.

We now consider the special case of the site $\Mf$ of manifolds with the open covering topology. As before, we write $\bS_{Mf}:= \Nerve(\Mf^{op})$.
The following lemma provides many examples of sheaves with values in chain complexes:
\begin{lem}\phantomsection\label{dqwhdlqdwqdqwdwqd}
\begin{enumerate}
\item For $b\in \Z$, let $\Ch^{\ge b}\subset \Ch$ be the  {full} subcategory of chain complexes which are  {bounded below} by $b$.
If $F\in \PSh_{\Ch^{\ge b}}(\Mf)$ is a presheaf of chain complexes consisting degree wise of fine sheaves, then
 {its image $\iota(F) \in \Fun(\bS_{Mf}, \Nerve(\Ch)[W^{-1}])$ satisfies descent:}
$$\iota(F)\in\Fun^{desc}(\bS_{Mf},\Nerve(\Ch)[W^{-1}])\ .$$
\item
For  {a chain complex of real vector spaces} $C\in \Ch$,  {we denote by $\Omega C \in \Sh_{\Ch}(\Mf)$ the sheaf of differential forms with coefficients in $C$. We} have
$$\iota(\Omega C)\in \Fun^{desc}(\bS_{Mf},\Nerve(\Ch)[W^{-1}])\ .$$
\end{enumerate}
\end{lem}
\begin{proof}
Clearly, condition (A) of Remark~\ref{rem:sheaves} is fulfilled. It remains to check (B).
The limit $\lim_{\Nerve(\Delta)} F(U_{\bullet})$ is realized by the total complex  
$\Tot  {F(U_{\bullet})}$ of a double complex which is acyclic in the \v{C}ech direction.
If $F$ is  {bounded below}, then the result  follows easily using the  spectral sequence induced by the filtration induced by the degree of $F$.\footnote{Note that the argument given in \cite[Lemma 6.15]{bg} or  \cite[Lemma 7.12]{Bunke:2013aa}  
only works under this boundedness assumption. This assumption was missed in the statement given in these references.} 

In the second case, since we do not assume $C$ {to be bounded below}, we  argue differently.
Using a partition of unity we construct a contracting homotopy $h$ for the \v{C}ech complex.
The homotopy does not commute with the differential of $\Omega C$, but the commutator $[h,d]$ increases the differential form degree. 
In the second case we therefore use the spectral sequence associated to the filtration of 
$\Tot  {\Omega C(U_{\bullet})}$ induced by the filtration of $\Omega C$ given by $\cF^{p}\Omega C:=\Omega^{p}\cdot  \Omega C$.
 {The homotopy $h$ above induces a contracting homotopy on the $E_{1}$-page of this spectral sequence. This implies the claim.}
 \end{proof}

 Let $f:  {F}\to  {G}$ be a morphism in $\Fun(\bS_{Mf},\bC)$. Sometimes one needs a criterion ensuring  that $L(f):L( {F})\to L( {G})$ is an equivalence  {in $\Fun^{desc}(\bS_{Mf},\bC)$}. In examples this turns out too difficult to check directly. In some cases one can use the following approximation  
 $\cL:\Fun(\bS_{Mf},\bC)\to \Fun(\bS_{Mf},\bC)$ of the sheafification functor \eqref{jan2704} which on objects acts as 
$$
\cL( {F})(M):=\colim_{\cU}  F(\iota(\cU_{\bullet}))\ ,
$$
where the  colimit is  {taken} over  a suitable filtered system of open coverings of $M$. Its precise construction goes as follows.

We consider a partially ordered set as a category in the natural way. 
We have a functor  $\Mf^{op}\to\Cat$
which associates to every manifold $M$ the filtered partially ordered 
set of open coverings {$(U_m)_{m\in M}$} of $M$ indexed by the points of $M$ {in such a way that $m\in U_m$}. By 
  $\tilde \Mf$ we denote the Grothendieck construction of this functor. An object of $\tilde \Mf$ is a pair $(M,\cU)$ of a manifold  {$M$} and an open covering $\cU:=(U_{m})_{m\in M}$  {of $M$. A} morphism     $(M,\cU)\to (M^{\prime},\cU^{\prime})$ is a smooth map  $f:M\to M^{\prime}$ such that $ U_{m}\subseteq f^{-1}(U^{\prime}_{f(m)}) $ for all $m\in M$.  We have a functor
$\tilde \Mf\to \Fun(\bS_{Mf}, \Nerve(\sSet))$ which associates to  $(M,\cU)$ the \v{C}ech nerve $\cU_{\bullet}$. 
 Precomposition with this functor gives the functor 
$$
\Fun(\Fun(\bS_{Mf}, \Nerve(\sSet)[W^{-1}])^{op},\bC) \to \Fun(\Nerve(\tilde\Mf)^{op}, \bC)\ .
$$
We further precompose with \eqref{eq:left-Kan} and get the functor
$$
\tilde \cL:\Fun(\bS_{Mf},\bC)\to  \Fun(\Nerve(\tilde\Mf)^{op},\bC)\ .
$$

We have an adjunction
$$\Phi_{*}:\Fun(\Nerve(\tilde \Mf)^{op},\bC)\leftrightarrows \Fun( \bS_{Mf},\bC):\Phi^{*}\ ,$$
where $\Phi:\Nerve(\tilde \Mf^{op})\to \bS_{Mf}$ forgets the coverings and $\Phi^{*}$ is the pull-back along this functor. Its left adjoint $\Phi_{*}$ is the Kan-extension functor.
\begin{ddd}\label{may1510n}
We define $\cL:=\Phi_{*} \circ \tilde \cL$. 
\end{ddd}

By the point-wise formula for the Kan extension functor we have the equivalence
$$
\cL( {F})(M) \simeq \colim_{(( N,\cU),N\to M)\in \tilde \Mf^{op}_{/M}}  F(\iota(\cU_{\bullet})) \ .
$$
Now
$\tilde \Mf^{op}_{/M}=(\tilde \Mf_{M/})^{op}$
contains the category of open coverings of $M$ as a cofinal subcategory. 
Hence it suffices to take the colimit over the open coverings of $M$ and we get 
$$
\cL( {F})(M) \simeq \colim_{(M,\cU)\in \tilde \Mf^{op}_{/M}} F(\iota(\cU_{\bullet}))\ .
$$
 
 If $ {F}\in \Fun(\bS_{Mf},\Nerve(\Ch))$ is a presheaf of chain complexes, then (abusing notation) we agree to 
  define
 $\cL  {F}\in \Fun(\bS_{Mf},\Nerve(\Ch))$ by
$$
\cL  {F}(M):=\colim_{(M,\cU)\in \tilde \Mf^{op}_{/M}} \Tot \check{C}(\cU,F) 
$$
where $\check{C}(\cU,F)$ is the usual \v{C}ech double complex of $F$ associated to $\cU$.
Note that $\cL  {F}$ is a model for $\cL ( \iota {(F)} )$, where 
$\iota {(F)} \in  \Fun(\bS_{Mf},\Nerve(\Ch)[W^{-1}])$ is the image of $F$ under localization.
 
  We have a natural transformation $ \id\to \cL$ given by the canonical maps
\[
{F}(M)\to  \colim_{(M,\cU)\in \tilde \Mf^{op}_{/M}} {F}(\iota(\cU_{\bullet})).\]
 
 \begin{fact}\label{dkjhwqjkdwqdwqdwqdwqdwqd}
If $ {F}$ is a sheaf, then ${F}\to \cL( {F})$ is an equivalence.
\end{fact}
 
 In order to relate $\cL$ with the sheafification, we define the functor
$$
\cL^{\infty}:=\colim (\id\to \cL\to \cL^{2}\to  \cL^{3}\to \dots):\Fun(\bS_{Mf},\bC)\to \Fun(\bS_{Mf},\bC)\ .
$$
By construction, the natural morphism  $\cL^{\infty}\to (\cL^{\infty})^{2}$ is an equivalence.
We let $$\Fun^{\cL^{\infty}}(\bS_{Mf},\bC)\subseteq \Fun(\bS_{Mf},\bC)$$  be the essential image of $\cL^{\infty}$.
This is a localization, since condition (3) of
the recognition principle
\cite[Prop.~5.2.7.4]{HTT} is satisfied.  If $ {F}$ is  a sheaf, then the morphism $F \to  \cL^{\infty}(F)$ is an equivalence. Hence we have a sequence of localizations
$$
\Fun^{desc}(\bS_{Mf},\bC)\subseteq \Fun^{\cL^{\infty}}(\bS_{Mf},\bC)\subseteq  \Fun (\bS_{Mf},\bC)\ .
$$ 
In particular, the natural transformation
\begin{equation}\label{jun3001}
L\to L\circ \cL^{\infty}
\end{equation} 
is an equivalence. We conclude:
\begin{fact}\label{may1611} 
If $f: {F}\to  {G}$ is a morphism in $\Fun (\bS_{Mf},\bC)$ such that $\cL(f):\cL( {F})\to \cL( {G})$ is an equivalence, then $L(f):L( {F})\to L( {G})$ is an equivalence.
\end{fact}

\subsection{Homotopy invariance}\label{mar0804}

We introduce the notion of homotopy invariance for (pre)sheaves on $\Mf$ with values in an $\infty$-category, provide a  relation to constant sheaves, and discuss two important examples.

Let $I:=[0,1]\in \Mf$ be the unit interval, and  {let $\bS$ be one of the $\infty$-categories $\bS_{Mf} = \Nerve(\Mf^{op})$, $\bS_{Mf,\C} = \Nerve(\Mf^{op}\times \Sm_{\C}^{op})$, or $\bS_{Mf,\Z} = \Nerve(\Mf^{op}\times \Reg_{\Z}^{op})$}.
For any   $\infty$-category $\bC$, the 
pull-back along taking the product with $I$ gives an endofunctor
$\cI:\Fun(\bS,\bC)\to \Fun(\bS,\bC)$. Moreover, the projection along $I$ gives a transformation
$\id\to \cI$.

\begin{ddd}\label{jul0740}
We call a  {presheaf}
$ {F}\in  \Fun(\bS,\bC)$ homotopy invariant, if the natural morphism
$ {F}\to \cI( {F})$ is an equivalence.
\end{ddd}

We let
$$\Fun^{h}(\bS,\cC)\subseteq  \Fun(\bS,\bC)$$
denote the full subcategory of homotopy invariant presheaves. If $\bC$ is presentable, then we have an adjunction
$$\cH^{pre}: \Fun(\bS,\bC)\leftrightarrows \Fun^{h}(\bS,\cC):incl$$
where the left-adjoint is called homotopification. In the following we construct an explicit model for $\cH^{pre}$.

Let $\Delta^{\bullet}\in \Mf^{\Delta}$ be the cosimplicial manifold given by the standard simplices. It gives a functor  {$\bS \times \Nerve(\Delta^{op}) \to \bS$ induced by $(M,[q]) \mapsto M \times \Delta^q$}. Pull-back along this functor defines the functor
\begin{equation}\label{jul0802}
\bs:\Fun(\bS,\bC)\to \Fun(\bS\times \Nerve(\Delta^{op}),\bC)\ .
\end{equation}
We define the endofunctor
\begin{equation}\label{jul0741}\bar \bs:=\colim_{\Nerve(\Delta^{op})}\circ \bs:\Fun(\bS,\bC)\to \Fun(\bS,\bC)\ .\end{equation}
 {The} projection $\Delta^{\bullet}\to *$ to the constant cosimplicial manifold given by the point
 induces the transformation 
\begin{equation}\label{jul0811}\id\to \bar \bs\ .\end{equation}

 
\begin{lem}\phantomsection\label{jul0730}%
\begin{enumerate}
\item We have an equivalence $\cH^{pre}\simeq \bar \bs$.
\item
If $ {F}\in \Fun(\bS,\bC)$ is homotopy invariant, then the natural morphism $ {F}\to \bar \bs( {F})$ is an equivalence.
\end{enumerate} 
\end{lem}
\begin{proof}
The first assertion follows from 
  \cite[Lemma 7.5]{Bunke:2013aa}. Then second is shown by 
 $F\simeq \cH^{pre}(F)\simeq \bar \bs(F)$.
\end{proof}

 Let $\bC$ be a presentable $\infty$-category. 
 We have a functor $p:\bS_{Mf}\to *$ which induces a pull-back
 $$p^{*}:\bC\simeq \Fun(*,\bC)\to \Fun(\bS_{ {Mf}},\bC)\ .$$
 We define the constant sheaf functor by
$$
\underline{...}:=L\circ p^{*}:\bC\to \Fun^{desc}(\bS_{ {Mf}},\bC)\ .
$$
Note that for $X\in \bC$ both,
$p^{*}X$ and $\underline{X}$, are  homotopy invariant.
 { 
\begin{ex}
We consider a spectrum  $E\in \Sp$. It represents a cohomology theory $E^{*}$. This is related to the constant sheaf of spectra $\underline{E}$ on $\Mf$ via natural isomorphisms
\begin{equation}\label{nov2604}
\pi_{i}(\underline{E} (M))\cong E^{-i}(M)
\end{equation}
for every integer $i$ and every manifold $M$.
\end{ex}
}

We consider the embedding $e:* \to \bS_{Mf}$ of the $\infty$-category $*$  as the subcategory of the one-point manifold, and  {we} let
$$e^{ {*}}:\Fun(\bS_{Mf},\bC)\to \Fun(*,\bC)\simeq \bC$$ 
 {be} the corresponding pull-back.  {For} every $ {F}\in \Fun^{desc}(\bS_{Mf},\bC)$ we have a natural morphism $\underline{e^{*} {F}}\to  {F}$. 
\begin{lem}\label{dkjhqwkdqwdwqdqwdwqd}
A sheaf $ {F}\in \Fun^{desc}(\bS_{Mf}, {\bC})$ is  homotopy invariant, if and only if
the natural  morphism $\underline{e^{*}{F}}\to  {F}$ is an equivalence.
 \end{lem}
\begin{proof}
This follows immediately from the fact that every manifold admits a cofinal system of good coverings,  i.e.~open coverings such that all components of  multiple intersections are contractible. 
\end{proof}
%
%
%
%
\begin{ex}
Let us discuss an application of Lemma \ref{dkjhqwkdqwdwqdqwdwqd}. We consider a chain complex $C$ of real vector spaces, and we let $\Omega C \in \Sh_{\Ch}(\Mf)$ be the sheaf of forms with coefficients in $C$. 
 \begin{lem}\label{efef234234fwefewfewfewfwfew}
We have an equivalence
 $$
H(\Omega C)\simeq \underline{H(C)}$$
in $\Fun^{desc,h}(\bS_{Mf},\Sp)$.
\end{lem}
\begin{proof}
By the Poincar\'e lemma, $H(\Omega C)\in \Fun(\bS_{Mf},\Sp)$ is a homotopy invariant sheaf. Hence we have a canonical equivalence
\[
\underline{H(C)}\simeq \underline{e^{*}H(\Omega C)}\overset{\ref{dkjhqwkdqwdwqdqwdwqd}}{\simeq} H(\Omega C)\ .\qedhere
\]
\end{proof}
\end{ex}

 { 
We extend the above considerations to the product sites $\Mf\times \Sm_{\C}$ and $\Mf\times \Reg_{\Z}$ using the equivalences
\begin{align*}
&\Fun^{desc,h}(\bS_{Mf,\C},\bC) \simeq \Fun^{desc,h}(\bS_{Mf},\Fun^{desc}(\bS_{\C},\bC)), \\
&\Fun^{desc,h}(\bS_{Mf,\Z},\bC) \simeq \Fun^{desc,h}(\bS_{Mf},\Fun^{desc}(\bS_{\Z},\bC)).
\end{align*}
In particular, we get a relative version of the constant sheaf functor
\begin{equation}\label{eq:relconstsheaf}
\underline{...}\colon \Fun^{desc}(\bS_{\C}, \bC) \to \Fun^{desc}(\bS_{Mf,\C},\bC)
\end{equation}
and an equivalence
\begin{equation}\label{eq:relA9}
\underline{e^{*}F} \xrightarrow{\simeq} F
\end{equation}
for every $F$ in $\Fun^{desc,h}(\bS_{Mf,\C}, \bC)$.
}

\bibliographystyle{amsalpha}
\bibliography{new}

\providecommand{\bysame}{\leavevmode\hbox to3em{\hrulefill}\thinspace}
\providecommand{\MR}{\relax\ifhmode\unskip\space\fi MR }
\providecommand{\MRhref}[2]{%
  \href{http://www.ams.org/mathscinet-getitem?mr=#1}{#2}
}
\providecommand{\href}[2]{#2}
\begin{thebibliography}{BSSW09}

\bibitem[Be{\u\i}84]{MR760999}
A.~A. Be{\u\i}linson, \emph{Higher regulators and values of {$L$}-functions},
  Current problems in mathematics, {V}ol. 24, Itogi Nauki i Tekhniki, Akad.
  Nauk SSSR Vsesoyuz. Inst. Nauchn. i Tekhn. Inform., Moscow, 1984,
  pp.~181--238. \MR{760999 (86h:11103)}

\bibitem[Be{\u\i}86]{MR862628}
\bysame, \emph{Notes on absolute {H}odge cohomology}, Applications of algebraic
  {$K$}-theory to algebraic geometry and number theory, {P}art {I}, {II}
  ({B}oulder, {C}olo., 1983), Contemp. Math., vol.~55, Amer. Math. Soc.,
  Providence, RI, 1986, pp.~35--68. \MR{862628 (87m:14019)}

\bibitem[BG75]{MR0377873}
J.~C. Becker and D.~H. Gottlieb, \emph{The transfer map and fiber bundles},
  Topology \textbf{14} (1975), 1--12. \MR{0377873 (51 \#14042)}

\bibitem[BG02]{MR1869655}
Jos{\'e}~I. Burgos~Gil, \emph{The regulators of {B}eilinson and {B}orel}, CRM
  Monograph Series, vol.~15, American Mathematical Society, Providence, RI,
  2002. \MR{1869655 (2002m:19002)}

\bibitem[BG13]{bg}
Ulrich Bunke and David Gepner, \emph{Differential function spectra, the
  differential {B}ecker-{G}ottlieb transfer, and applications to differential
  algebraic {$K$}-theory},
  \href{http://arxiv.org/abs/1306.0247}{arXiv:1306.0247}, 2013.

\bibitem[BL95]{MR1303026}
Jean-Michel Bismut and John Lott, \emph{Flat vector bundles, direct images and
  higher real analytic torsion}, J. Amer. Math. Soc. \textbf{8} (1995), no.~2,
  291--363. \MR{1303026 (96g:58202)}

\bibitem[BNV13]{Bunke:2013aa}
Ulrich Bunke, Thomas Nikolaus, and Michael V{\"o}lkl, \emph{Differential
  cohomology theories as sheaves of spectra},
  \href{http://arxiv.org/abs/1311.3188}{arXiv:1311.3188}, 2013.

\bibitem[Bor74]{MR0387496}
Armand Borel, \emph{Stable real cohomology of arithmetic groups}, Ann. Sci.
  \'Ecole Norm. Sup. (4) \textbf{7} (1974), 235--272 (1975). \MR{0387496 (52
  \#8338)}

\bibitem[BS09]{MR2664467}
Ulrich Bunke and Thomas Schick, \emph{Smooth {$K$}-theory}, Ast{\'e}risque
  (2009), no.~328, 45--135 (2010). \MR{2664467 (2012a:19015)}

\bibitem[BS10]{MR2608479}
\bysame, \emph{Uniqueness of smooth extensions of generalized cohomology
  theories}, J. Topol. \textbf{3} (2010), no.~1, 110--156. \MR{2608479
  (2011e:55011)}

\bibitem[BSSW09]{MR2550094}
Ulrich Bunke, Thomas Schick, Ingo Schr{{\"o}}der, and Moritz Wiethaup,
  \emph{Landweber exact formal group laws and smooth cohomology theories},
  Algebr. Geom. Topol. \textbf{9} (2009), no.~3, 1751--1790. \MR{2550094
  (2011d:55005)}

\bibitem[Bun09]{Bunke-index}
Ulrich Bunke, \emph{Index theory, eta forms, and {D}eligne cohomology}, Mem.
  Amer. Math. Soc. \textbf{198} (2009), no.~928, vi+120. \MR{2191484
  (2010d:58023)}

\bibitem[Bun12]{skript}
\bysame, \emph{Differential cohomology}, Course notes, Regensburg,
  \href{http://arxiv.org/abs/1208.3961}{arXiv:1208.3961}, 2012.

\bibitem[Bur94]{BurgosLogarithmic}
Jos{\'e}~Ignacio Burgos, \emph{A {$C^\infty$} logarithmic {D}olbeault complex},
  Compositio Math. \textbf{92} (1994), no.~1, 61--86. \MR{1275721 (95g:32056)}

\bibitem[BV73]{Boardman-Vogt}
J.~M. Boardman and R.~M. Vogt, \emph{Homotopy invariant algebraic structures on
  topological spaces}, Lecture Notes in Mathematics, Vol. 347, Springer-Verlag,
  Berlin-New York, 1973. \MR{0420609 (54 \#8623a)}

\bibitem[BW98]{MR1621424}
Jose~Ignacio Burgos and Steve Wang, \emph{Higher {B}ott-{C}hern forms and
  {B}eilinson's regulator}, Invent. Math. \textbf{132} (1998), no.~2, 261--305.
  \MR{1621424 (99j:14008)}

\bibitem[Che79]{MR528965}
Jeff Cheeger, \emph{Analytic torsion and the heat equation}, Ann. of Math. (2)
  \textbf{109} (1979), no.~2, 259--322. \MR{528965 (80j:58065a)}

\bibitem[CS85]{MR827262}
Jeff Cheeger and James Simons, \emph{Differential characters and geometric
  invariants}, Geometry and topology ({C}ollege {P}ark, {M}d., 1983/84),
  Lecture Notes in Math., vol. 1167, Springer, Berlin, 1985, pp.~50--80.
  \MR{827262 (87g:53059)}

\bibitem[Del71]{HodgeII}
Pierre Deligne, \emph{Th\'eorie de {H}odge. {II}}, Inst. Hautes \'Etudes Sci.
  Publ. Math. (1971), no.~40, 5--57. \MR{0498551 (58 \#16653a)}

\bibitem[Del74]{HodgeIII}
\bysame, \emph{Th\'eorie de {H}odge. {III}}, Inst. Hautes \'Etudes Sci. Publ.
  Math. (1974), no.~44, 5--77. \MR{0498552 (58 \#16653b)}

\bibitem[Del87]{MR902592}
P.~Deligne, \emph{Le d\'eterminant de la cohomologie}, Current trends in
  arithmetical algebraic geometry ({A}rcata, {C}alif., 1985), Contemp. Math.,
  vol.~67, Amer. Math. Soc., Providence, RI, 1987, pp.~93--177. \MR{902592
  (89b:32038)}

\bibitem[DWW03]{MR1982793}
W.~Dwyer, M.~Weiss, and B.~Williams, \emph{A parametrized index theorem for the
  algebraic {$K$}-theory {E}uler class}, Acta Math. \textbf{190} (2003), no.~1,
  1--104. \MR{1982793 (2004d:19004)}

\bibitem[FH00]{MR1769477}
Daniel~S. Freed and Michael Hopkins, \emph{On {R}amond-{R}amond fields and
  {$K$}-theory}, J. High Energy Phys. (2000), no.~5, Paper 44, 14. \MR{1769477
  (2001k:81221)}

\bibitem[FL10]{Freed-Lott}
Daniel~S. Freed and John Lott, \emph{An index theorem in differential
  {$K$}-theory}, Geom. Topol. \textbf{14} (2010), no.~2, 903--966. \MR{2602854
  (2011h:58036)}

\bibitem[Gro10]{Groth}
Moritz Groth, \emph{A short course on {$\infty$}-categories}, preprint
  available from the author's home page \url{http://www.math.ru.nl/~mgroth/},
  July 2010.

\bibitem[GS90a]{MR1038362}
Henri Gillet and Christophe Soul{\'e}, \emph{Characteristic classes for
  algebraic vector bundles with {H}ermitian metric. {I}}, Ann. of Math. (2)
  \textbf{131} (1990), no.~1, 163--203. \MR{1038362 (91m:14032a)}

\bibitem[GS90b]{GS2}
\bysame, \emph{Characteristic classes for algebraic vector bundles with
  {H}ermitian metric. {II}}, Ann. of Math. (2) \textbf{131} (1990), no.~2,
  205--238. \MR{1043268 (91m:14032b)}

\bibitem[Hak72]{Hakim}
Monique Hakim, \emph{Topos annel\'es et sch\'emas relatifs}, Springer-Verlag,
  Berlin, 1972, Ergebnisse der Mathematik und ihrer Grenzgebiete, Band 64.
  \MR{0364245 (51 \#500)}

\bibitem[Hir64]{Hironaka}
Heisuke Hironaka, \emph{Resolution of singularities of an algebraic variety
  over a field of characteristic zero. {I}, {II}}, Ann. of Math. (2) 79 (1964),
  109--203; ibid. (2) \textbf{79} (1964), 205--326. \MR{0199184 (33 \#7333)}

\bibitem[HQ15]{HopkinsQuick}
Michael~J. Hopkins and Gereon Quick, \emph{Hodge filtered complex bordism},
  Journal of Topology \textbf{8} (2015), no.~1, 147--183.

\bibitem[HS05]{MR2192936}
M.J. Hopkins and I.M. Singer, \emph{{Quadratic functions in geometry, topology,
  and {M}-theory}}, J. Differential Geom. \textbf{70} (2005), no.~3, 329--452.
  \MR{MR2192936 (2007b:53052)}

\bibitem[Joy02]{Joyal}
A.~Joyal, \emph{Quasi-categories and {K}an complexes}, J. Pure Appl. Algebra
  \textbf{175} (2002), no.~1-3, 207--222, Special volume celebrating the 70th
  birthday of Professor Max Kelly. \MR{1935979 (2003h:55026)}

\bibitem[Kar83]{KaroubiCR}
Max Karoubi, \emph{Homologie cyclique et r\'egulateurs en {$K$}-th\'eorie
  alg\'ebrique}, C. R. Acad. Sci. Paris S\'er. I Math. \textbf{297} (1983),
  no.~10, 557--560. \MR{735692 (85f:18006)}

\bibitem[Kar87]{KaroubiAst}
\bysame, \emph{Homologie cyclique et {$K$}-th\'eorie}, Ast\'erisque (1987),
  no.~149, 147. \MR{913964 (89c:18019)}

\bibitem[Kar90]{KaroubiTheoGen}
\bysame, \emph{Th\'eorie g\'en\'erale des classes caract\'eristiques
  secondaires}, $K$-Theory \textbf{4} (1990), no.~1, 55--87. \MR{1076525
  (92a:55006)}

\bibitem[Lot00]{MR1724894}
John Lott, \emph{Secondary analytic indices}, Regulators in analysis, geometry
  and number theory, Progr. Math., vol. 171, Birkh\"auser Boston, Boston, MA,
  2000, pp.~231--293. \MR{1724894 (2001c:58024)}

\bibitem[Lur09]{HTT}
Jacob Lurie, \emph{Higher topos theory}, Annals of Mathematics Studies, vol.
  170, Princeton University Press, Princeton, NJ, 2009. \MR{2522659
  (2010j:18001)}

\bibitem[Lur14]{highalg}
\bysame, \emph{Higher algebra}, Unpublished manuscript available from the
  author's homepage \url{http://www.math.harvard.edu/~lurie/}, dated September
  14, 2014, 2014.

\bibitem[M{\"u}l78]{MR498252}
Werner M{\"u}ller, \emph{Analytic torsion and {$R$}-torsion of {R}iemannian
  manifolds}, Adv. in Math. \textbf{28} (1978), no.~3, 233--305. \MR{498252
  (80j:58065b)}

\bibitem[Nag62]{Nagata}
Masayoshi Nagata, \emph{Imbedding of an abstract variety in a complete
  variety}, J. Math. Kyoto Univ. \textbf{2} (1962), 1--10. \MR{0142549 (26
  \#118)}

\bibitem[Qui73]{Quillen}
Daniel Quillen, \emph{Higher algebraic {$K$}-theory. {I}}, Algebraic
  {$K$}-theory, {I}: {H}igher {$K$}-theories ({P}roc. {C}onf., {B}attelle
  {M}emorial {I}nst., {S}eattle, {W}ash., 1972), Springer, Berlin, 1973,
  pp.~85--147. Lecture Notes in Math., Vol. 341. \MR{0338129 (49 \#2895)}

\bibitem[Rap88]{MR944994}
M.~Rapoport, \emph{Comparison of the regulators of {B}e\u\i linson and of
  {B}orel}, Be\u\i linson's conjectures on special values of {$L$}-functions,
  Perspect. Math., vol.~4, Academic Press, Boston, MA, 1988, pp.~169--192.
  \MR{944994 (89j:11114)}

\bibitem[Sch87]{Schechtman}
V.~V. Schechtman, \emph{On the delooping of {C}hern character and {A}dams
  operations}, {$K$}-theory, arithmetic and geometry ({M}oscow, 1984--1986),
  Lecture Notes in Math., vol. 1289, Springer, Berlin, 1987, pp.~265--319.
  \MR{923139 (89c:18015)}

\bibitem[Sch12]{2012arXiv1205.3890S}
Jakob Scholbach, \emph{{Arakelov motivic cohomology II}},
  \href{http://arxiv.org/abs/1205.3890}{arXiv:1205.3890}, 2012.

\bibitem[Sou92]{MR1208731}
C.~Soul{\'e}, \emph{Lectures on {A}rakelov geometry}, Cambridge Studies in
  Advanced Mathematics, vol.~33, Cambridge University Press, Cambridge, 1992,
  With the collaboration of D. Abramovich, J.-F. Burnol and J. Kramer.
  \MR{1208731 (94e:14031)}

\bibitem[SS08a]{MR2365651}
James Simons and Dennis Sullivan, \emph{Axiomatic characterization of ordinary
  differential cohomology}, J. Topol. \textbf{1} (2008), no.~1, 45--56.
  \MR{2365651 (2009e:58035)}

\bibitem[SS08b]{MR2521641}
\bysame, \emph{Structured bundles define differential {$K$}-theory},
  Ast{\'e}risque (2008), no.~321, 1--3, G{{\'e}}om{{\'e}}trie
  diff{{\'e}}rentielle, physique math{{\'e}}matique, math{{\'e}}matiques et
  soci{{\'e}}t{{\'e}}. I. \MR{2521641 (2011f:19006)}

\bibitem[Tak05]{MR2153537}
Yuichiro Takeda, \emph{Higher arithmetic {$K$}-theory}, Publ. Res. Inst. Math.
  Sci. \textbf{41} (2005), no.~3, 599--681. \MR{2153537 (2006i:14022)}

\bibitem[Tam12]{Tamme-Beil}
Georg Tamme, \emph{Karoubi's relative {C}hern character and {B}eilinson's
  regulator}, Ann. Sci. {\'E}c. Norm. Sup{\'e}r. (4) \textbf{45} (2012), no.~4,
  601--636.

\bibitem[TT90]{TT}
R.~W. Thomason and Thomas Trobaugh, \emph{Higher algebraic {$K$}-theory of
  schemes and of derived categories}, The {G}rothendieck {F}estschrift, {V}ol.\
  {III}, Progr. Math., vol.~88, Birkh\"auser Boston, Boston, MA, 1990,
  pp.~247--435. \MR{1106918 (92f:19001)}

\bibitem[Wei13]{WeibelKBook}
Charles~A. Weibel, \emph{The {$K$}-book}, Graduate Studies in Mathematics, vol.
  145, American Mathematical Society, Providence, RI, 2013, An introduction to
  algebraic $K$-theory. \MR{3076731}

\end{thebibliography}
\end{document}